\documentclass[a4paper,10pt,fleqn]{amsart}
\usepackage[utf8]{inputenc}
\usepackage{epsfig,graphics}
\usepackage[all,knot,poly,color,curve]{xy}
\usepackage{array}
\usepackage{amsmath,amsthm,amssymb,amsbsy,bbm,a4wide,verbatim,url}
\usepackage[marginpar]{todo}
\usepackage{longtable}
\usepackage{rotating}

\usepackage{fancybox}
\usepackage{enumitem}
\usepackage{ulem}
\usepackage[colorlinks=true,linkcolor=blue,draft=false]{hyperref}
\usepackage{aliascnt}

\usepackage{tikz,tkz-euclide}
\usepackage{tikz-cd}
\usepackage{tikz-3dplot}
\usetikzlibrary{arrows, decorations.text, decorations.pathmorphing, decorations.markings, positioning, shapes}
\usetikzlibrary{}
\usetikzlibrary{calc}
\usetikzlibrary{decorations.markings}

\def\Id{\mathrm{Id}}

\def\into{\hookrightarrow}
\def\onto{\twoheadrightarrow}

\def\om{\omega}

\def\C{\ensuremath{\mathbbm{C}}}

\def\Z{\mathbbm{Z}}
\def\R{\mathbbm{R}}
\def\N{\mathbbm{N}}

\def\sl{\mathfrak{sl}}

\def\eps{\varepsilon}
\def\ii{\mathrm{i}}

\DeclareMathOperator{\Ker}{\mathrm{Ker}}
\DeclareMathOperator{\Imm}{\mathrm{Im}}
\DeclareMathOperator{\codim}{\mathrm{codim}}

\def\dlog{\mathrm{dlog}}

\def\GL{\mathrm{GL}}

\def\RC{\mathsf{R}}
\def\lcm{\mathrm{lcm}}

\DeclareMathOperator{\Div}{Div}
\DeclareMathOperator{\PC}{PC}

\theoremstyle{plain}
\newtheorem{theorem}{Theorem}[section]

\newaliascnt{lemma}{theorem}
\newtheorem{lemma}[lemma]{Lemma}
\aliascntresetthe{lemma}

\newaliascnt{proposition}{theorem}
\newtheorem{proposition}[proposition]{Proposition}
\aliascntresetthe{proposition}

\newaliascnt{corollary}{theorem}
\newtheorem{corollary}[corollary]{Corollary}
\aliascntresetthe{corollary}

\newaliascnt{conjecture}{theorem}

\aliascntresetthe{conjecture}

\newaliascnt{lemme}{theorem}

\aliascntresetthe{lemme}

\theoremstyle{remark}

\newaliascnt{claim}{theorem}

\aliascntresetthe{claim}

\newaliascnt{remark}{theorem}
\newtheorem{remark}[remark]{Remark}
\aliascntresetthe{remark}

\newaliascnt{remarq}{theorem}

\aliascntresetthe{remarq}

\newaliascnt{question}{theorem}

\aliascntresetthe{question}

\theoremstyle{definition}

\newaliascnt{definition}{theorem}
\newtheorem{definition}[definition]{Definition}
\aliascntresetthe{definition}

\newaliascnt{example}{theorem}
\newtheorem{example}[example]{\it Example\/}
\aliascntresetthe{example}

\newaliascnt{notation}{theorem}

\aliascntresetthe{notation}

\numberwithin{equation}{section}

\date{October 22, 2025}

\author{Juan Gonz\'alez-Meneses}
\address{Departemento de \'algebra, Universidad de Sevilla, Spain}
\email{meneses@us.es}

\author{Ivan Marin}
\address{LAMFA, UMR CNRS 7352, Universit\'e de Picardie Jules Verne, Amiens, France}
\email{ivan.marin@u-picardie.fr}

\title[]{Parabolic subgroups of complex braid groups}

\begin{document}

\maketitle

{\bf Abstract.}
In this paper we introduce a class of `parabolic' subgroups for the generalized braid group
associated to an arbitrary irreducible complex reflection group, which maps onto the collection of parabolic subgroups of the reflection group. Except for one case, which is proven separately elsewhere, we prove that this collection forms a lattice, so that intersections of parabolic subgroups are parabolic subgroups. In particular, every element admits a parabolic closure, which is the smallest parabolic subgroup containing it. We furthermore prove that it provides a simplicial complex which generalizes the curve complex of the usual braid group. In the case of real reflection groups, this complex generalizes the one previously introduced by Cumplido, Gebhardt, Gonz\'alez-Meneses and Wiest for Artin groups of spherical type. We show that it shares similar properties, and similarly conjecture its hyperbolicity, with a few additional results in this direction.\footnote{First author partially supported by PID2020-117971GB-C21 funded by MCIN/AEI/10.13039/501100011033, US-1263032 (US/JUNTA/FEDER, UE), and P20\_01109 (JUNTA/FEDER, UE). This paper is partially based upon work supported by the National Science Foundation under Grant No. DMS-1929284 while the first author was in residence at the Institute for Computational and Experimental Research in Mathematics in Providence, RI, during the Braids semester program.}

\tableofcontents

\section{Introduction}

Let $W$ be a complex reflection group, and $B$ the generalized braid group associated to
it, as in \cite{BMR}. In the case $W$ is a real reflection group, $B$ can be described as an Artin group, and it was shown in~\cite{CGGW} how to generalize the curve complex of the usual braid group in this setting: curves are replaced with irreducible parabolic subgroups, which are defined algebraically as conjugates of the irreducible standard parabolic subgroups of the Artin group. Actually, these standard parabolic subgroups are the standard parabolic subgroups of the corresponding Garside structure, as defined in \cite{GODELLE2007}.
A disadvantage of this construction is however that it heavily depends on the presentation chosen for $B$, or equivalently on its construction as the group of fractions of some specific monoid.

Here we provide a generalization of this work to the general setting of a complex reflection group. For this we first provide, in Section 2, a purely topological definition of a parabolic subgroup of $B$, which is a
subgroup isomorphic to the braid group of some parabolic subgroup
of $W$. This
topological definition as some kind of \emph{local fundamental group} implies `on the nose' two main properties. First of all it implies that, when two reflection groups are 'isodiscriminantal' (see the
end of \autoref{sect:parabdefinition}), then the collections of the parabolic subgroups for their braid groups are the same. It also enables us to identify, when we have two equally likeable Garside monoids with group of fractions $B$, the parabolic subgroups obtained algebraically from each of these structures: we do so by proving their coincidence with the parabolic subgroups defined topologically in this paper.

In this paper we prove the following 3 main theorems for all irreducible complex reflection groups \emph{but the
exceptional group $G_{31}$}. Dealing with this group indeed requires a thorough generalization of several of the main ideas
of the present paper to the setting of Garside categories, and has been settled in O. Garnier's PhD thesis, written under
the guidance of the second author, see \cite{PARAB2}.
Let $W$ be an irreducible complex reflection group 
and let $B$ be its braid group.

\begin{theorem}
\label{maintheo:parabolicclosure}

For every element
$x \in B$, there exists a unique minimal parabolic subgroup $\PC(x)$ of $B$ containing it,
and we have $\PC(x^m) = \PC(x)$ for every $m \neq 0$.
\end{theorem}
This unique minimal parabolic subgroup of $B$ containing $x$ is called the \emph{parabolic closure} of $x$.

\begin{theorem}
\label{maintheo:intersectparabolics}
 If $B_1,B_2$ are two parabolic subgroups of $B$, then $B_1 \cap B_2$ is
a parabolic subgroup. More generally, every intersection of a family of parabolic subgroups is a parabolic subgroup.
\end{theorem}

Defining the join of two parabolic subgroups as the intersection of all the parabolic subgroups containing
them then defines, as in \cite{CGGW}, a \emph{lattice structure} on the collection of parabolic
subgroups.

The proof goes as follows. In Section 3 we prove that the main Garside monoids used for dealing with the complex braid groups are well adapted to our study of parabolic subgroups. More precisely, we prove in these cases that every parabolic subgroup is a conjugate of some subgroup generated by certain specific subsets of the generators of the given presentation, and that the Garside-theoretical notion and our topological notion of parabolic subgroups coincide. When there are no available Garside monoids, but $B$ is a normal subgroup of finite index of another complex braid group endowed with a Garside structure, we show how to determine the collection of parabolic subgroups of the former from the one of the latter. This covers all cases of complex braid groups except $G_{31}$.

In Section 4 we develop the needed Garside machinery. Under suitable assumptions on the
Garside monoid, we prove (see \autoref{T:parabolic_closure_exists}) that parabolic closures exist for arbitrary elements. In Section 5 we prove that these conditions are satisfied for the Garside monoids which appeared
in Section 3. In so doing, we determine the standard parabolic subgroups for the Garside structure, and find
the same `standard' subgroups determined in Section 3, so that the Garside-theoretic notion of a parabolic subgroup
and the topological notion coincide in each case.
 This proves that every element admits a parabolic closure. We then prove in Section 6,
using additional properties of our Garside structures, that the intersection of two parabolic subgroups
are parabolic subgroups. Using additional methods, we show in the same section how to conclude the
proofs for the other groups.

From this construction we build an analogue of the curve complex for the usual braid groups, generalizing
the ideas of \cite{CGGW}. Its vertices
are the irreducible parabolic subgroups of $B$ as defined above, and two distinct vertices $B_1,B_2$ are connected
if and only if either $B_1 \subset B_2$, or $B_2 \subset B_1$, or $B_1\cap B_2 = [B_1,B_2]=1$. This forms
the \emph{curve graph} $\Gamma$ for $B$. The associated simplicial
complex is then
constructed as the
clique complex of $\Gamma$, namely the flag complex
made of all the simplices whose edges belong to the curve graph. It is easy to
prove (see \autoref{prop:actionfidele}) that the group $B/Z(B)$ acts faithfully on it.

As in \cite{CGGW} for the case of real reflection groups, we prove that the graph admits a simpler description. For this, recall from \cite{BMR} and \cite{BESSIS} (see also \cite{DMM})  that
the center of any irreducible complex braid group $B$ admits a canonical positive generator $z_B$. Here positive means that its image under
the natural map $B \to \Z$ is positive. From this, a well-defined element $z_{B_0} \in B$ is associated to every irreducible parabolic subgroup $B_0$ of $B$,
that is every parabolic subgroup whose associated (parabolic)
reflection subgroup is irreducible.

\begin{theorem}
\label{maintheo:curvecomplex}
If $B_0$ is an irreducible parabolic subgroup of $B$ we have $B_0 = \PC(z_{B_0})$.
Moreover, if $B_1$ and $B_2$ are two distinct irreducible parabolic subgroups of the irreducible complex braid group $B$, an element $g \in B$ satisfies $(B_1)^g = B_2$ if and only if $(z_{B_1})^g = z_{B_2}$, and $B_1,B_2$ are adjacent if and only if $z_{B_1}z_{B_2}=z_{B_2}z_{B_1}$.
\end{theorem}
Of course this theorem admits the following immediate corollary.
\begin{corollary}
If $B_0$ is an irreducible parabolic subgroup of $B$, then the normalizer of $B_0$ is equal to the
centralizer of $z_{B_0}$.
\end{corollary}

 The above results allow us to redefine the curve graph associated to $B$ as the graph whose vertices correspond to the elements $z_{B_0}$, for $B_0$ an irreducible parabolic subgroup of $B$, and where two vertices are adjacent if and only if their corresponding elements commute. The curve complex is defined from the graph in the same way as above.

Since the simplicial complex constructed here generalizes the usual curve complex, it is natural to conjecture that the associated simplicial complex is hyperbolic in the sense of Gromov (as in e.g. \cite{GHYSHARPE} or \cite{BOWDITCH}). We prove that, in addition to the classical case, it is indeed the case
for the following additional families of complex reflection groups.

\begin{theorem}
\label{maintheo:hyperbolic} If $W$ is an irreducible complex reflection group of rank $2$, or
of type $G(de,e,n)$ for $d > 1$, then the curve complex is hyperbolic.
\end{theorem}

The case of $G(de,e,n)$ for $d > 1$ is an extension of the work of Calvez and Cineros in \cite{CALVEZCISNEROS} for $G(2,1,n)$, which is a real reflection group.
We shall prove this result in the final Section \ref{sect:proofhyperbolic} of the paper.
We also mention, as pointed to us by L. Paris, that a generalization of \cite{CGGW} in a different direction has been done in \cite{DAVIS}, where
the authors consider fundamental groups of simplicial hyperplane complement.

We end this introduction by a few comments on the Garside-theoretic aspects. The machinery used to prove the main theorems is established in the setting of a general Garside group satisfying two main properties, and the parabolic subgroups being defined in this case are the ones introduced by Godelle~\cite{GODELLE2003}. We introduce a new kind of conjugation (called `swap') and a new kind of elements (`recurrent elements'), which turn out to be very useful tools to study conjugacy in Garside groups, simplifying the usual techniques. We have a notion of `support' of an element (using recurrent elements) which generalizes the one in~\cite{CGGW}, and we provide in several cases much simpler proofs of the results that were given in~\cite{CGGW} for Artin groups of spherical type. Notably, the proof that the intersection of two parabolic subgroups is a parabolic subgroup is significantly shorter, and is valid for the more general case of the Garside monoids involved in the study of complex braid groups.

The extension of our results to the single remaining case of $G_{31}$
requires specific work, including a deeper understanding of the
parabolic subgroups of centralizers of regular elements, in terms
of Garside groupoids. This is the theme of the
companion paper \cite{PARAB2}. We thank his author O. Garnier for numerous comments on the present paper.

\section{Parabolic subgroups}

In this section we define a purely topological concept of a local fundamental group which is suitable
for our purposes, and then define a parabolic subgroup of a complex braid group as
such a local fundamental group with respect to a topological pair $(X/W,V/W)$. We then
establish its basic properties.

\subsection{Normal rays}

Let $Y$ be a topological space and $X \subset Y$ an open subset. A \emph{closed normal ray} in the topological pair $(X,Y)$ is a (continuous) path $\gamma : [0,1] \to Y$ such that
\begin{enumerate}
\item $\eta([0,1[) \subset X$
\item $\eta(1) \not\in X$
\item there exists $\alpha_0 \in ]0,1[$ such that $\eta([1-\alpha,1[)$ is simply connected for
every $\alpha \in ]0,\alpha_0]$.
\end{enumerate}
Such a closed normal ray is said to \emph{terminate at} $\eta(1)$ and to be \emph{based at} $\eta(0) \in X$. An \emph{open normal ray} is defined similarly to be a map
$\eta : ]0,1] \to Y$ with $\eta(]0,1[)\subset X$, fulfilling the same two conditions (2) and (3). Of course,
if $\eta$ is a closed normal ray, then its restriction $\check{\eta} = \eta_{|_{]0,1]}}$ is an open normal ray and, if $\eta$ is an open
normal ray, one can build a closed normal ray via $\hat{\eta} : u \mapsto \eta( \frac{1+u}{2} )$.

Probably every person reading this definition first wonders
why it is not sufficient to check condition (3) for $\alpha = \alpha_0$, therefore we provide an example. Consider the case $Y = \C$, $X = \C \setminus \{ 0 \}$ with
basepoint $1$, and consider the following path $\gamma : [0,1] \to \C$. From $0$
to $1/4$, $\gamma$ is a straight path from $1$ to $1/2 \in \C$, from
$3/4$ to $1$ it is a straight path from $1/2$ to $0$, from $1/2$ to
$3/4$ it is a circle tangent at $1/2$, and from $1/4$ to $1/2$ a Peano
curve filling the corresponding disc. Then $\gamma([0,1[)$
is simply connected but $\gamma([1/2,1[)$ is not. Therefore $\gamma$
is \emph{not} a normal ray.

\medskip

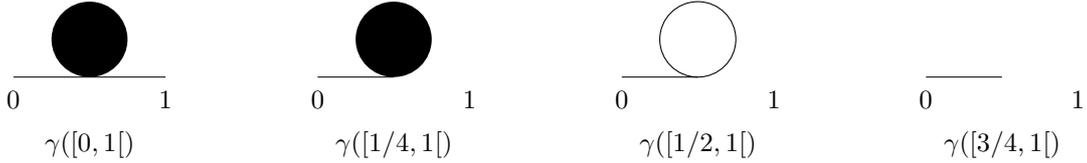
\begin{center}

\begin{tikzpicture}[scale=2]
\begin{scope}
\draw (0,0) -- (0.5,0);
\draw (1,0) -- (0.5,0);
\fill (0.5,0.25) circle (0.25);
\draw (0,-.15) node {$0$};
\draw (1,-.15) node {$1$};
\draw (0.5,-.45) node {$\gamma([0,1[)$};
\end{scope}
\begin{scope}[shift={(2,0)}]
\draw (0,0) -- (0.5,0);
\fill (0.5,0.25) circle (0.25);
\draw (0,-.15) node {$0$};
\draw (1,-.15) node {$1$};
\draw (0.5,-.45) node {$\gamma([1/4,1[)$};
\end{scope}
\begin{scope}[shift={(4,0)}]
\draw (0,0) -- (0.5,0);
\draw (0.5,0.25) circle (0.25);
\draw (0,-.15) node {$0$};
\draw (1,-.15) node {$1$};
\draw (0.5,-.45) node {$\gamma([1/2,1[)$};
\end{scope}
\begin{scope}[shift={(6,0)}]
\draw (0,0) -- (0.5,0);
\draw (0,-.15) node {$0$};
\draw (1,-.15) node {$1$};
\draw (0.5,-.45) node {$\gamma([3/4,1[)$};
\end{scope}
\end{tikzpicture}

\end{center}

A special situation which is sometimes more handy and usually sufficient for applications is when $\eta$ satisfies the stronger condition that there exists $\alpha_0 > 0$ such that the restriction of $\gamma$ to $[1-\alpha_0,1[$ is
a homeomorphism $[1-\alpha_0,1[ \to \gamma([1-\alpha_0,1[)$. We call this a \emph{straight normal ray}.

Recall that, if $A \subset X$ is path-connected and simply-connected, then $\pi_1(X,A)$ is well-defined: either as the groups $\pi_1(X,a)$ for $a \in A$ canonically identified with each other via an arbitrary path joining the base points inside $A$, or as the set of classes of paths from some point of $A$ to some other one, up to a homotopy leaving the endpoints inside $A$.

Because of the defining conditions of a normal ray,
the fundamental groups $\pi_1(X,\eta([1-\alpha,1[))$ for $\alpha \in ]0,\alpha_0]$ can be canonically identified under the natural maps $\pi_1(X,\eta([1-\alpha_2,1[)) \to \pi_1(X,\eta([1-\alpha_1,1[))$
for $\alpha_0 \geq \alpha_1 > \alpha_2 > 0$. We denote it $\pi_1(X,\eta)$.
Of course, if $x \in \eta(]1-\alpha_0,1[)$, then the natural morphism
$\pi_1(X,x) \to \pi_1(X,\eta)$ is an isomorphism.

From this definition it is readily checked that, if $\eta$ is closed, then we have a natural isomorphism $\pi_1(X,\check{\eta}) \to \pi_1(X,\eta)$ and,
if $\eta$ is open, we have similarly $\pi_1(X,\hat{\eta}) \stackrel{\simeq}{\to} \pi_1(X,\eta)$.

If $\eta$ is a closed normal ray, an isomorphism $\pi_1(X,\eta(0)) \to \pi_1(X,\eta)$ can be
defined via
$$
\pi_1(X,\eta) \simeq \pi_1(X,\eta([1-\alpha_0,1[)  \simeq \pi_1(X,\eta(1-\alpha_0))  \to \pi_1(X,\eta(0)),
$$
the map $\pi_1(X,\eta(1-\alpha_0)) \to \pi_1(X,\eta(0))$ being $[\gamma] \mapsto [ \eta_{|_{[0,1-\alpha_0]}}^{-1}*\gamma *\eta_{|_{[0,1-\alpha_0]}}]$, where the concatenated path $\alpha*\beta$ means $\beta$ followed by $\alpha$.

We extend in an obvious way the notation $\pi_1(X,\eta)$ to
$\pi_1(U\cap X,\eta)$ where $U \subset Y$ is some open neighborhood of $\eta(1)$,
and
$\pi_1(U\cap X,\eta)$ means
$\pi_1(U\cap X,\eta')$ where
$\eta'$ is the restriction of
$\eta$ to some $]1-\alpha,1[$ with $\alpha < \alpha_0$ and $\eta(]1-\alpha,1[) \subset U$. By the above observations, this does not depend on the choice of such an $\alpha$.

\subsection{Local fundamental groups}

The topological pair $(X,Y)$ is said to admit a \emph{local fundamental group} at $\eta$ if the
image of the obvious map $\pi_1(X\cap U , \eta) \to \pi_1(X,\eta)$ for $U \subset Y$ an open neighborhood of $\eta(1)$ does not depend on $U$ for $U$ small enough. This means that
there exists an open neighborhood $U_0$ such that, for any other $U \subset U_0$, the composite
of the maps $\pi_1(X\cap U,\eta) \to \pi_1(X\cap U_0,\eta) \to \pi_1(X,\eta)$ has the
same image as $\pi_1(X\cap U_0,\eta) \to \pi_1(X,\eta)$.
It is obviously equivalent to saying that this image is the same for $U$ belonging to some
fundamental system of neighborhoods of $\eta(1)$. If it is the case, we denote $\pi_1^{loc}(X,\eta)$ this image.

Notice that, if $\eta$ is closed, then $(X,Y)$ admits a local fundamental fundamental group at $\eta$ if and only if it admits one at the open normal ray $\check{\eta}$, and conversely it admits a local fundamental group
at the open normal ray $\eta$ if and only if it admits one at the closed normal ray $\hat{\eta}$.

\begin{proposition} \label{prop:parabsbmr}
If $W < \GL(V)$ is a complex reflection group and $X$ its hyperplane complement, then the
topological pair $(X/W,V/W)$ admits a local fundamental group at every normal ray.
\end{proposition}
\begin{proof}
Let $\eta$ be such a normal ray. By the above remark we can assume that it is closed. Let $x_0 \in X$ such that $W.x_0 = \eta(0)$.
By definition $\eta_{|_{[0,1[}}$ is a path inside $X/W$, so since $X \to X/W$ is a covering map it can be lifted to a path $\tilde{\eta} : [0,1[ \to X$ such that $\tilde{\eta}(0)=x_0$. Now $\eta(1) = W.v$ for some $v \in V$. Since $W.v \subset V$
is finite and $V$ is Hausdorff, $W.v$ admits an open neighborhood $U \subset V$ such that
each connected component $U_0$ of $U$ meets $W.v$ at exactly one point. Since $[1-\alpha,1[$ is connected
we get that, for $\alpha$ small enough, it is contained in such a connected component $U_0$. Defining $v_0$
by $\{ v_0 \} = U_0 \cap W.v$ and setting $\tilde{\eta}(1) = v_0 \in V \setminus X$ one gets a lifting $\tilde{\eta} : [0,1] \to V$
of $\eta$.
Let $W_0$ the parabolic subgroup of $W$ fixing $v_0$, $L_0$ the intersection of its
reflecting hyperplanes, and $X_0 \supset X$ the complement of the reflecting hyperplanes
of $W_0$.
We denote $\eta_0 : [0,1] \to V/W_0$ the composite of $\tilde{\eta}$ with the natural
projection map $V \to V/W_0$. Since $\eta_0([0,1[) \subset X/W_0$, the restriction
$(\eta_0)_{|[0,1[}$ is equal to the unique lifting of $\eta : [0,1[ \to X/W$ under
the covering map $X/W_0 \to X/W$ starting at $W_0.x_0$. Since $\eta$ is a normal ray
this implies that each $\eta([1-\alpha,1[)$ is simply connected for $\alpha$ small enough
and therefore each $\eta_0([1-\alpha,1[)$ is simply connected; indeed, the image inside $\eta([1-\alpha,1[)$
of any
non-trivial loop of $\eta_0([1-\alpha,1[)$ could be homotopied to a constant loop, and the corresponding
homotopy then lifted to a homotopy of the original loop. This proves that $\eta_0$ is also a (closed) normal ray.

 Choosing some $W$-invariant norm on $V$,
for some $\eps_0 > 0$
the open balls $\Omega_{\eps}$ of radius $\eps> 0$ and center $v_0$
do not cross any other reflecting hyperplanes
than the ones of $W_0$ for $0 < \eps \leq \eps_0$. This provides morphisms
$\pi_1(\Omega_{\eps}\cap X/W_0,\eta_0) \to \pi_1(X/W,\eta)$, which do not depend on $\eps \in ]0,\eps_0[$
as, for $0 < \eps' < \eps < \eps_0$, the composite
$\pi_1(\Omega_{\eps'}\cap X/W_0,\eta_0) \to \pi_1(\Omega_{\eps}\cap X/W_0,\eta_0)
\to \pi_1(X_0/W_0,\eta_0)$ is known to be an isomorphism (see \cite{BMR}  Proposition 2.29), whence
$\pi_1(\Omega_{\eps'}\cap X/W_0,\eta_0) \stackrel{\simeq}{\to} \pi_1(\Omega_{\eps}\cap X/W_0,\eta_0)$.
Since the $\Omega_{\eps}\cap X/W_0$ for $\eps < \eps_0$ are mapped homeomorphically into $V/W$ with images forming a fundamental system of open neighborhoods
of $\eta(1)$, this proves that $\pi_1^{loc}(X/W,\eta)$ is well-defined.

\end{proof}

Actually, this proof shows also that $\pi_1^{loc}(X/W,\eta) \simeq \pi_1(X_0/W_0,\eta_0)$
where the parabolic subgroup $W_0$ has been defined
as the stabilizer of $\eta(1)$ and $X_0$ is its hyperplane complement. It also shows that these local fundamental groups are unchanged if $\eta$ is modified near $t=1$ while staying inside some neighborhood of its endpoint, so in particular all such local fundamental groups can be afforded by \emph{straight} normal rays.
From these arguments,
the proof of the following proposition is straightforward.

\begin{proposition} Let $\tilde{\eta}$ be a normal ray inside $(X,\C^n)$ where $X$ is the hyperplane complement of the complex reflection group $W < \GL_n(\C)$. Then the composite $\eta$ of $\tilde{\eta}$ with $X \to X/W$
is a normal ray. Moreover
\begin{enumerate}
\item If $\tilde{\eta}(1)$ belongs to the intersection
of all the reflecting hyperplanes
of $W$, then $\pi_1^{loc}(X/W,\eta) \simeq \pi_1(X/W,\eta)$.
\item If $W_0$ is the stabilizer of $\tilde{\eta}(1)$ and $X_0 \supset X$ is its hyperplane complement,
then $\pi_1^{loc}(X/W,\eta) \simeq \pi_1^{loc}(X_0/W_0,\eta_0)$, where $\eta_0$ is the composite of
$\tilde{\eta}$ with $X_0 \to X_0/W_0.$
\end{enumerate}
\end{proposition}

\subsection{Parabolic subgroups of complex braid groups}
\label{sect:parabdefinition}

From the results above we can define the following concept.

\begin{definition} Let $W < \GL(V)$ be a complex reflection group and $X$ be its hyperplane complement.
Let $v_0 \in X/W$ and $B = \pi_1(X/W,v_0)$ the braid group of $W$. A \emph{parabolic subgroup} of
$B$ is the image of $\pi_1^{loc}(X/W,\eta)$ for some closed (straight) normal ray $\eta$ based at $v_0$
under the maps $\pi_1^{loc}(X/W,\eta) \to \pi_1(X/W,\eta) \simeq \pi_1(X/W,v_0) = B$.
\end{definition}

Choosing a preimage $\tilde{v}_0$ of $v_0$ inside $X$, the covering map $X \to X/W$ defines a projection map $B \to W$ with kernel $P = \pi_1(X,\tilde{v}_0)$.
It is readily checked from e.g. \cite{BMR} that the image of the parabolic
subgroup attached to $\eta$ under the projection map $\pi : B = \pi_1(X/W,v_0) \to W$
is a parabolic subgroup of $W$. More precisely, it is the parabolic subgroup fixing
the collection of reflecting hyperplanes containing $\tilde{\eta}(1)$, where $\tilde{\eta}$
is the unique lift of $\eta$ with $\tilde{\eta}(0) = \tilde{v}_0$. In this context, we shall need the following lemma.

\begin{lemma}\label{lem:avoidinghyperplanes} Let $H_1,\dots,H_r$ be a collection of hyperplanes
such that $v_0 \not\in W.H_i$, $i=1,\dots,r$. Let $B_0$ be a parabolic subgroup of $B = \pi_1(X/W,v_0)$, obtained as the image of $\pi_1^{loc}(X/W,\eta)$ for some $\eta$, and
let $L_0$ the intersection of the reflecting hyperplanes containing $\tilde{\eta}(1)$, for $\tilde{\eta} : [0,1] \to V$ lifting $\eta$. Assume that $\forall i\ L_0 \not\subset H_i$. Then $B_0$ can be obtained as the image of $\pi_1^{loc}(X/W,\eta')$
with
  $\eta'$ such that $\forall i \ \eta'([0,1]) \cap W.H_i = \varnothing$.
\end{lemma}
\begin{proof} Let $W_0$ be the pointwise stabilizer of $L_0$, which is the parabolic subgroup of $W$ obtained as the image of $B_0$ under $\pi : B \onto W$.
From the assumptions we get that the complement of the $H_i$'s is dense inside $L_0$.
From this and the construction as the image of $\pi_1(\Omega_{\eps}\cap X/W_0,\eta)$ for some $\eps > 0$ in the proof of
\autoref{prop:parabsbmr} it is then clear that $\eta$ can be slightly modified near $\eta(1)$ in such a way
that it remains the same on some $[0,1-\beta]$, does not cross any of the $H_i$'s on $[1-\beta,1]$,
and still provide the same image inside $B_0$. But now $\eta_{|_{[0,1-\beta]}}$ is a path inside some open set of $\C^n$ joining two elements not belonging to $F = H_1\cup \dots \cup H_r$. Since $F$ is a finite union of (real) codimension 2 subspaces, $\eta_{_{[0,1-\beta]}}$ is homotopically equivalent
to a path which avoids $F$, so we can modify $\eta$ accordingly on $[0,1-\beta]$ and get the
same image inside $B = \pi_1(X/W,v_0)$.
\end{proof}

We then have the following proposition, which summarizes the basic properties of parabolic subgroups of braid groups. Informally, (1) says that the concept of a parabolic is
stable after reducing to essential reflection arrangements; (2) says that conjugates of parabolic subgroups are parabolic subgroups ; (3) says that a parabolic subgroup of $B$ is naturally isomorphic to the braid group of a parabolic subgroup of $W$ ; (4) says that, under this identification,
 parabolic subgroups of parabolic subgroups are parabolic subgroups.

\begin{proposition} \label{prop:basicpropsparabs}
Let $\mathcal{A}$ denote the hyperplane arrangement of $W$,
and $B= \pi_1(X/W,v_0)$. Then the following hold.
\begin{enumerate}
\item Let $V_+$
be a linear subspace of $V$ contained in all the reflecting hyperplanes. Let $\bar{V} = V/V_+$ and let
$\bar{v}_0$, $\bar{X}$ be the images of $v_0$ and $X$ inside $\bar{V}/W$ and $\bar{V}$, respectively.
Then the natural isomorphism $\pi_1(X/W,v_0) \to \pi_1(\bar{X}/W,\bar{v}_0)$ induces a bijection
between their sets of parabolic subgroups.
\item If $B_0 < B$ is a parabolic subgroup, then $g B_0 g^{-1}$ is a parabolic subgroup
for every $g \in B$
\item If $B_0 < B$ is a parabolic subgroup, then $B_0 \simeq \pi_1(X_0/W_0,v'_0)$
where $v'_0$ is a preimage of $v_0$ under $X/W_0 \to X/W$, and $X_0$ is the hyperplane complement for the image $W_0$ of $B_0$ under $B \to W$.
\item If $B_0 < B$ is a parabolic subgroup of $B$ and $v'_0, W_0, X_0$ as above, then every parabolic subgroup
of $\pi_1(X_0/W_0,v'_0)$ is identified with a parabolic subgroup of $B$ via
 $\pi_1(X_0/W_0,v'_0) \simeq B_0 < B$.
\end{enumerate}
\end{proposition}
\begin{proof}
We first prove (1).
The isomorphism $\pi_1(X/W,v_0) \to \pi_1(\bar{X}/W,\bar{v}_0)$ induces a bijection
between their sets of subgroups.
The fact that parabolic subgroups of $\pi_1(X/W,v_0)$ are mapped to
parabolic subgroups of $\pi_1(\bar{X}/W,\bar{v}_0)$ is an immediate consequence of the commutativity
of the diagram
\begin{center}
\begin{tikzcd}
\pi_1^{loc}(X/W,\eta) \arrow[r] \arrow[d] & \pi_1^{loc}(\bar{X}/W,\bar{\eta}) \arrow[d] \\
\pi_1(X/W,\eta) \arrow[r, "\simeq"] & \pi_1(\bar{X}/W,\bar{\eta})
\end{tikzcd}
\end{center}
In order to prove that, conversely, every parabolic subgroup of $\pi_1(\bar{X}/W,\bar{v}_0)$
is obtained from a parabolic subgroup of $\pi_1(X/W,v_0)$, it then remains to show that
every normal ray inside $\bar{X}/W$ can be lifted to some normal ray inside $X/W$, and
this is clear from the existence of $W$-equivariant homeomorphisms $V \simeq \bar{V} \times V_+$
given by choosing some orthogonal complement of $V_+$ inside $V$ for some $W$-invariant orthogonal
form on $V$. This proves (1).

If $B_0$ is the image of $\pi_1^{loc}(X/W,\eta) \to \pi_1(X/W,\eta) \simeq \pi_1(X/W,v_0) = B$,
and $g = [\gamma]$ for some $\gamma : [0,1] \to X/W$ with $\gamma(0)=\gamma(1) = v_0$, then
$g B_0 g^{-1}$ is the image of $\pi_1^{loc}(X/W,\eta*\gamma^{-1})$,
as $\eta*\gamma^{-1}$ is easily checked to be another normal ray based at $v_0$. This proves (2). Part (3) was already noticed after \autoref{prop:parabsbmr} as a consequence of
its proof.

We prove (4). Let us consider a parabolic subgroup given as the image of $\pi_1^{loc}(X_0/W_0,\tau)\to \pi_1(X_0/W_0,v'_0)$ for some closed normal ray $\tau$ based at $v'_0$, and let $\tilde{\tau} : [0,1] \to V$ be a lift of $\tau$. Let $W_{00}$ be the
stabilizer of $\tau(1)$ inside $W_0$, and $L_{00}$ its fixed point set. If $H$ is a reflecting
hyperplane for $W$ but not for $W_0$, it does not contain $L_{00}$, for otherwise it would contain
the fixed point set of $W_0$ and would be a reflecting hyperplane for $W_0$. Therefore we can apply
\autoref{lem:avoidinghyperplanes} and assume that $\tau$ does not cross any of these (orbits of) hyperplanes.
It follows that $\tau$ defines a normal ray inside $X/W_0$.  Without changing the parabolic subgroup it defines, we know that
we can modify its end, provided that it stays inside a suitable neighborhood $\Omega_{\eps}\cap X/W_0$ of its endpoint.
From this we can assume that its image inside $X/W$ remains a normal ray.
Therefore the subgroup we are interested in is the image of the map
$$
\pi_1\left(\frac{X_0 \cap \Omega_{\eps}}{W_0},\tau\right) =
\pi_1\left(\frac{X \cap \Omega_{\eps}}{W_0},\tau\right)
\to \pi_1\left(\frac{X}{W},v_0\right)
$$
which is a parabolic subgroup of $B$ by the proof of \autoref{prop:parabsbmr}.

\bigskip

\end{proof}

We now relate conjugacy classes of parabolic subgroups with their images inside $W$.

\begin{proposition} \label{prop:conjparabs}
Let $B_1,B_2$ be two parabolic subgroups of $B=\pi_1(X/W,v_0)$, and $W_1,W_2$ their images inside $W$.
We set $P = \Ker(B \onto W)$.
\begin{enumerate}
\item $W_1=W_2$ if and only if $B_1$ and $B_2$ are $P$-conjugates.
\item $W_1$ and $W_2$ are $W$-conjugates if and only if $B_1$ and $B_2$ are $B$-conjugates.
\end{enumerate}
\end{proposition}
\begin{proof}
The 'if' parts are clear, so we shall prove only the other direction. Moreover, assuming that
(1) is true, if $W_1$ and $W_2=w W_1 w^{-1}$ are $W$-conjugates we can find $b \in B$ with image
$w$, so that $b B_1 b^{-1}$ is still a parabolic subgroup by \autoref{prop:basicpropsparabs} (2), with image $W_2$, so that it is a $P$-conjugate of $B_2$, which proves (2).
Therefore it only remains to prove that if $W_1=W_2$, then $B_1,B_2$ are $P$-conjugates.

In order to do this, we set $W_0 = W_1=W_2$, denote $L_0=V^{W_0}$ its fixed point set, and choose some lift $\tilde{v}_0 \in X$ of $v_0$.
We fix a $W$-invariant hermitian scalar product on the ambient space.
The parabolic subgroups $B_i$ are given as the image of $\pi_1^{loc}(X/W_i,\eta_i)$
inside $\pi_1(X/W,v_0)$,
for some closed normal rays $\eta_i$ with $\eta_i(0) = v_0$.

 We set $\tilde{\eta}_i : [0,1] \to V$ the lift of $\eta_i$ such
that $\tilde{\eta}_i(0) = \tilde{v}_0$, so that $\tilde{\eta}_i(1) = x_i^0  \in L_0$. Then, $P_i = P  \cap B_i$ is the image
of $\pi_1^{loc}(X,\tilde{\eta}_i)$ inside $\pi_1(X,\tilde{v}_0)$, and we need to prove in particular that $P_1,P_2$
are conjugates inside $P = \pi_1(X,\tilde{v}_0)$.

 Since the complement $L'_0$ inside $L_0$ of the
union of the set $\mathcal{A} \setminus \mathcal{A}_0$ of the reflecting hyperplanes of $W$ not containing $L_0$ is open and path connected, one can choose a $C^{\infty}$ simple curve $\eta_0 : [0,1] \to L'_0$ with $\eta_0(\eps_0) = x_1^0$, $\eta_0(1-\eps_0) = x_2^0$, for some $\eps_0 > 0$ close to $0$, and we can assume $\eta_0'(t) \neq 0$ for every $t$. Then, $\eta_0(]0,1[)$ is a 1-dimensional submanifold of $L'_0$, and is contained in some normal tube $T_0 = \{ \eta_0(t)+D(t,r) ; t \in
]0,1[ \}$, where $D(t,r)$ is the disc of radius $r$ centered at $0$ inside the subspace orthogonal to $\gamma'(t)$ inside $L_0$.
Moreover, $r$ can be chosen small enough to be a tubular neighborhood, that is $T_0$ is
homeomorphic to $]0,1[ \times D$ for some $(\dim_{\R} L_0-1)$-disc $D$ via $(t,v) \mapsto \eta_0(t) + h(t,v)$ with
$h(t,v) \in D(t,r)$. Then we set $T = T_0+F$, where $F = \{  y \in L_0^{\perp} , \| y \| \leq r \}$.
Up to possibly reducing $r$, one can assume that $T$ does not cross any hyperplane of $\mathcal{A}\setminus \mathcal{A}_0$. We then have a homeomorphism $\Psi : ]0,1[\times D \times F \to T$, $(t,v,f) \mapsto \eta_0(t) + h(t,v)+f$, with the property that the inverse image of $X \cap T$ is $]0,1[\times D \times F'$
with $F'$ the complement inside $F$ of the union of the hyperplanes of $\mathcal{A}_0$.

Then, notice that, for $0 < \eps < \eps_0$,  $T^1(\eps) = \{ \eta_0(t) + h(t,v) + y ;\ t\in ]-\eps+\eps_0,\eps_0+\eps [, v\in D, y \in L_0^{\perp} , \| y \| \leq r \}$ is an open neighborhood of $x_1^0$, and
similarly $T^2(\eps) = \{ \eta_0(t) + h(t,v) + y ; t\in ]-\eps+1-\eps_0,1-\eps_0+\eps [, v\in D, y  \in L_0^{\perp} , \| y \| \leq r \}$ is an open neighborhood of $x_2^0$. Therefore, one can choose $\eps$ and $x_i =\eta_i(\alpha_i) \in T^i(\eps)$ with $\alpha_i<1$ so that
$\pi_1^{loc}(X,\eta_i) \simeq \pi_1(X\cap T^i(\eps),x_i)$ for $i=1,2$. Then,
the map $ \pi_1(X\cap T^i(\eps),x_i) \to \pi_1(X,\tilde{v}_0)$, whose image is $P_i$,
factors as $ \pi_1(X\cap T^i(\eps),x_i) \to  \pi_1(X\cap T,x_i) \to \pi_1(X,\tilde{v}_0)$.
Now, the inclusion map $N_i : X\cap T^i(\eps)\to X \cap T$ is a homotopy equivalence. Indeed,
if $\Psi^{(i)}$ denotes the restriction of $\Psi$ to $H_i\times D \times F \to T$,
with $H_1 = ]\eps_0 - \eps,\eps_0 + \eps[$ and $H_2 = ]1-\eps_0 - \eps,1-\eps_0 + \eps[$,
then $\Psi^{-1} \circ N_i\circ  \Psi^{(i)}$ is the inclusion map $H_i\times D \times F \to T\subset ]0,1[\times D \times F$, which is a homotopy equivalence. It follows that each $\pi_1(X\cap T^i(\eps),x_i) \to  \pi_1(X\cap T,x_i)$ is an isomorphism, and therefore each $P_i$ is the image inside of $\pi_1(X,\tilde{v}_0)$ of $\pi_1(X\cap T,x_i)$.
From this, and noticing that each $T_i(\eps)$ and $T$ are $W_0$-invariant, one
gets that the natural morphism $\pi_1(X \cap T_i(\eps)/W_0,x_i) \to \pi_1(X \cap T/W_0,x_i)$ is also an isomorphism, and therefore $B_i$ is the image inside $\hat{B}= \pi_1(X/W_0,\tilde{v}_0)$ of $\pi_1(X\cap T/W_0,x_i)$.
But then, letting $\hat{\eta}_i$ denote the restriction of $\eta_i$ to $[0,\alpha_i]$, and $h$ some path $x_1 \leadsto x_2$
inside $X \cap T$, we check immediately that the loop $\hat{\eta}_2^{-1}*h*\hat{\eta}_1$ based at $\tilde{v}_0$
conjugates the two images. Identifying $\hat{B}$ with the subgroup $\pi^{-1}(W_0)$
of $B$, this concludes the proof of the proposition.
 \end{proof}

Our definition of a parabolic subgroup as depending only on the \emph{topological pair} $(X/W,V/W)$ has the following consequence. Assume that we are given a choice of (homogeneous) basic invariants $f_1,\dots,f_n \in \C[V]$ for $W$,
that is, $W$-invariant homogeneous polynomials such that
$\C[V]^W = \C[f_1,\dots,f_n]$,
and denote $f = (f_1,\dots,f_n) : V \to \C^n$ the corresponding polynomial map.
The union of the reflecting hyperplanes is mapped under $f$ onto some hypersurface $\mathcal{H}$.
Then $f$ induces a map $\hat{f} : V/W \to \C^n$ which maps $X/W$ to the hypersurface complement $\C^n \setminus \mathcal{H}$.
This map is known to be (continuous and) bijective (see e.g. \cite{LEHRERTAYLOR} Proposition 9.3), and easily seen to be proper, because $\C[V]$ is an integral extension of $\C[V]^W$ (see e.g. \cite{LEHRERTAYLOR}, Lemma 3.11). Therefore it is a homeomorphism, and the topological pair
$(X/W,V/W)$ is homeomorphic to the topological pair $(\C^n \setminus \mathcal{H},\C^n)$, and the concept of a
parabolic subgroup depends only on the hypersurface $\mathcal{H}$.

\begin{proposition} \label{prop:parabolicsdiscri}
Assume that $W,W' < \GL(V)$ are two complex reflection groups with two collections of basic invariants
$f_1,\dots,f_n$ and $f'_1,\dots,f'_n$ for which the
associated discriminant hypersurfaces $\mathcal{H} \subset \C^n$ are the same. Then the parabolic subgroups for $W$ and $W'$ define the same collection of parabolic subgroups of their common braid group $B = \pi_1(\C^n \setminus \mathcal{H},z)$.
\end{proposition}

The term \emph{isodiscriminantal} was coined by Bessis
in order to describe such pair of complex reflection groups. This applies in particular to the so-called Shephard groups. Indeed, it is immediately checked in rank 2 and for the infinite series, and it was shown by Orlik and Solomon (see \cite{ORLIKSOLOMON}) that the exceptional Shephard groups of higher rank, namely $G_{25}$, $G_{26}$ and $G_{32}$, have the same discriminant as
the Coxeter groups of type $A_3,B_3$ and $A_4$, respectively.

\subsection{Subgroups of parabolic subgroups and products}
\label{sect:subgroupsproducts}

In this section we shall use systematically the construction of parabolic subgroups
from the choice of convenient balls inside the hyperplane complement, due to \cite{BMR}, as in
the proof of \autoref{prop:parabsbmr}.
We have the following companion to part (4) of \autoref{prop:basicpropsparabs}.

\begin{proposition} \label{prop:parabsherite}
Under the previous notations, if $B_0,B_1$ are two parabolic subgroups of $B=\pi_1(X/W,W.\tilde{v}_0)$,
and $B_1<B_0$, then $B_1$ is a parabolic subgroup of $B_0 \simeq \pi_1(X_0/W_0,W_0.\tilde{v}_0)$.
\end{proposition}

\begin{proof} In order to lighten notations in this proof, as no ambiguity should occur we allow ourselves to write $\pi_1(E/G,x)$ for $x \in E$ instead of $\pi_1(E/G,G.x)$ whenever $E$ is a $G$-space.
Since $B_1 < B_0$ we have $W_1 = \pi(B_1) < \pi(B_0) = W_0$ ; since $W_1$ is a parabolic subgroup of $W$, if follows that $W_1$ is a parabolic subgroup of $W_0$. Now
let $\mathcal{A}_0,\mathcal{A}_1,\mathcal{A}$ denote
the hyperplane arrangements of $W_0,W_1$ and $W$, respectively. We have $\mathcal{A}_1 \subset \mathcal{A}_0$.
As a $\C W_0$-module, $V$ can be canonically decomposed as
$V = V_0 \oplus U$, with
$U= V^W$ being the isotypic component of the trivial representation. Then, every reflecting hyperplane inside $\mathcal{A}_0$ can be written as $H \oplus U$ for $H$ a hyperplane of $V_0$. We denote $\mathcal{B}_0$ the
collection of such hyperplanes, and $\mathcal{B}_1$ the subset of those which originate from $\mathcal{A}_1$.

Then, setting $Y_0 = V_0 \setminus \bigcup\mathcal{B}_0$, we have $X_0 = Y_0 \times U$ and
$Y_0$ can be identified as $\overline{X}_0$ defined as in (1) of \autoref{prop:basicpropsparabs}.
Writing $\tilde{v}_0 = \tilde{v}_0'+u_0 \in V_0 \oplus U$, this provides an identification of $B_0 \simeq \pi_1(X_0/W_0,\tilde{v}_0)$ as a parabolic subgroup of $B$
with the braid group $\pi_1(Y_0/W_0,\tilde{v}_0')$ which preserves the collection of parabolic subgroups.

We endow $V$ with a $W_0$-invariant norm of the form $\| (y_1,y_2)\| =
\max( \| y_1 \|, \| y_2 \|)$ for $(y_1,y_2) \in V_0 \times U = V$, for some arbitrary $W_0$-invariant norm on $V_0$ and some arbitrary norm on $U$. We choose it non-Euclidean in order to get a convenient Cartesian decomposition
of the balls.

Since the statement is invariant under conjugation, we can assume that $B_0$ is the image of
$$
j_0 : \pi_1\left( \frac{\Omega_0 \cap X}{W_0},x_0 \right) \to B = \pi_1\left( \frac{X}{W},x_0 \right)
$$
induced by the natural inclusion and quotient by $W$,
for some $x_0 \in \Omega_0 \setminus \bigcup\mathcal{A}$
and $\Omega_0$ some open ball. Since $B_1$ is a parabolic subgroup of $B$ admitting no other hyperplanes than the ones of $\mathcal A_1\subset \mathcal A_0$, there is a $W_1$-invariant ball $\Omega_1$ containing
some $x_1 \in \Omega_1 \setminus \bigcup \mathcal{A}$
and meeting no other hyperplane than the ones of $\mathcal{A}_0$.

Then, a path $\gamma : x_0 \leadsto x_1$ inside $X$
such that $b \mapsto (\pi(b).\gamma^{-1})* b * \gamma$ defines an embedding
$$
J_{\gamma} : \pi_1\left( \frac{\Omega_1 \cap X}{W_1},x_1 \right) \to
\pi_1\left( \frac{X}{W},x_0 \right)
$$
whose image is $B_1$. Because of the specific properties of the chosen norm, under
$V = V_0 \oplus U$ we have a decomposition $\Omega_1 = \Omega_1^0 \times \Omega_1'$ with $\Omega_1^0$ a $W_1$-invariant ball inside $V_0$. Then
$$
\Omega_1 \cap X = \Omega_1 \cap X_0 = (\Omega_1^0 \times \Omega_1') \cap (Y_0 \times U)
= (\Omega_1^0 \setminus \bigcup \mathcal{B}_0) \times U \subset Y_0 \times U
$$
is naturally homotopically equivalent to $\Omega_1^0 \setminus \bigcup \mathcal{B}_0 = \Omega_1^0 \cap Y_0$.
Therefore, the same formula $b \mapsto (\pi(b).\gamma^{-1})* b * \gamma$ defines an embedding
$$
J^0_{\gamma} : \pi_1\left(\frac{\Omega_1 \cap X}{W_1},x_1\right) = \pi_1\left(\frac{\Omega_1^0\setminus \bigcup \mathcal{B}_0}{W_1}\times U ,x_1\right) \to
\pi_1\left(\frac{Y_0}{W_0}\times U,x_0\right)
$$

 The natural inclusions together with the morphisms $J_{\gamma},J_{\gamma}^0$ induce
the following commutative diagram, which proves that
$B_1$ is indeed the image under $j_0$ of a parabolic subgroup of $B_0 = \pi_1(X_0/W_0,x_0)$
attached to $W_1<W_0< \GL(V)$.
$$
\xymatrix{
\pi_1\left( \frac{X}{W},x_0  \right) & \ar[l]_{j_0} \pi_1\left(\frac{\Omega\cap X}{W_0},x_0 \right) \ar[r]_{\simeq}  & \pi_1\left( \frac{X_0}{W_0},x_0\right)\ar@{=}[r] & \pi_1\left(\frac{Y_0}{W_0}\times U,x_0 \right) \\
 & \pi_1\left( \frac{\Omega_1 \cap X}{W_1},x_1 \right) \ar[ul]^{J_{\gamma}} \ar@{.>}[u]_{\ } \ar@{=}[rr] & &  \pi_1\left(\frac{\Omega_1^0\cap Y_0}{W_1} \times \Omega_1',x_1 \right)\ar[u]_{J_{\gamma}^0}
}
$$

\end{proof}

An immediate consequence of the proposition and its proof is the following.

\begin{corollary}\label{cor:B0inclusB1egal} If $B_0,B_1$ are two parabolic subgroups of $B$ with $B_1 < B_0$, then $B_1$ is equal to $B_0$ if and only if their images inside $W$ are the same, if and only if they have the same rank.
\end{corollary}

Assume now that $V =V_1\oplus V_2$ and $W = W_1 \times W_2$
as reflection groups, that is $W_i < \GL(V)$
acting trivially on $V_j$ for $\{i,j\}=\{1,2 \}$.
The hyperplane arrangement associated to $W$ is $\mathcal{A} = \mathcal{A}_1 \cup \mathcal{A}_2$, where $\mathcal{A}_i$ denotes the hyperplane arrangement attached to $W_i$. Of course $V_j \subset H$ for all $H \in \mathcal{A}_i$, $\{i,j\}=\{1,2 \}$.
Let us choose $y = (y_1,y_2) \in V \setminus \bigcup \mathcal{A}$ as basepoint,
$$
B = \pi_1\left(\frac{V\setminus \bigcup \mathcal{A}}{W},(y_1,y_2)\right),\ \ \
B_1 = \pi_1\left(\frac{V_1\setminus \bigcup \mathcal{A}_1}{W_1},y_1\right),\ \ \
B_2 = \pi_1\left(\frac{V_2\setminus \bigcup \mathcal{A}_2}{W_2},y_2\right)
$$
There is a natural embedding $B_1 \into B$ given by $[\gamma] \mapsto [t \mapsto (\gamma(t),y_2)]$ and a natural
projection $B \onto B_1$, and similarly for $B_2$. These maps provide a natural isomorphism $B \simeq B_1 \times B_2$. For $i = 1,2$, let us choose $W_i$-invariant norms on $V_i$, and choose the norm on $V = V_1 \oplus V_2$ defined by
$\| (z_1,z_2) \| = \max( \| z_1 \|, \| z_2 \|)$.
Then, a ball $\Omega$ centered at $(y_1^0,y_2^0)$ is the cartesian product $\Omega_1 \times \Omega_2$ of the balls of the same radius centered at $y_1^0 \in V_1$ and $y_2^0 \in V_2$. Let $W'_i<W_i$ the stabilizer of $y_i^0$.
The stabilizer of $(y_1^0,y_2^0)\in V$ is $W_1^0 \times W_2^0 \subset W_1 \times W_2 = W$,
and we have $\Omega \setminus\bigcup \mathcal{A} =
(\Omega_1 \setminus \bigcup \mathcal{A}_1) \times
(\Omega_2 \setminus \bigcup \mathcal{A}_2)$. If $\check{B}$, $\check{B}_1$ and $\check{B}_2$ denote the parabolic subgroups of $B$, $B_1$ and $B_2$, respectively, attached to these
data, we get from this that $\check{B} = \check{B}_1 \times \check{B}_2$, where $\check{B}_i$ is identified with its image under the natural embedding $B_i \into B$.
The case of an arbitrary parabolic subgroup, which is conjugated to one attached to such a ball, is immediately deduced from it, as
$$(\check{B}_1 \times \check{B}_2)^{([\gamma_1],[\gamma_2])} = \check{B}_1^{[\gamma_1]} \times \check{B}_2^{[\gamma_2]}
$$

\subsection{The curve graph and the curve complex}
\label{sect:defcurvegraph}

As in the introduction, we define the (nonoriented) curve graph $\Gamma$ of $B$ with vertices the irreducible
parabolic subgroups of $B$, and edges the pairs $\{ B_1,B_2 \}$ with $B_1 \neq B_2$ such that,
either $B_1 \subset B_2$, or $B_2 \subset B_1$, or $B_1 \cap B_2 = [B_1,B_2]=\{ 1 \}$. We call rank of a given
vertex the rank of the corresponding irreducible subgroup.

Similarly, we
recall that, for $B_0$ an irreducible parabolic subgroup,
$z_{B_0} \in B_0$ is defined as the unique positive generator of
$Z(B_0) \simeq \Z$, where positive means the following. Consider
the standard map $X \to \C^*$ obtained by taking the product
of a $W$-invariant collection of defining linear forms for the
hyperplane arrangement $\mathcal{A}$. The induced morphism
$B = \pi_1(X/W) \to \pi_1(\C^*) \simeq \Z$ identifies $Z(B_0)$ with a
subgroup of $\Z$, and such a subgroup admits a unique positive generator. By definition this generator is $z_{B_0}$. The positive generators attached to rank 1 parabolic subgroups are by definition the so-called \emph{distinguished braided reflections}.

 The group $B$
obviously acts by conjugation on $\Gamma$, and this action factorizes through $B/Z(B)$.
We have the following
general fact.

\begin{proposition} \label{prop:actionfidele} The action of $B/Z(B)$ on $\Gamma$ is faithful. It is actually already faithful on the
set of vertices of rank $1$.
\end{proposition}
\begin{proof}
Let $b \in B$ acting trivially on the vertices of rank $1$. For any distinguished braided reflection $\sigma$, the
group $\langle \sigma \rangle$ is an irreducible parabolic subgroup of rank $1$, so that
$\langle \sigma^b \rangle = \langle \sigma \rangle \ \simeq \Z$.
But since $\sigma$ and $\sigma^b$ have the same image under the
homomorphism $B \to \Z$ defined above, this proves $\sigma^b = \sigma$.
Since $B$ is generated by its distinguished braided reflections this implies
$ b \in Z(B)$ and this proves the claim.
\end{proof}

The curve complex $\mathcal{K}$ associated to $B$ can then be defined as the collection of finite
sets $\{ B_1,\dots, B_r \}$ of irreducible parabolic subgroups such that each $B_i$ is adjacent to each $B_j$
in $\Gamma$. By the above proposition, the natural conjugacy action of $B/Z(B)$ on $\mathcal{K}$ is faithful, too.

The statement of \autoref{maintheo:curvecomplex} says that the adjacency relation between
two irreducible parabolic subgroups $B_1$ and $B_2$ is detected by the relation $[z_{B_1},z_{B_2}]=1$, that is $z_{B_1}z_{B_2} = z_{B_2} z_{B_1}$. In order to prove this,
we shall later need the following result, which we can prove without using the
classification of complex reflection groups.

\begin{theorem} \label{theo:zzcommW}
 Let $W_1,W_2$ be two irreducible parabolic subgroups of the 2-reflection group $W$, and $B_1,B_2$ two (irreducible) parabolic subgroups of its braid group $B$ with image
$W_1,W_2$. If $[z_{B_1},z_{B_2}] = 1$, then
either $W_1 \subset W_2$, or $W_2 \subset W_1$, or $W_1 \cap W_2 = [W_1,W_2] = \{ 1 \}$
\end{theorem}

This theorem is an immediate consequence of the forthcoming
auxiliary results. The general idea of the proof is that we look at the Taylor expansion of degree
$1$ of the images of the $z_{B_i}$'s in the Hecke algebra representation of $B$ -- which
is defined in the general case by a monodromy construction. This yields the following lemma.

\begin{lemma} Under the assumptions of \autoref{theo:zzcommW},
let $u_i\in \Z W$ denote the sum of the reflections of $W_i< W$ inside the group algebra of $W$. If $[z_{B_1},z_{B_2}] = 1$, then $u_1u_2 = u_2u_1$.
\end{lemma}
\begin{proof}
Let $N_i$ denote the order of $Z(W_i)$. Then $z_{B_i}^{N_i}$ belongs to the pure
braid group of $W$, and we have $[z_{B_1}^{N_1},z_{B_2}^{N_2}] = 1$. Let $X$ be
the hyperplane complement for $W$ and $\mathfrak{g}_X$ its holonomy Lie algebra,
as in \cite{KOHNONAGOYA,KOHNOINVENT}. We refer to \cite{ORLIKTERAO} and to \cite{KRAMCRG} \S 2 for
basic results on it and for its use in the construction of linear representations of the braid group. It is a graded Lie algebra generated by elements $t_H, H \in \mathcal{A}$ of degree $1$, where $\mathcal{A}$ is the hyperplane arrangement of $W$. Let $\omega_H$ be the logarithmic 1-form
associated to $H$, that is $\omega_H = (1/2 \pi \ii) \dlog \alpha_H$ where
$\alpha_H$ is some linear form with kernel $H$. The monodromy of the 1-form $\om = \sum_H t_H \om_H$ provides a morphism from the pure braid group $P = \pi_1(X; x_0)$ to the invertible elements of the completed Hopf algebra
$\widehat{\mathsf{U} \mathfrak{g}_X}$. This morphism can be described via Chen's iterated integrals, as in \cite{CHEN1,CHEN2}. In particular, the image of $[\gamma]$ for $\gamma$ a loop
based at $x_0$ is equal to $1 + \int_{\gamma} \om$ plus terms of higher degree.
Now, $z_{B_i}^{N_i}$ is the homotopy class of a simple loop $\gamma_i$
based at $x_0$ around the reflecting hyperplanes of $W_i$, and therefore it is mapped to
$$
1 + \int_{\gamma} \om + \dots = 1 + \sum_{H} t_H \int_{\gamma} \omega_H + \dots
= 1 + \sum_{H \in \mathcal{A}_i} t_H  + \dots
$$
where $\mathcal{A}_i\subset \mathcal{A}$ is the reflection arrangement for $W_i < W$, and the dots represent
terms of higher degree inside $\widehat{\mathsf{U} \mathfrak{g}_X}$. It follows that  $[z_{B_1}^{N_1},z_{B_2}^{N_2}] = 1$ implies
 $t_1 t_2=t_2 t_1$ where $t_i = \sum_{H \in \mathcal{A}_i} t_H$.

 We then endow the group algebra $\Z W$ of $W$ with the Lie bracket $[a,b]=ab-ba$.
There is a Lie algebra morphism $\mathfrak{g}_X \to \Z W$,
 which has been investigated together with its image in \cite{IH,LIETRANSP,IH2},
 mapping each $t_H$ to the reflection around $H$. Each $t_i$ is obviously mapped
 to the sum $u_i$ of the reflections of $W_i$, and the conclusion follows immediately
 from $[t_1,t_2]=0$.
\end{proof}

Then, we translate this equation holding inside the infinitesimal Hecke algebra of \cite{IH,IH2} into
a commutation relations between some of the reflections involved here.

\begin{lemma} Let $W_1,W_2$ be two parabolic subgroups of the 2-reflection group $W$, and let $u_i\in \Z W$ denote the sum of the set $\mathcal{R}_i$ of the reflections of $W_i$. Then $u_1u_2 = u_2u_1$
if and only if,
for every $s_1 \in \mathcal{R}_1 \setminus \mathcal{R}_2$
and $s_2 \in \mathcal{R}_2 \setminus \mathcal{R}_1$, we have $s_1s_2=s_2s_1$.
\end{lemma}
\begin{proof} For $w \in W$, we set
$K(w) = \Ker(w-1)$. Then $E_i = \bigcap_{s \in \mathcal{R}_i} K(s)$ is the fixed point set of $W_i$ and,
for $w \in W$, $w \in W_i \Leftrightarrow w_{|E_i} = \Id_{E_i}$. We endow again the algebra $\Z W$ with the Lie
bracket $[a,b] = ab-ba$, so by asumption we have $[u_1,u_2]= 0$.

Let $\mathcal{A}$ be the hyperplane arrangement of $W$. We denote $\eps : W \to \{ \pm 1 \}$ the determinant map
mapping each reflection to $-1$, and $W^+ = \Ker \eps$ its rotation subgroup. More generally,
for $G < W$, we set $G^+ = G \cap W^+ = \Ker \eps_{| G}$.
Let $\mathcal{Z}$ denote the collection of codimension 2 subspaces of the form $H \cap H'$ for $H,H' \in \mathcal{A}$, $H \neq H'$ and, for $F \subset E$, let $W_F$ be the pointwise stabilizer of $F$. If $Z \in \mathcal{Z}$ and $w \in W_Z^+ \setminus \{ 1 \}$, we have $K(w) = Z$, as $K(w)$ necessarily has codimension
$1$ or $2$, and if it had codimension $1$ it would be a reflection, contradicting $w \in W^+$.
From this one gets that, if $Z_1,Z_2 \in \mathcal{Z}$ with $Z_1\neq Z_2$, then $W_{Z_1}^+ \cap W_{Z_2}^+ = \{ 1\}$.

For $s,u \in \mathcal{R}$, if $[s,u] \neq 0$ we have in particular $s\neq u$ whence $K(su) = K(s) \cap K(u) \in \mathcal{Z}$, and $su \in W_{K(su)}^+$. From this, it is easily checked that $[u_1,u_2]$ belongs to the linear span of $\{W_Z^+,\ Z \in
\mathcal{Z}\}$, hence to $\bigoplus_{Z \in \mathcal{Z}} \Z W_Z^+$.

Let $s_1 \in \mathcal{R}_1 \setminus \mathcal{R}_2$,
and $s_2 \in \mathcal{R}_2 \setminus \mathcal{R}_1$.
We first consider the projection of
$[u_1,u_2]$ onto $\Z W_Z^+$ for $Z = K(s_1s_2)$. Given $s'_1 \in \mathcal{R}_1$ and $s'_2 \in \mathcal{R}_2$, suppose that the projection of $[s'_1,s'_2]\neq 0$. This is only possible if $K(s'_1s'_2)=Z=K(s_1s_2)$. If $s'_1 \in \mathcal{R}_1\cap\mathcal{R}_2$, then $K(s'_1)$ contains $E_1 + E_2$ and therefore $K(s_1s_2) = K(s'_1s'_2)$ contains $E_2$,
whence $K(s_1) \supset E_2$ contradicting $s_1 \not\in \mathcal{R}_2$. Therefore $s'_1 \in \mathcal{R}_1 \setminus
\mathcal{R}_2$ and similarly one proves $s'_2 \in \mathcal{R}_2 \setminus
\mathcal{R}_1$.

Now, we have $K(s_1) \supset K(s_1s_2) + E_1$, and this inclusion is strict only
if $E_1 \subset K(s_1s_2)$, which implies $E_1 \subset K(s_2)$ and $s_2 \in \mathcal{R}_1$,
a contradiction. Therefore $K(s_1)= K(s_1s_2) + E_1$ and similarly $K(s_2) = K(s_1s_2) + E_2$.
Since $(s'_1,s'_2)$ satisfies the same properties as $(s_1,s_2)$, one gets
$K(s'_1)= K(s'_1s'_2)+E_1= K(s_1s_2)+E_1 = K(s_1)$ and similarly $K(s'_2)=K(s_2)$, so this
proves $s'_1=s_1$ and $s'_2=s_2$. This proves that the projection onto $\Z W_Z^+$ of $[u_1,u_2]$ is equal to $[s_1,s_2]$.

Therefore, if $[u_1,u_2]=0$, this implies that $s_1s_2=s_2s_1$ for every $s_1 \in \mathcal{R}_1 \setminus \mathcal{R}_2$
and $s_2 \in \mathcal{R}_2\setminus \mathcal{R}_1$.

Conversely, suppose that $s_1s_2=s_2s_1$ for every $s_1 \in \mathcal{R}_1 \setminus \mathcal{R}_2$ and $s_2 \in \mathcal{R}_2\setminus \mathcal{R}_1$. Set $\mathcal{R}_0 = \mathcal{R}_1\cap \mathcal{R}_2$ and $\mathcal{R}'_i = \mathcal{R}_i \setminus \mathcal{R}_0$. For short, let $\sum_S$ denote $\sum_{s \in S} s$. Since $\sum_{\mathcal{R}_i}$ belongs to
the center of $\Z W_i$, we have $[s,\sum_{\mathcal{R}_i}]=0$ for every $s \in \mathcal{R}_0$. Therefore,
since we know by hypothesis that $[\sum_{\mathcal{R}'_1},\sum_{\mathcal{R}'_2}]=0$, we get
that $[\sum_{\mathcal{R}_1},\sum_{\mathcal{R}_2}]$ is equal to
$$
[\sum_{\mathcal{R}'_1} +\sum_{\mathcal{R}_0} ,\sum_{\mathcal{R}_2}]=
[\sum_{\mathcal{R}'_1},\sum_{\mathcal{R}_2}]=
[\sum_{\mathcal{R}'_1},\sum_{\mathcal{R}_0}]+[\sum_{\mathcal{R}'_1},\sum_{\mathcal{R}'_2}]=
[\sum_{\mathcal{R}_1} - \sum_{\mathcal{R}_0},\sum_{\mathcal{R}_0}] + [\sum_{\mathcal{R}'_1},\sum_{\mathcal{R}'_2}]=
[\sum_{\mathcal{R}'_1},\sum_{\mathcal{R}'_2}]=0
$$
and this proves the converse statement.

\end{proof}

Finally, in order to use the irreducibility assumptions we have, we use the
following group-theoretic result.
\begin{proposition}\label{prop:complparabsengendre} Let $W < \GL(E)$ be an \emph{irreducible} complex 2-reflection group with set of reflections $\mathcal{R}$,
and $W_0 \subsetneq W$ a nontrivial parabolic subgroup. Then $W$ is generated by $\mathcal{R} \setminus
\mathcal{R}_0$, for $\mathcal{R}_0 = W_0 \cap \mathcal{R}$.
\end{proposition}
\begin{proof}

We assume $W_0 \neq \{ 1 \}$, for otherwise the statement is trivial.
 Let us set $\mathcal{R}' = \mathcal{R}\setminus \mathcal{R}_0$ and $W' = \langle \mathcal{R}'
\rangle$. We first prove that $\bigcap_{s \in \mathcal{R}'} K(s) = \{ 0 \}$, where we denote $K(s) = \Ker(s-1)$.
For this, set $F = \bigcap_{s \in \mathcal{R}'} K(s)$.
Then
$E = F \oplus F^{\perp}$ and, since $w K(s) = K(wsw^{-1})$ we get that $F$ is $W_0$-stable. Since $\mathcal{R}'$
acts trivially on $F$ and $W$ is generated by $\mathcal{R} \subset W_0 \cup \mathcal{R}'$ it follows
that $F$ is $W$-stable. Since $W$ is irreducible this proves $F = \{ 0 \}$.

Now consider $s_0 \in \mathcal{R}_0$. We want to prove $s_0 \in W'$. We have $\{ 0 \} \neq K(s_0)^{\perp}
\not\subset  \bigcap_{s \in \mathcal{R}'} K(s) = \{ 0 \}$ so there exists $s \in \mathcal{R}'$ such
that $K(s_0)^{\perp} \not\subset K(s)$. Let $Z = K(s_0)\cap K(s)$ and $W_Z$ the parabolic subgroup
of $W$ fixing $Z$. It has rank $2$. Assume there exists $s'_0 \in \mathcal{R}_0 \cap W_Z$ with $K(s'_0)\neq K(s_0)$.
Then $Z = K(s_0)\cap K(s'_0)$ and $W_Z \subset W_0$ whence $s \in W_0$, a contradiction. Therefore
$\mathcal{R}_0 \cap W_Z = \{ s_0 \}$. Now, $K(s_0)^{\perp} \not\subset K(s)$ implies $ss_0 \neq s_0 s$,
hence $s s_0 s^{-1} \in (\mathcal{R} \cap W_Z) \setminus \{ s_0 \} \subset \mathcal{R}'$. But then
$s_0 = s^{-1}.s s_0 s^{-1}.s$ is a product of elements of $\mathcal{R}'$ hence belongs
to $\langle \mathcal{R}' \rangle = W'$. It follows that $\mathcal{R}_0 \subset W'$
hence $\mathcal{R} \subset W'$ and $W' = \langle\mathcal{R}\rangle = W$. This concludes the proof.
\end{proof}

The above statement with the same proof holds more generally for an irreducible complex reflection group $W$
having reflections of arbitrary order, and the set of reflections can even be replaced by the set of distinguished
ones, but we shall not need this here.

This statement then proves that our provisional result $u_1u_2=u_2u_1$ actually
means the following, and this completes the proof of \autoref{theo:zzcommW}.

\begin{corollary}
Let $W_1,W_2$ be two irreducible parabolic subgroups of the 2-reflection group $W$, with sets of reflections
$\mathcal{R}_i, i=1,2$. Then the condition that, for every $s_1 \in \mathcal{R}_1 \setminus \mathcal{R}_2$
and $s_2 \in \mathcal{R}_2 \setminus \mathcal{R}_1$, we have $s_1s_2=s_2s_1$, is equivalent to
saying that, either $W_1 \subset W_2$, or $W_2 \subset W_1$, or $W_1 \cap W_2 = [W_1,W_2] = \{ 1 \}$.
\end{corollary}
\begin{proof}
One implication being obvious, we prove the other one.
Assume $s_1s_2 = s_2s_1$ for every $s_1 \in \mathcal{R}_1 \setminus \mathcal{R}_2$
and $s_2 \in \mathcal{R}_2 \setminus \mathcal{R}_1$.
If $\mathcal{R}_1 \subset \mathcal{R}_2$ we have
$W_1\subset W_2$ and similarly $\mathcal{R}_2 \subset \mathcal{R}_1 \Rightarrow W_2 \subset W_1$. So
we can assume that $W_0 = W_1\cap W_2$ is a proper parabolic subgroup of both $W_1$ and $W_2$.
Let $\mathcal{R}_0 = \mathcal{R} \cap W_0$. We have $\mathcal{R}_1 \setminus \mathcal{R}_2 =
\mathcal{R}_1 \setminus \mathcal{R}_0$, so by \autoref{prop:complparabsengendre}
we know that $W_1$ is generated by $\mathcal{R}_1 \setminus \mathcal{R}_2$. It follows that $W_1$ commutes with $\mathcal{R}_2 \setminus \mathcal{R}_1$.
But since $\mathcal{R}_2 \setminus \mathcal{R}_1 = \mathcal{R}_2 \setminus \mathcal{R}_0$ generates
$W_2$ it follows that $[W_1,W_2] = \{ 1 \}$. Now, $W_0 \subset W_2$ commutes
with all $W_1$ ; since it is a parabolic subgroup of $W_1$ this contradicts the irreducibility of $W_1$ unless $W_0 = \{ 1 \}$, and this concludes the proof.
\end{proof}

\section{Description of the parabolic subgroups}

The goal of this section is to describe the parabolic subgroups
of the braid groups of every irreducible complex reflection group up to conjugacy.
The subsections below cover all of them, except for the group $G_{31}$.

\subsection{Real reflection groups and Shephard groups}
\label{sect:descrArtinShephard}

By \autoref{prop:parabolicsdiscri}, Shephard groups have the same collection of parabolic subgroups as their
associated real reflection group, so we can concentrate on the case where
$W$ is a real reflection group. It admits a
Coxeter system $(W,S)$, and its braid group $B$ admits a presentation as an Artin group,
with set of generators $\Sigma$, called the Artin-Brieskorn generators of $B$, which are in natural 1-1 correspondence with $S$. More precisely, picking a Weyl chamber $C \subset X$, and choosing some
base-point $x \in C$, the element of $\Sigma$  in $B = \pi_1(X/W,\bar{x})$ corresponding to $s \in S$ is obtained by joining
$x$ to its image under $s$, viewed as an element of $W$, by a straight line, only modified at
the intersection with the wall $\Ker(s-1)$ so that it makes a positive half turn in the normal complex
direction (see \cite{BRIESKORN}).

If $S_0 \subset S$,
there is a standard parabolic subgroup $W_0 < W$
generated by $S_0$, and the subgroup of $B$ generated by the corresponding copy $\Sigma_0 \subset \Sigma$ is a parabolic subgroup of $B$. Indeed, $C$ is included inside a unique Weyl chamber for $W_0$, denoted $C_0$. Let $L_0$ denote the intersection
of the reflecting hyperplanes for $W_0$, and $L'_0$ the complement of all the other reflecting hyperplanes of $W$
inside $L_0$. Let us choose as normal ray a straight path $\eta$ from $x$ to some point in $L'_0 \cap \overline{C}$.
It is then immediately checked that $\pi_1^{loc}(X/W,\eta)$ can be identified with the Artin subgroup
associated to $S_0$, the Artin-Brieskorn generators of $B$ corresponding to the walls of $C_0$ being mapped to the Artin-Brieskorn generators of the braid group $B_0$ of $W_0$. This proves that every such subgroup $\langle \Sigma_0 \rangle$ is a parabolic subgroup of $B$. We call it a standard parabolic subgroup of $B$.

\begin{proposition} \label{prop:parabArtin}
For a braid group $B$ associated to a real reflection group or a Shephard group,
the parabolic subgroups are the conjugates of the standard parabolic subgroups of $B$, viewed as an Artin group by the above construction. \end{proposition}

\begin{proof}
We can focus on the case of $W$ being a real reflection group.
We already saw that conjugates of standard parabolic subgroups are parabolic subgroups.
Conversely, let $B_0$ be a parabolic subgroup of $B$, and let $W_0 < W$ denote its image under the natural projection $B \to W$. Then $W_0$ is a parabolic subgroup of $W$, hence there exists $g \in W$ such that $g W_0 g^{-1}$ is a standard parabolic subgroup of the chosen Coxeter system $(W,S)$. Let $B_0'$ be the standard parabolic subgroup associated to $g W_0 g^{-1}$. By \autoref{prop:conjparabs}, $B_0$ and $B_0'$ are conjugate in $B$, as their images under the natural projection are conjugate in $W$. Hence $B_0$ is conjugate to a standard parabolic subgroup.\end{proof}

\subsection{The groups $G(de,e,n)$, $d > 1$}

Recall that $G(de,e,n)$ is the group of $n\times n$ monomial matrices with coefficients inside $\mu_{de}(\C)$ such that the product of their nonzero entries belongs to $\mu_d(\C)$, where $\mu_k(\C)$ denotes the group of complex $k$-th roots of $1$.
In this section we set $X_n(r) = \{ (z_1,\dots,z_n) \in \C^n \ | \ z_i \neq
 0, z_i/z_j \not\in \mu_r(\C) \}$. It is the hyperplane complement for all the subgroups
$W = G(de,e,n)$ of the group $\hat{W} = G(r,1,n)$ with $r = de$ and $d>1$. We choose some
basepoint $x \in X_n(r)$ in order to define the braid groups $B$ and $\hat{B}$ of $W$ and $\hat{W}$,
respectively. Since $\hat{W}$ is a Shephard group, the parabolic subgroups of $\hat{B}$ are already known up to conjugacy.
Here we show how to deduce the parabolic subgroups of $B$ from the ones of $\hat{B}$.

  Let $\bar{\varphi} : G(r,1,n) \onto \mu_{r}(\C)/\mu_d(\C)$ denote the composite of the morphism
$G(r,1,n) \onto \mu_{r}(\C)$ given by the product of the (nonzero) monomial entries with the natural
projection morphism. Composed with the projection map $\hat{B} \onto \hat{W}= G(r,1,n)$ it provides a surjective morphism $\varphi : \hat{B} \onto
\mu_{de}(\C)/\mu_d(\C) \simeq \Z/e\Z$ whose kernel is naturally identified with $B$.

\begin{proposition}\label{prop:parabsGdeen} The parabolic subgroups of $B$ are exactly the kernels of the
restriction of $\varphi$ to the parabolic subgroups of $\hat{B}$. Moreover, if $B_0$
is an irreducible parabolic subgroup of $B$ associated to the parabolic subgroup $W_0$, then
$B_0
 = \hat{B}_0 \cap B$ with
$\hat{B}_0$ an irreducible parabolic subgroup of $\hat{B}$ associated to the irreducible parabolic subgroup $\hat{W}_0$
of $\hat{W}$,
and we have $z_{B_0} = z_{\hat{B}_0}^{|Z(\hat{W}_0)|/|Z(W_0)|}$.
\end{proposition}

\begin{proof}
Since the hyperplane arrangements of $W$ and $\hat{W}$ are the same, their parabolic subgroups
are in natural 1-1 correspondance with the elements of the intersection lattice of this hyperplane arrangement. Let $E_0$ be some element of this intersection lattice, and $W_0$, $\hat{W}_0$
their pointwise stabilizer in $W$ and $\hat{W}$, respectively. It is sufficient to prove
that, for $\eta : x \leadsto y_0$ some normal ray  inside $(X_n(r),\C^n)$ and a ball $\Omega$ of center $y_0 \in E_0$
and radius $\epsilon > 0$ with respect to some $\hat{W}$-invariant norm, then the corresponding parabolic subgroups $B_0, \hat{B}_0$ of $B$ and $\hat{B}$ satisfy $B_0 = \hat{B}_0 \cap B$.

We set $\Omega^*=\Omega \cap X_n(r)$. By definition, $B_0$ and $\hat{B}_0$ are then the images of $\pi_1(\Omega^*/W_0,W_0.\eta) \to \pi_1(\Omega^*/W,W.\eta)$ and $\pi_1(\Omega^*/\hat{W}_0,\hat{W}_0.\eta) \to \pi_1(\Omega^*/\hat{W},\hat{W}.\eta)$, respectively.
Because of the covering map properties, the natural maps between the topological spaces involved induce the following commutative diagrams of groups, with the additional property that all commutative squares are cartesian.
\begin{center}
\begin{tikzcd}
\pi_1(\frac{\Omega^*}{W_0},W_0.\eta) \arrow[r, hookrightarrow] \arrow[d, hookrightarrow]
 & \pi_1(\frac{X_n(r)}{W_0},W_0.\eta) \arrow[r, hookrightarrow] \arrow[d, hookrightarrow]
  & \pi_1(\frac{X_n(r)}{W},W.\eta) \arrow[d, hookrightarrow] \arrow[r, "\simeq"] & \pi_1(\frac{X_n(r)}{W},W.x)\arrow[d, hookrightarrow] \\
\pi_1(\frac{\Omega^*}{\hat{W}_0},\hat{W}_0.\eta) \arrow[r, hookrightarrow] &
\pi_1(\frac{X_n(r)}{\hat{W}_0},\hat{W}_0.\eta) \arrow[r] & \pi_1(\frac{X_n(r)}{\hat{W}},\hat{W}.\eta)
\arrow[r, "\simeq"] &  \pi_1(\frac{X_n(r)}{\hat{W}},\hat{W}.x)
\end{tikzcd}
\end{center}
This implies that $B_0$ is identified inside $\hat{B}$ with $\hat{B}_0 \cap B = \hat{B}_0 \cap \Ker \varphi$, which proves the first claim. Finally, if $W_0$ is irreducible then $\hat{W}_0$ is also irreducible, as
it acts on the same space and contains $W_0$. We choose some $y_0+v \in \Omega^* \cap \eta(]1-\alpha,1[)$ for $\alpha$ small enough, and naturally identify $B_0$ and $\hat{B}_0$ with fundamental groups based at $W_0.(y_0+v)$
and $\hat{W}_0.(y_0+v)$, respectively. Then, the paths $\gamma : t \mapsto y_0 + v \exp(2\pi \ii t/|Z(W_0)|)$
and $\hat{\gamma} : t \mapsto y_0 + v \exp(2\pi \ii t/|Z(\hat{W}_0)|)$ provide loops $t \mapsto W_0.\gamma(t)$
and $t \mapsto \hat{W}_0.\gamma(t)$ whose homotopy classes are $z_{B_0}$ and $z_{\hat{B}_0}$, respectively.
It is then immediately checked that $t \mapsto \hat{W}_0.\gamma(t)$ has for homotopy class $z_{\hat{B}_0}^{|Z(\hat{W}_0)|/|Z(W_0)|}$, which proves the last claim.
\end{proof}

\subsection{Interval monoids}

The presentations we are going to use are strongly connected with monoids
with strong properties. These turn out to be \emph{interval monoids}, so we define this
concept first.

One first considers a finite group, which in our case will always be the reflection group $W$, and
some generating set $S$ for $W$.

The \emph{length with respect to $S$} of $w \in W$ is defined as the minimal $\ell_S(w) = r \geq 0$ so that one can write $w = s_1\dots s_r$ where $s_i \in S$, with $\ell_S(1) = 0$. An expression $(s_1,\dots,s_r) \in S^r$ such that $w=s_1s_2\dots s_r$ with $r = \ell_S(w)$ is called a reduced expression and the set of such reduced expressions is denoted $Red(w)$.

From this one can define two partial orderings on $W$, setting $a \prec b$ if $\ell_S(a) + \ell_S(a^{-1} b) = \ell(b)$
and $b \succ a$ if $\ell_S(ba^{-1}) + \ell_S(a) = \ell_S(b)$.
An element $c \in W$ is then said to be \emph{balanced} if $[1,c] = \{ a \in W; 1 \prec a \prec c \}  =
\{ a \in W; c \succ a \succ 1 \}$, that is if
$\{ u \in W; \ell_S(u) + \ell_S(u^{-1}c) = \ell_S(c) \}
= \{ u \in W; \ell_S(cu^{-1}) + \ell_S(u) = \ell_S(c) \}$. The \emph{interval monoid} $M$ attached to the
data $(W,S,c)$ is defined by
taking for generators a copy of $[1,c]$,
denoting $\mathbf{u}$ the generator corresponding to $u \in [1,c]$, and for relations
$$
\mathbf{w} = \mathbf{u}\mathbf{v} \ \mbox{when} \ w = uv, \ell_S(w) = \ell_S(u) + \ell_S(v)
$$

Such monoids with their main properties have been basically introduced in \cite{MICHEL} and we refer to \cite{DGKM} for
a modern treatment of the subject.
A general important result is that, if the partial orderings $\prec$ and $\succ$ restricted to the set $[1,c]$
are \emph{lattices}, then $M$ together with $\mathbf{c}$ provides a Garside structure on the
group $G$ defined by the same presentation as $M$, where the concept of a Garside structure will
be recalled in detail in \autoref{sect:parabclosuresgarsidegroups} below.

Another important property is that such an interval monoid is homogeneous, and more precisely there
exists a monoid morphism $\ell : M \to \N = \Z_{\geq 0}$ such that $M$ is generated by its elements
of degree $1$ -- these are the \emph{atoms} of the monoid $M$, and they are exactly the $\mathbf{s}$ for $s \in S
\cap [1,c]$. Then, another general property is that, for $w \in [1,c]$, one has $\ell(\mathbf{w}) = \ell_S(w)$.

The archetypical example of interval monoids providing Garside structures is given by
the Artin monoids associated to Artin groups: there, the set $S$ is given by
the simple reflections of the Coxeter presentation of the real reflection group. The complex braid groups for
 $G(e,e,n)$ and for many exceptional groups can also be described as the group of fractions of
 interval monoids.

In our case, the set $S$ will contain only elements of order $2$. This has the following consequence,
which will be of crucial importance in the proof of our main theorems (see \autoref{sect:intersectparabs}).

\begin{proposition} \label{prop:intervalissquarefree} If all the elements of $S$ have order $2$ then,
For every $u \in [1,c]$, $\mathbf{u} \in M$ is squarefree, that is no expression as a product of
atoms may contain $\mathbf{a}\mathbf{a}$ for some $a \in S$.
\end{proposition}
\begin{proof}
This is the consequence of the fact that, for $u \in [1,c]$, if $\mathbf{u}=\mathbf{a_1a_2\dots a_r}$
with $a_i \in S$ and $r = \ell(\mathbf{u}) = \ell_S(u)$, then $a_1a_2\dots a_r$ has to be a reduced
decomposition, which excludes the possibility $a_i a_{i+1} = a a = 1$ for some $i$, as every element of
$S$ has order $2$.
\end{proof}

\subsection{Groups $G(e,e,n)$}\label{sect:G(e,e,n)}

Let $W = G(e,e,n)$, with $e \geq 1$. The group $B$ admits the following ('standard') presentation
with generating set
$S = \{ t_0,t_1,\dots,t_{e-1}, s_3, s_4, \dots, s_n \}$ and relations
\begin{enumerate}
\item $t_i t_{i+1} = t_j t_{j+1}$, with the convention $t_e = t_0$,
\item $s_3 t_i s_3 = t_i s_3 t_i$
\item $s_k t_i  = t_i s_k$ for $k \geq 4$
\item $s_k s_{k+1} s_k = s_{k+1} s_k s_{k+1} $ for $k \geq 3$
\item $s_k s_l = s_l s_k$ when $| l-k | \geq 2$.
\end{enumerate}
We refer to \cite{CORPIC,CALMAR,NEAIME} for general results on this presentation and the corresponding monoid, which was first introduced
by Corran and Picantin in \cite{CORPIC}. We call it the \emph{standard} monoid for $G(e,e,n)$.
 When $e=1$ (resp. $e=2$) one recovers the Artin monoid of type $A_{n-1}$ (resp. $D_n$). For
products of reflection groups of these types, we take for presentation of $B$ the obvious direct product presentations. This covers in particular the case of the reflection subgroups of $W$ of the form
$G(e,e,n_1)\times G(1,1,n_2)\times \dots \times G(1,1,n_k) \subset W$ with $n_1+n_2+\dots+n_k= n$,
which are actually parabolic subgroups of $W$.

We recall that the reflecting hyperplanes for $W$ are given by the equations $z_i = \zeta z_j$
for $i \neq j$ and $\zeta \in \mu_e(\C)$. We also
recall from \cite{CORPIC,CALMAR,NEAIME} that, under the natural morphism $B \onto W$, the action
on $\C^n$ of the above generators of $B$ is given
by the permutation matrices $(k-1,k)$ for $s_{k}$ and, for $t_k$, by
$$
\begin{pmatrix}
0 & \zeta_e^k \\
\zeta_e^{-k} & 0
\end{pmatrix} \oplus \Id_{n-2}
 \ \mbox{ where }\zeta_e = \exp(2 \ii \pi/e).
$$
A nontrivial important result from \cite{NEAIME} is the following

\begin{theorem}\label{theo:neaime} (Neaime)
The standard monoid for $G(e,e,n)$ is an interval monoid with respect to the image of $S$ under $B \onto W$.
\end{theorem}

The basepoint corresponding to this presentation is $\underline{b} = (b_1,\dots,b_n) \in \R^n\subset \C^n$ with $0<b_1 < b_2 < \dots < b_n$, when needed with the condition that $b_{i+1}/b_i$ is big enough. Notice that the set of such basepoints is convex and therefore simply connected and also that, when $e \in \{1,2 \}$, it lies
inside a natural Weyl chamber.

It was proven in \cite{CALMAR} Proposition 6.2 that, with respect to such a basepoint, the generators we have correspond to local fundamental groups as follows. For $s \in S$, consider the straight line segment $[\underline{b},s.\underline{b}]$, where $s.\underline{b}$ is given by the action of $B$ on $\C^n$ described above. It crosses only one reflecting hyperplane,
at its middle point $\beta(s) = (\underline{b}+s.\underline{b})/2$. Then
$s$ is a generator of the local fundamental group for $X/W$ w.r.t. the normal ray $[\underline{b},\beta(s)]$.
The same property holds, with the same proof, if $W$ is replaced by one of its reflection subgroups of the form
$G(e,e,k)\times G(1,1,n-k)$ for $0 \leq k < n$. Actually these subgroups are
 maximal parabolic subgroups of $G(e,e,n)$. More
 generally, as a consequence of \cite{TAYLORREFL} Theorem 3.11,
every maximal parabolic subgroup of $W$ is a conjugate
of the stabilizer of either $\underline{b}_0(\zeta) = (\zeta,1,\dots,1)$ for $\zeta \in \mu_e(\C)$
or of
$$\underline{a}_k = (0,\dots,0,\underbrace{1,\dots,1}_{n-k}) \mbox{ for } 1 \leq k \leq n-1.
$$
The former are isomorphic to $\mathfrak{S}_n=G(1,1,n)$ while the latter are isomorphic
to $G(e,e,k)\times G(1,1,n-k)$ for $1 \leq k < n$. Notice that, since $b_{k+1}>0$, the stabilizer of $\underline{a}_k$ is equal to
the stabilizer of $b_{k+1} \underline{a}_k = (0,\dots,0,b_{k+1},\dots,b_{k+1})$.

From this we deduce
the following\footnote{Part of this statement was already
claimed in \cite{CALMAR} Prop. 6.3, but the proof was incomplete, because of the erroneous statement given there that
the stabilizer of $(\zeta , 1, 1, \dots , 1)$ is
conjugate inside $W$ to the stabilizer of $(1, 1, \dots , 1)$.}

\begin{proposition} \label{prop:parabsCPmax}
Let $W = G(e,e,n)$, and $S_0 \subsetneq S =\{ s_k, t_i, 3 \leq k \leq n,0 \leq i < e \}$ be such that $\{ i ; t_i \in S_0 \}$
has cardinality $1$ or $e$, and which is maximal for this property. Then $\langle S_0 \rangle \subset B$
is a maximal parabolic subgroup, and every maximal parabolic subgroup of $B$ is a conjugate of such a subgroup.
\end{proposition}

\begin{proof}

By the description of the maximal parabolic subgroups given above and \autoref{prop:conjparabs} it is sufficient to prove that, for every $\underline{b}_0$ equal to $\underline{b}_0(\zeta)$
or $b_{k+1}\underline{a}_k$ as above,
 there exists a normal ray $\eta : \underline{b} \leadsto \underline{b}_0$ such that the
parabolic subgroup $W_0$ stabilizing $\underline{b}_0$ is generated by the subset $S_0 \subset S$ given by $\{ s \in S; s.\underline{b}_0 = \underline{b}_0 \}$. Indeed, it is readily checked that these collections $S_0$
are the ones of the statement.

Under our assumptions, up to possibly increasing the value of the ratios $b_{i+1}/b_i$, we can
assume that the straight line segment $[\underline{b},\underline{b}_0]$ crosses the hyperplane arrangement only at $\underline{b}_0$, and we take $\eta = [\underline{b},\underline{b}_0]$ for normal ray.

We prove that $\langle S_0 \rangle$ is the corresponding parabolic subgroup by completing the following
commutative diagram, whose plain arrows are the natural ones and $X_0$ is the hyperplane complement of $W_0$.
\begin{center}
\begin{tikzcd}
\pi_1(\frac{X}{W},W.\eta) & \pi_1^{loc}(\frac{X}{W},W.\eta) \arrow[l, hookrightarrow] \arrow[dr, "\simeq"] &  \\
 & \langle S_0 \rangle \arrow[u, dashed ] \arrow[r, dashed, "\simeq"] & \pi_1(\frac{X_0}{W_0},W_0.\eta)
\end{tikzcd}
\end{center}
We have a natural embedding
$$\langle S_0 \rangle \subset B = \pi_1(X/W,W.\underline{b})= \pi_1(X/W,W.\eta(0)) \simeq
\pi_1(X/W,W.\eta),
$$
which provides our first 'dashed' arrow $\langle S_0 \rangle \into \pi_1^{loc}(\frac{X}{W},W.\eta)$. The second one is obtained as its composite with the natural
map $\pi_1^{loc}(\frac{X}{W},W.\eta)\to  \pi_1(\frac{X_0}{W_0},W_0.\eta)$. It remains to prove
that this composite map is surjective.

When $\underline{b}_0 = b_{k+1} \underline{a}_k$ or when $\underline{b}_0 = \underline{b}_0(1)$
this is immediate because
in these two cases the elements of $S_0$ are mapped to the generators
of the corresponding braid group
of $\pi_1(\frac{X_0}{W_0},W_0.\eta)$.
In case $\underline{b}_0=\underline{b}_0(\zeta)$
with $\zeta\neq 1$, we consider the map $p : (z_1,z_2,\dots,z_n) \mapsto (\zeta^{-1}z_1,z_2,\dots, z_n)$. It maps $X_0$ to the hyperplane complement $X'_0$
associated to
$W'_0 = \mathfrak{S}_n$,
and we have $p(\underline{b}_0) =\underline{1}=(1,\dots,1)$.
Without loss of generality we can replace the path $\eta= [\underline{b},\underline{b}_0]$ by some basepoint
inside $[\underline{b},\underline{b}_0]$ close enough to $\underline{b}_0$, and thus the path $p\circ \eta$ by
$$\underline{c}_0(\eps)
= p(\underline{b}_0 + \eps(\underline{b}-\underline{b}_0))
=\underline{1}+\eps(\zeta^{-1}b_1-1,b_2-1,\dots,b_n-1)
=(1-\eps)\underline{1}+\eps(\zeta^{-1}b_1,b_2,\dots,b_n)
$$ for $\eps > 0$ small enough.
We want to prove that $\pi_1(X'_0/W'_0,W'_0.\underline{c}_0(\eps))$
is generated by the images of $S_0$. For fixed $k \in \{1,\dots,n-1\}$, we introduce $\underline{d}_0(\eps) =
(1-\eps)\underline{1}+\eps\underline{b}$, $\underline{c}_0(\eps,u) =
(1-u)\underline{c}_0(\eps)+u\underline{d}_0(\eps)$
and, for $\alpha>0$ small enough,
$$
A(t,u)=(1-t) \underline{c}_0(\eps,u)+t \sigma_k .\underline{c}_0(\eps,u) +\alpha  \ii t(1-t)(\sigma_k-1).\underline{c}_0(\eps,u)
$$
where $\sigma_k = (k,k+1) \in \mathfrak{S}_n$.
The images of $S_0$ are the homotopy classes of the maps $t \mapsto A(t,0)$ for $t \in [0,1]$.
The maps $t \mapsto A(t,1)$ provide homotopy classes inside $\pi_1(X'_0/W'_0,W'_0.\underline{d}_0(\eps))$ which are the Artin generators of the braid group of $W'_0 \simeq \mathfrak{S}_n$ ; indeed, it
is readily checked that, provided $b_1 > 1$, $\underline{d}_0(\eps)$ belongs to the Weyl chamber
associated to the Coxeter system $\{ \sigma_1,\dots,\sigma_{n-1} \}$. In order to get that
$\pi_1(X'_0/W'_0,W'_0.\underline{c}_0(\eps))$
is generated by the images of $S_0$ it is then sufficient to check that the map $A$ is
a homotopy map with values in $X'_0$.

We have that $A(0,\bullet)$ describes the line segment $[\underline{c}_0(\eps),
\underline{d}_0(\eps)]$, and $A(1,u) =\sigma_k.A(0,u)$, so we only need to prove that
$A(t,u) \in X'_0$ for every $(t,u) \in [0,1]^2$.
Writing $A(t,u) = (1-\eps)\underline{1}+\eps \check{A}(t,u)$,
it is equivalent to proving that $\check{A}(t,u) \in X'_0$.
Now we compute that
$$
\underline{c}_0(\eps,u)
=(1-\eps)\underline{1}+\eps \check{c}_0(\eps,u),\ \
\underline{\check{c}}_0(\eps,u) =
\left(
((1-u)\zeta^{-1}+u)b_1,b_2,\dots,b_n\right)
$$
and $\check{A}(t,u) =
(1-t) \underline{\check{c}}_0(\eps,u)+t \sigma_k .\underline{\check{c}}_0(\eps,u) +\alpha  \ii t(1-t)(\sigma_k-1).\underline{\check{c}}_0(\eps,u)$.
Let $\beta_i$ be the linear
form $\underline{z} \mapsto z_{i+1}-z_i$.
If $i \geq 2$, then the real part of $\beta_i(\check{A}(t,u))$ is positive, except
when $i=k$ ; in this case, it is still nonzero, except for $t = 1/2$, but then
its imaginary part is nonzero.
So it only remains to consider the case $i =1$.
For $k\geq 2$, the real part of $\beta_1(\check{A}(t,u))$ is positive,
so we can assume $k=1$.
We have
$$
\beta_1(\check{A}(t,u))=
\left( (1-2t)-2\alpha \ii t(1-t) \right)
\beta_1(\underline{\check{c}}_0(\eps,u))
$$
and we have that $\beta_1(\underline{c}_0(\eps,u))$ has positive real part for every $u \in [0,1]$
so that $\beta_1(\check{A}(t,u)) \neq 0$ for every $t,u$, which concludes the proof.
\end{proof}

Using the compatibility of parabolic subgroups with products established in \autoref{sect:subgroupsproducts},
this statement is immediately extended to the case of an arbitrary parabolic subgroup of $G(e,e,n)$.
 From this one gets a complete description of the
parabolics as follows.

\begin{corollary} \label{cor:parabsCP}
Let $W = G(e,e,n)$, and $S_0 \subsetneq S =\{ s_k, t_i, 3 \leq k \leq n,0 \leq i < e \}$ be such that $\{ i ; t_i \in S_0 \}$
has cardinality $0$, $1$ or $e$. Then $\langle S_0 \rangle \subset B$
is a parabolic subgroup, that we call a standard parabolic subgroup, and every parabolic subgroup of $B$ is a conjugate of such a subgroup.
\end{corollary}
\begin{proof} We prove that $\langle S_0 \rangle \subset B$ is a parabolic subgroup by descending induction on $|S_0|$. If $S_0$ is maximal for this property, we have the conclusion by \autoref{prop:parabsCPmax}.
Otherwise, we have $S_0 \subsetneq S_1 \subsetneq S$ with $\{ i ; t_i \in S_1 \}$
of cardinality $0$, $1$ or $e$, and we can assume that $S_1$ is minimal for this property. Then, by the induction
assumption, $\langle S_1 \rangle \subset B$ is a parabolic subgroup $B_1$, attached to the parabolic subgroup $W_1$ of $W$. On the other hand, $S_1$ can be identified with the generators of the standard presentation for the corresponding
braid group of $B_1$, and $S_0 \subsetneq S_1$ is maximal for the property that $\{ i ; t_i \in S_0\}$ has cardinality $0$, $1$ or $e$. From \autoref{prop:parabsCPmax} it then follows that $\langle S_0 \rangle \subset B_1 \subset B$ is a parabolic subgroup of $B_1$, and therefore of $B$ by \autoref{prop:basicpropsparabs} (4).

Conversely, assume that $B_0$ is a parabolic subgroup of $B$ mapped to the parabolic subgroup $W_0$ of $W$. Then it is included inside a maximal parabolic subgroup $B_1$, with corresponding maximal parabolic subgroup $W_1$. By \autoref{prop:parabsCPmax} we can assume that $B_1 = \langle S_1 \rangle$ with
$S_1 \subsetneq S$ maximal such that $\{ i ; t_i \in S_1 \}$
has cardinality $0$, $1$ or $e$. Then $B_0$ is a parabolic subgroup of $B_1$ by \autoref{prop:parabsherite}, and by induction on the rank of $W$ one gets that $B_0$ is a conjugate inside $B_1$ (hence inside $B$) of $\langle S_0 \rangle$, where $\{ i ; t_i \in S_0 \}$
has cardinality $0$, $1$ or $e$.
\end{proof}

\subsection{Well-generated 2-reflection groups}
\label{sect:wellgen2refl}

Our main references are \cite{BESSIS} and
\cite{RIPOLL} -- see also \cite{THEO}. Let $W<\GL_n(\C)$ be a well-generated complex reflection group, that is a group that can be generated by at most $n$ reflections. All the exceptional groups of rank at least $3$ fall into this category except $G_{31}$. We denote $\mathcal{R}$ the set of all reflections, and assume that all of them have order 2. This is also the case for all well-generated irreducible complex reflection groups, except for some of them which are Shephard groups, so that we already know their parabolic subgroups.

By definition the Coxeter number $h = h(W)$ of $W$ is its largest reflection degree. We have $h =2 |\mathcal{R}|/n$. We construct an interval monoid by considering the generating set $\mathcal{R}$. When $W$ is irreducible, a \emph{Coxeter element} for $W$ is a regular element for the eigenvalue $\exp(2 \pi \ii/h)$, in the sense of Springer. When it is not, this notion is also defined, by taking the direct sum of Coxeter elements for the irreducible constituents of $W$. We refer to \cite{BESSIS}, Definition 7.1, for more details. Let $c$ be such a Coxeter element for $W$. Such an element is balanced, and then the `dual braid monoid' $M(c) = M_W(c)$ in the sense of \cite{BESSIS} is defined as the interval monoid attached to $(W,\mathcal R, c)$.
We denote $G(c)$ the group with the same presentation.

The following two propositions are proven in \cite{BESSIS} and in the references there.
\begin{proposition}  \ 
\begin{enumerate}
\item If $W$ is essential (that is $\bigcap_{w \in W} \Ker(w- \Id) = \{ 0 \}$), then $\ell(c)= n$.
\item The poset $[1,c]$ is a lattice.
\end{enumerate}
\end{proposition}

\begin{proposition} \
\label{prop:dualmonoidpropbase2}
\begin{enumerate}
\item $M(c)$ is a Garside monoid, with fundamental element $\mathbf{c}$, and
$u \mapsto \mathbf{u}$ is an isomorphism of lattices from $[1,c]$ to
the set of divisors of $\mathbf{c}$, for the partial ordering induced by left divisibility.
\item Let $\mathcal{R}_c = \mathcal{R} \cap [1,c]$. Then
$M(c)$ has a presentation with generators $\mathbf{u}, u \in \mathcal{R}_c$
and relations $\mathbf{u} \mathbf{v} = \mathbf{v}\mathbf{w}$
whenever $u,v, w = v^{-1} u v \in \mathcal{R}_c$. These generators are the
atoms of $M(c)$.
\item For $u \in [1,c]$, $\ell(u)$ is equal to the length of $\mathbf{u}$ with respect to the atoms of $M(c)$.
\item The Garside group $G(c)$ associated to $M(c)$ is isomorphic to the braid
group $B$ of $W$, in such a way that each atom $\mathbf{u}, u \in \mathcal{R}_c$
is sent to a distinguished braided reflection associated to $u$.
\end{enumerate}
\end{proposition}

From this we get immediately the following. The first two items are classical and the last one is immediate.
\begin{proposition}\label{prop:dual_simple} \
\begin{enumerate}
\item The length with the respect to the atoms of $M(c)$ defines a monoid morphism $\ell : M(c) \to (\N, +)$,
that is $\ell(m_1m_2) = \ell(m_1)+\ell(m_2)$ for every $m_1,m_2 \in M(c)$.
\item For every $u \in [1,c]$, $\mathbf{u} \in M(c)$ is balanced, that is its set of left and right divisors are
the same.
\item For every $u \in [1,c]$, $\mathbf{u} \in M(c)$ is squarefree, that is no expression as a product of
atoms may contain $\mathbf{a}\mathbf{a}$ for some $a \in \mathcal{R}_c$.
\end{enumerate}
\end{proposition}
\begin{proof} 
(1) is immediate for instance from the homogeneity of the presentation. For (2), if $a \prec u$
then $ab = u$ and with $\ell(a)+\ell(b)= \ell(u)$. But since $\mathcal{R}$
is a union of conjugacy classes we have $\ell(b) = \ell(aba^{-1})$, hence $(aba^{-1})a = u$ with $\ell(aba^{-1}) + \ell(a) = \ell(u)$, so $a$ divides $u$ on the right too, and $\mathbf{u}$ is balanced.
(3) is a special case of \autoref{prop:intervalissquarefree}.
\end{proof}

In the sequel we shall need the following classical lemma. We denote by $\pi : M(c) \to W$
the natural monoid morphism mapping each $\mathbf{u}$ to $u \in [1,c]$, and we denote by $j: [1,c] \to M(c)$ the map sending $u$ to $\mathbf u$.

\begin{lemma}
\label{lem:caractsimplesintervalles}
 For $m \in M(c)$ with $\pi(m) \in [1,c]$, $m$ divides $\mathbf{c}$ if and only if $\ell(m) = \ell(\pi(m))$,
where $\ell(m)$ is the length of $m$ with respect to the atoms of $M(c)$.
\end{lemma}
\begin{proof} By the previous proposition we have $\ell(m) = \ell(\pi(m))$ when $m$ divides $\mathbf{c}$,
so we only need to prove the converse. Notice that one always has $\ell(m) \geq \ell(\pi(m))$,
and assume $m \in M(c)$ satisfies $\ell(m) = \ell(\pi(m))$. Write $m$ as $m_1m_2$ with $m_1,m_2 \in M(c)$,
and notice that
$$
\ell(m_1)+\ell(m_2)  = \ell(m) = \ell(\pi(m))=\ell(\pi(m_1)\pi(m_2)) \leq \ell(\pi(m_1))+\ell(\pi(m_2)) \leq \ell(m_1)+ \ell(m_2)
$$
so that we have $\ell(m_i) = \ell(\pi(m_i))$ for $i=1,2$. Then by induction on $\ell(m)$ one can assume $m=m_1m_2$ with $m_1$ diving $\mathbf{c}$
and $m_2$ an atom, that is $m_2 = \mathbf{s}$ for $s=\pi(m_2) \in [1,c]$ a reflection. But then
$\ell(\pi(m)) = \ell(m) = \ell(m_1) + 1 = \ell(\pi(m_1)) + \ell(s)$. But from the very definition
of $M(c)$ this implies that
 $j(\pi(m)) = m_1\mathbf{s}$, that is
  $j(\pi(m)) =m$.
This proves that $m$ divides $\mathbf{c}$ by \autoref{prop:dualmonoidpropbase2} (1).  

\end{proof}

In addition, we will need the following proposition, which combines results from
\cite{BESSIS} and \cite{RIPOLL}.
\begin{proposition} \label{prop:parabsreflwellgens}\ 
\begin{enumerate}
\item Every parabolic subgroup of $W$ is well-generated.
\item Let $W_0$ be a parabolic subgroup of $W$, and $c_0$ a Coxeter element of $W_0$.
Then $W_0$ is the pointwise stabilizer of $\Ker(c_0 - \Id)$, and there
exists $g \in W$ such that $gc_0g^{-1} \in [1,c]$. 
Moreover, the length inside the interval $[1,c_0]$ of $W_0$ is the same when computed with respect to $\mathcal{R}$ and with respect fo $\mathcal{R}_0 = 
  \mathcal{R} \cap W_0$.
\item Let $c_0 \in [1,c]$. Then $c_0$ is a Coxeter element for the pointwise stabilizer $W_0$ of $\Ker(c_0-\Id)$. Moreover, if $(s_1,\dots,s_k) \in Red(c_0)$, then $\langle s_1,\dots, s_k \rangle = W_0$.
\end{enumerate}
\end{proposition}
\begin{proof}
For (1), this is a direct check on the tables, for instances the ones of Appendix C in \cite{ORLIKTERAO}.
 For (2) and (3), see
Proposition 1.36 of \cite{RIPOLL} and its proof.
\end{proof}

An immediate consequence of this proposition is that the inclusion $[1,c_0] \subset [1,c]$ is an injective
morphism of posets.
Moreover,
we have a well-defined
homomorphism $\varphi : M(c_0) \to M(c)$ extending the inclusion map $[1,c_0] \subset [1,c]$,
for any $c_0 \in [1,c]$ and $W_0$ the corresponding parabolic subgroup.

For the sequel, we shall need the following additional property.
\begin{equation}
\label{eq:hypdual} \forall c_0 \in [1,c] \ \ \ W_0 \cap [1,c] = [1,c_0]
\end{equation}

We were unable to find a reference for it -- nor a proof avoiding the classification of complex reflection groups.
For the groups we are interested in it is easily checked
by computer -- for this notice that, since all Coxeter elements are conjugates, it is sufficient to
check this for one of them. Then, for each $c_0 \in [1,c]$, it is sufficient to get a decomposition of $c_0$
in order to determine $W_0$ as the group generated by the reflections it contains. However, it actually holds true in general, as shows the following proposition. The main idea of using \cite{LEWISMORALES} in order to prove the non-real case was communicated to us by Theodosios Douvropoulos. We give his proof here.

\begin{proposition}
\label{prop:prophypdual} Property (\ref{eq:hypdual}) holds true.
\end{proposition}
\begin{proof}
We endow the ambient space $V$ with an hermitian structure for which $W$ acts by unitary transformations,
that is $W \subset U(V)$.
We use the Brady-Watt order defined on the unitary group by $a \prec_{BW} b$ if $\dim \codim \Ker(a-1) +
\dim \codim \Ker (a^{-1}b-1) = \dim \codim \Ker(b-1)$, or equivalently $\Imm(a-1)\oplus \Imm(a^{-1}b-1)
= \Imm (b-1)$, defined in \cite{BRADYWATT}. We denote $[a,b]_{BW}$ the set of all unitary transformations $u$
such that $a \prec_{BW} u$ and $u\prec_{BW} b$.
 By Theorem 2 of \cite{BRADYWATT} (see also Note 1 in \cite{BRADYWATT})
we have that $[1,u]_{BW}$ is isomorphic as a poset to the set of subspaces of $V$ containing $\Ker(u-1)$
ordered by reverse inclusion, the
bijection being given by $u' \mapsto \Ker(u-1)$.

On the other hand, it is proved in \cite{LEWISMORALES} that $[1,c] = W \cap [1,c]_{BW}$
for every $W$ -- the real case being a classical result, the general case
being proved using the classification, see Corollary 6.6 of \cite{LEWISMORALES}.

Now, let $w_0 \in [1,c] \cap W_0$. Then $w_0 \in W_0$ implies that $\Ker(w_0-1)$ contains $\bigcap_{w \in W_0} \Ker(w-1)$, which is equal to $\Ker(c_0-1)$.
Since $w_0,c_0 \in [1,c]_{BW}$, by the theorem of Brady-Watt this implies $w_0 \prec_{BW} c_0$.
But then $w_0 \in W_0 \cap [1,c_0]_{W_0} = [1,c_0]$, and this proves the claim.
\end{proof}

Let $\mathcal{R}_0 = \mathcal{R} \cap W_0$. We have $\mathcal{R}_{c_0} = \mathcal{R}_0 \cap [1,c_0] = \mathcal{R}_0 \cap [1,c] =\mathcal{R} \cap W_0 \cap [1,c] =
\mathcal{R} \cap [1,c_0]$.
For $s,t \in \mathcal{R}_{c_0}$ the right lcm of $\mathbf{s},\mathbf{t}$ is equal
to $\mathbf{w}$ for $w$ the join of $s,t$ in the corresponding poset $[1,c_0]$.
Since it is also equal to the join of $s,t$ inside $[1,c]$ we get
that $\lcm(\varphi(\mathbf{s}),\varphi(\mathbf{t})) = \varphi(\lcm(\mathbf{s},\mathbf{t}))$.

\medskip

From this, general arguments (see \cite{CALMAR} Lemmas 5.1, 5.2) imply
that $\varphi$ is injective, so that $M(c_0)$ is identified with a submonoid
of $M(c)$. Moreover, if $m' \in M(c)$ divides (on the left) some $m \in M(c_0)$, then
$m' \in M(c_0)$ ; indeed, if $m = m'm''$, there should be a sequence of Hurwitz relations
applied on some decompositions of $m$ in atoms of $M(c_0)$ so that it is changed to some
decomposition involving some atoms of $M(c)$ which do not belong to $M(c_0)$. But for Hurwitz
relations $\mathbf{u}\mathbf{v} = \mathbf{v}\mathbf{w}$, if $u,v \in W_0$ then $v,w \in W_0$,
so this is not possible. This property then implies by general arguments (\cite{CALMAR} Lemma 5.3) the
following.

\begin{proposition}
\label{prop:dualinjectionmonoid}
Let $c_0 \in [1,c]$.
\begin{enumerate}
\item The map $\varphi : M(c_0) \to M(c)$ is injective and induces an injective group homomorphism $G(c_0) \to G(c)$, identifying $G(c_0)$ with a subgroup of $G(c)$.
\item For every $a,b \in M(c_0) \ \lcm(\varphi(a),\varphi(b)) = \varphi(\lcm(a,b))$.
\item The image of $M(c_0)$ inside $M(c) \subset G(c)$ is equal to $G(c_0) \cap M(c)$.
\end{enumerate}
\end{proposition}

We then have the following.
\begin{proposition}\label{prop:dualstandardparabs} For $c_0 \in [1,c]$ the image of the injective homomorphism $G(c_0) \to G(c) = B$  is a parabolic subgroup of $B$, that we call a standard parabolic subgroup.
Moreover, every parabolic subgroup of $B$ is conjugate to such a standard parabolic subgroup.
\end{proposition}
\begin{proof}
The first part of the proposition is a reformulation of  Proposition 20 and Proposition 39 of \cite{THEO}.
The second one then follows from the combination of
\autoref{prop:parabsreflwellgens} and \autoref{prop:conjparabs}.
\end{proof}

Of course, this concept of standard parabolic subgroup here depends on the choice of
the Coxeter element $c$, although we know that the concept of parabolic subgroup does not.

\subsection{Exceptional groups of rank 2}
\label{sect:descrparabsrank2}

Most exceptional groups of rank $2$ are Shephard groups, or isodiscriminantal to
a group of the form $G(de,e,n)$. The only ones which are not are $G_{12}$, $G_{13}$ and $G_{22}$.

If $W =G_{12},G_{22}$, we choose the presentation of $B$ given in \cite{BMR}, that is
$\langle s,t,u \ | \ stus=tust=ustu \rangle$ and $\langle s,t,u \ | \ stust=tustu=ustus \rangle$.
In both cases all reflections have order $2$ and they form a conjugacy class. Since the
generators $s,t,u$ are distinguished braided reflections this proves that every proper parabolic subgroup of $B$
is a conjugate, say, of $\langle s \rangle$.

For $W = G_{13}$, we use the presentation of $B$ obtained in \cite{BANNAI}. By the construction
of \cite{BANNAI}, the generators are distinguished braided reflections, as they are meridians with respect to irreducible components of the discriminant, which can be written as $z_1(z_1^2 - z_2^3)=0$. Since there are two such irreducible components, every distinguished braided reflection is conjugate to one of the generators.
These generators are $g_1,g_2,g_3$ with relations
$g_1g_2g_3g_1=g_3g_1g_2g_3, g_3g_1g_2g_3g_2=g_2g_3g_1g_2g_3$. Then, as noticed in \cite{BANNAI}, this group is isomorphic to the Artin group $\langle a,b \ | \ ababab=bababa \rangle$ of type $I_2(6)$, an isomorphism being given by $a = g_3g_1g_2g_3$, $ab = g_3g_1g_2$, that is $b = g_3^{-1}$, and its inverse by $g_3 = b^{-1}$, $g_1 = a^{-1}b^{-1} a$, $g_2 =\Delta a^{-2}$ with $\Delta = ababab$. Notice that $\Delta$ is central here. Applying the results of \autoref{sect:descrArtinShephard} we get that every parabolic subgroup of $B$ is a conjugate of
either $\langle b^{-1} \rangle$ or $\langle \Delta a^{-2}\rangle$.

\section{Parabolic closures in Garside groups}
\label{sect:parabclosuresgarsidegroups}
In this section we describe some basic properties of Garside groups, recall the definition of a parabolic subgroup in such a group, and show that, under certain conditions, an element in a Garside group admits a parabolic closure, that is, a unique minimal (by inclusion) parabolic subgroup containing it.

\subsection{Garside groups and normal forms}

A Garside group is a group $G$ which admits a so-called Garside structure $(G,G^+,\Delta)$, where $G^+$ is a submonoid of $G$ (the submonoid of positive elements) and $\Delta\in G^+$ is called the Garside element, satisfying some suitable properties~\cite{Dehornoy-Paris}. We recall that such a Garside structure determines two partial orders in $G$: we say that $a\preccurlyeq b$ ($a$ is a prefix of $b$) if  $a^{-1}b\in G^+$, and we say that $b\succcurlyeq c$ ($c$ is a suffix of $b$) if $bc^{-1}\in G^+$. Both are lattice orders in $G$, the former is invariant under left-multiplication and the latter is invariant under right-multiplication. We will denote $a\vee b$ and $a\wedge b$ the join and meet of $a$ and $b$, respectively, with respect to $\preccurlyeq$, and we will denote $a\vee^{\Lsh} b$ and $a\wedge^{\Lsh} b$ the join and meet of $a$ and $b$, respectively, with respect to $\succcurlyeq$. Notice that $\wedge$ and $\vee$ are invariant under left-multiplication: if $d=a\wedge b$ and $m=a\vee b$ then $cd=ca\wedge cb$ and $cm=ca\vee cb$ for every $a,b,c\in G$. Similarly,  $\wedge^{\Lsh}$ and $\vee^{\Lsh}$ are invariant under right-multiplication.

By definition, an element $a$ is {\em positive} if and only if $1\preccurlyeq a$ or, equivalently, $a\succcurlyeq 1$. We will say that an element $x$ is {\em negative} if $x^{-1}$ is positive. Then $x$ is negative if and only if $x\preccurlyeq 1$ or, equivalently, $1\succcurlyeq x$. We say that a positive element $\delta\in G^+$ is {\em balanced} if the set of its positive prefixes coincides with the set of its positive suffixes, and in this case we call it the set of divisors of $\delta$:
$$
 \operatorname{Div}(\delta)=\{a\in G;\ 1\preccurlyeq a \preccurlyeq \delta \}= \{a\in G;\ \delta\succcurlyeq a\succcurlyeq 1\}.
$$

We remark that, given a balanced element $\delta$ and positive elements $a,b,c\in G^+$ such that $\delta=abc$, one has $b\in \operatorname{Div}(\delta)$. Indeed, we have $ab\preccurlyeq \delta$, hence $ab\in \operatorname{Div}(\delta)$ which implies that $\delta \succcurlyeq ab$, and then $\delta \succcurlyeq b$, that is, $b\in \operatorname{Div}(\delta)$.

In a Garside structure, the Garside element $\Delta$ is balanced, and its divisors are called {\em simple elements}. We denote by $S$ the set of simple elements: $S=\operatorname{Div}(\Delta)$. We will assume that the set $S$ is finite (this is usually part of the definition of a Garside group). The nontrivial simple elements which do not admit proper prefixes are called {\em atoms}. It is required that the set $S$ of simple elements  (and hence the set $\mathcal A$ of atoms) generates $G$.

For every element $x\in G$ one has $\Delta^p \preccurlyeq x \preccurlyeq \Delta^q$ for some integers $p\leq q$. The maximal integer $p$ and the minimal integer $q$ satisfying this property are called the {\it infimum} and the {\it supremum} of $x$, respectively. Every element $x\in G$ admits a unique decomposition $x=\Delta^p x_1\cdots x_r$ (its {\it left normal form}), where $x_1,\ldots,x_r$ are proper simple elements (not trivial and not $\Delta$) such that $x_ix_{i+1}\wedge \Delta=x_i$ for $i=1,\ldots,r-1$. In this case $p=\inf(x)$ and $p+r=\sup(x)$.

If $x\in G^+$ then $\inf(x)=p\geq 0$, and we can write its left normal form as $x=\Delta \Delta \cdots \Delta x_1\cdots x_r$, where the first $p$ factors are equal to $\Delta$. This is called the {\it left-weighted factorization} of $x$. It is the only way to decompose $x$ as $x=s_1\cdots s_q$, where each $s_i$ is a nontrivial simple element and $s_is_{i+1}\wedge \Delta=s_i$ for every $i$.

The analogous definitions can be done with respect to $\succcurlyeq$, so every element admits a right normal form, and every positive element admits a right-weighted factorization. It turns out that the infimum and the supremum of an element with respect to $\preccurlyeq$ coincide with its infimum and its supremum with respect to $\succcurlyeq$.

Let us now see a decomposition which is valid for every element in $G$, and which will be the most important for our purposes.

\begin{definition}
Given $x\in G$, we say that $x=a^{-1}b$ is the {\em reduced left-fraction decomposition} of $x$ if $a,b\in G^+$ and $a\wedge b=1$. We will say that $a$ and $b$ are the {\em left-denominator} and the {\em left-numerator} of $x$, respectively: $a=D_L(x)$, $b=N_L(x)$.
\end{definition}

It is well-known that the reduced left-fraction decomposition of an element $x\in G$ exists and is unique. We will later use the following simple result.

\begin{proposition}\label{P:fraction_simplification}
Let $x\in G$ and suppose that $x=c^{-1}d$ for some $c,d\in G^+$. Let $\alpha=c\wedge d$ and write $c=\alpha a$ and $d=\alpha b$. Then $x=a^{-1}b$ is the reduced left-fraction decomposition of $x$.
\end{proposition}

\begin{proof}
It is clear that $a$ and $b$ are positive (as $\alpha\preccurlyeq c$ and $\alpha\preccurlyeq d$). Also $x=c^{-1}d=a^{-1}\alpha^{-1}\alpha b =a^{-1}b$. On the other hand, $\alpha= c\wedge d = \alpha a \wedge \alpha b$. Left-multiplying by $\alpha^{-1}$ we get $a\wedge b=1$, so $a^{-1}b$ is the reduced left-fraction decomposition of $x$.
\end{proof}

Let us point out that $x$ is positive if and only if $D_L(x)=1$, and that $x$ is negative if and only if $N_L(x)=1$. Also, if $x=a^{-1}b$ is the reduced left-fraction decomposition of $x$, then $x^{-1}=b^{-1}a$ is the reduced left-fraction decomposition of $x^{-1}$.

In the same way as one defines the reduced left-fraction decomposition of an element $x\in G$, there is also a reduced right-fraction decomposition $x=uv^{-1}$, where $u,v\in G^+$ and $u\wedge^{\Lsh} v=1$. This decomposition is also unique, and can be obtained from a given decomposition $x=wy^{-1}$ with $w,y\in G^+$ by removing from $w$ and $y$ their greatest common suffix. We denote $u=N_R(x)$ and $v=D_R(x)$ the right-numerator and right-denominator of $x$, respectively.

Given a Garside structure $(G,G^+,\Delta)$, it is well known that $(G,G^+,\Delta^N)$ is also a Garside structure, for any $N\geq 1$. The positive elements and the prefix and suffix orders of both structures coincide, but the simple elements of the latter structure are those elements $x$ such that $1\preccurlyeq x \preccurlyeq \Delta^N$. It follows that, for every positive element $a\in G^+$, there is some $N$ big enough so that $a$ is a simple element with respect to $(G,G^+,\Delta^N)$.

We shall later need the following well known result:

\begin{proposition}\label{P:turn_into_positive}
Let $G$ be a Garside group. For every $x\in G$ there exists a central element $z\in Z(G)$ such that $zx\in G^+$.
\end{proposition}

\begin{proof}
  It suffices to multiply the left normal form of $x$ by a sufficiently large power of $\Delta$ to make it positive. Since conjugation by $\Delta$ permutes the atoms (which is a finite generating set), some power of $\Delta$ is central, so one can take a central power of $\Delta$ whose exponent is as big as needed.
\end{proof}

\subsection{Parabolic subgroups of a Garside group}

The main concept of this paper is that of parabolic subgroup. In the framework of Garside groups, parabolic subgroups have been defined by Godelle in \cite{GODELLE2007}.

\begin{definition}\label{D:support_balanced}
Let $(G,G^+,\Delta)$ be a Garside structure. The {\em support} of a balanced element $\delta\in G^+$ is the set of atoms which are divisors of $\delta$:
$$
  \operatorname{Supp}(\delta)=\operatorname{Div}(\delta)\cap \mathcal A.
$$
\end{definition}

The word {\em support} comes from the following property, which is clear from the previous arguments: An atom belongs to the support of $\delta$ if and only if it appears in some representative of $\delta$ as a product of atoms.

\begin{definition}\cite{GODELLE2007}
Let $(G,G^+,\Delta)$ be a Garside structure.
\begin{enumerate}
\item[(i)] Let $\delta$ be a balanced element of $\operatorname{Div}(\Delta)$, let $G_\delta$ be the subgroup of $G$ generated by $\operatorname{Supp}(\delta)$, and let $G_{\delta}^+=G_{\delta}\cap G^+$. We say that $G_\delta$ is a {\em standard parabolic subgroup} of $G$ {\em associated to $\delta$} if $\operatorname{Div}(\delta) = \operatorname{Div}(\Delta)\cap G_\delta^+$.

\item[(ii)] A {\em parabolic subgroup} of $G$ is a subgroup of $G$ which is conjugate to a standard parabolic subgroup (associated to some $\delta$).
\end{enumerate}
\end{definition}

Godelle checked that this definition indeed generalizes the concept of a standard parabolic subgroups for Artin groups as we defined it in \S \ref{sect:G(e,e,n)}. We will check in Section \ref{sect:parabclosurecbg} that is also coincides with the concepts of standard parabolic subgroups we defined in \S \ref{sect:G(e,e,n)} and \S \ref{sect:wellgen2refl}.

Godelle also shows in \cite{GODELLE2007} that standard parabolic subgroups in a Garside group are also Garside groups, and that both Garside structures are closely related. We enumerate here the main properties:

\begin{theorem}\cite{GODELLE2007}
Let $(G,G^+,\Delta)$ be a Garside structure. Let $\delta\in \operatorname{Div}(\Delta)$ be a balanced element such that $G_\delta$ is a standard parabolic subgroup of $G$. Then $(G_{\delta},G_{\delta}^+,\delta)$ is a Garside structure, where the lattice of $G_{\delta}$ is a sublattice of the lattice of $G$. Moreover, $G_{\delta}^+$ is closed under positive prefixes and positive suffixes in $G$. This implies that the gcd (lcm) of two elements in $G_\delta$ is the same seen in $G_\delta$ and seen in $G$. And the reduced left-fraction (right-fraction) decomposition of an element of $G_\delta$ is the same seen in $G_\delta$ and seen in $G$.
\end{theorem}

\subsection{LCM-Garside structures}
In this paper, the Garside structures we will be interested in will satisfy a very convenient property. For every subset $X=\{x_1,\ldots,x_r\}\subset \mathcal A$, denote $\Delta_X=x_1\vee \cdots \vee x_r$, which is always a simple element. If $X=\varnothing$, we consider $\Delta_X=1$. If $\Delta_X$ is balanced, we will denote $G_X=G_{\Delta_X}$ the subgroup generated by $\operatorname{Div}(\Delta_X)$, that is, the subgroup generated by $\operatorname{Supp}(\Delta_X)$. The property we require to a Garside structure is the following:

\begin{definition}\label{D:LCM-Garside_structure}
Let $(G,G^+,\Delta)$ be a Garside structure. We say that it is an LCM-Garside structure if:
\begin{enumerate}
\item $\Delta = \Delta_{\mathcal A}$.

\item For every $X\subset \mathcal A$, the element $\Delta_X$ is balanced.

\item For every $X\subset \mathcal A$, the subgroup $G_X=G_{\Delta_X}$ is a standard parabolic subgroup
associated to $\Delta_X$.
\end{enumerate}
\end{definition}

\begin{lemma}
If $(G,G^+,\Delta)$ is an LCM-Garside structure, then the standard parabolic subgroups of $G$ are precisely the subgroups of the form $G_X$, for $X\subset \mathcal{A}$.
\end{lemma}

\begin{proof}
If $(G,G^+,\Delta)$ is an LCM-Garside structure, all subgroups of the form $G_X$ are standard parabolic subgroups. Conversely, let $G_\delta$ be a standard parabolic subgroup, for some balanced element $\delta\in \operatorname{Div}(\Delta)$. Consider $X=\operatorname{Supp}(\delta)=\{x_1,\ldots,x_r\}\subset \mathcal A$. Since $x_i\preccurlyeq \delta$ for $i=1,\ldots,r$, it follows that $\Delta_X=x_1\vee \cdots \vee x_r\preccurlyeq \delta$. This implies that $\operatorname{Div}(\Delta_X)\subset \operatorname{Div}(\delta)$. Hence
$$
  X \subset \operatorname{Supp}(\Delta_X) =\operatorname{Div}(\Delta_X)\cap \mathcal A \subset \operatorname{Div}(\delta)\cap \mathcal A = \operatorname{Supp}(\delta)=X,
$$
which implies that $\operatorname{Supp}(\Delta_X)=X=\operatorname{Supp}(\delta)$, and then $G_X=G_{\Delta_X}=G_\delta$.
\end{proof}

We notice that, given $X\subset \mathcal A$, we do not necessarily have $\operatorname{Supp}(\Delta_X)=X$. Subsets of atoms satisfying this property will be important for us:

\begin{definition}
Let $(G,G^+,\Delta)$ be an LCM-Garside structure with set of atoms $\mathcal A$. For every nonempty subset $X=\{x_1,\ldots,x_r\}\subset \mathcal A$, let $\Delta_X=x_1\vee\cdots \vee x_r$. We define the {\it closure of $X$} as:
$$
      \overline{X}=\operatorname{Supp}(\Delta_X)=\operatorname{Div}(\Delta_X) \cap \mathcal A.
$$
We say that $X$ is {\em saturated} if $\overline{X}=X$.
\end{definition}

With the above definitions and results, the next result follows immediately:

\begin{proposition}
Let $(G,G^+,\Delta)$ be an LCM-Garside structure. A {\it standard parabolic subgroup} is a subgroup generated by a {\em saturated} set of atoms:
$$
             G_X=\left\langle x_1,\ldots,x_r \right\rangle, \qquad X=\{x_1,\ldots,x_r\}=\overline{X}.
$$

A {\it parabolic subgroup} is a subgroup $P$ of $G$ which is conjugate to a standard parabolic subgroup: $P=\left(G_X\right)^g=g^{-1}G_X g$ for some $g\in G$ and some saturated set of atoms $X$.
\end{proposition}

\subsection{Swaps and recurrent elements}

In Garside groups, the conjugacy problem is solved by using special kinds of conjugations (cyclings, decyclings, cyclic slidings) in order to compute suitable finite sets (super summit sets, ultra summit sets, sets of sliding circuits), see~\cite{ELRIFAIMORTON,GEBHARDT,GEBHARDTGM}. Their definitions and computations are sometimes technical, although they are quite efficient in practice.

In this paper we will show that there is a much simpler procedure to treat problems related to conjugacy, centralizers and parabolic subgroups in a Garside group. The algorithms are not faster than the ones mentioned above, but it is much better when one needs to show theoretical results.

Since we have a theoretical goal (to prove the existence of parabolic closures in some Garside groups), we will use this new approach. We will just use one kind of conjugation (that we call {\it swap}), and one finite set of elements ({\it recurrent elements} for swap).

\begin{definition}
Let $G$ be a Garside group. We define the {\it left-swap} function (or just the {\em swap} function) to be the map $\phi$ which sends $x=a^{-1}b$ (written as a reduced left-fraction) to $\phi(x)=ba^{-1}$.
\end{definition}

Notice that $\phi(x)=axa^{-1} = bxb^{-1}$, so $\phi(x)$ is conjugate to $x$.

\begin{definition}
Given $x\in G$, we say that $x$ is {\em recurrent for swap} (or just {\em recurrent}) if $\phi^m(x)=x$ for some $m>0$.
\end{definition}

One could think of $\phi$ as a conjugation that simplifies elements, and think of recurrent elements as those which are as {\it simplified} as possible.  It is important to see that every element can be conjugated to a recurrent element using swaps:

\begin{proposition}\label{P:swap_to_get_recurrent}
For every $x\in G$, there are integers $0\leq m < n$ such that $\phi^m(x)=\phi^n(x)$. The elements in the set $\{\phi^m(x),\ldots,\phi^{n-1}(x)\}$ are all recurrent, and the set will be called a {\em circuit for swap}.
\end{proposition}

\begin{proof}
The second sentence in the statement is trivial, so we just need to show the first one.

Let $x=a^{-1}b$ be the reduced left-fraction decomposition of $x$. This reduced left-fraction decomposition of $x$ depends only on $G^+$ (which determines the lattice order) and is independent of the Garside element we take. Hence, to simplify the arguments we will work with a Garside structure $(G,G^+,\Delta^N)$, with $N$ big enough so that $a$ and $b$ are simple elements. To avoid confusion, we will denote $\mathbf \Delta=\Delta^N$.

We have $aa'=\mathbf \Delta$ and $bb'=\mathbf \Delta$ for some simple elements $a'$ and $b'$. If we denote $\tau$ the conjugation by $\mathbf \Delta$, we have:
$$
    \phi(x)=ba^{-1}= b a' \mathbf \Delta^{-1} = \mathbf \Delta^{-1} \tau^{-1}(b) \tau^{-1}(a') = \left(\tau^{-1}(b')\right)^{-1} \tau^{-1}(a').
$$
Since conjugation by $\mathbf \Delta$ preserves the Garside structure, it follows that $\tau^{-1}(b')$ and $\tau^{-1}(a')$ are simple elements. In other words, there are simple elements $c$ and $d$ such that $\phi(x)=c^{-1}d$. By \autoref{P:fraction_simplification}, if $\alpha=c\wedge d$ and we write $c=\alpha c_0$ and $d=\alpha d_0$ for some simple elements $c_0, d_0$, then $\phi(x)=c_0^{-1}d_0$ is the reduced left-fraction decomposition of $\phi(x)$. Therefore the left-numerator and the left-denominator of $\phi(x)$ are simple elements. Since the number of simple elements is finite, the sequence $\{\phi^i(x)\}_{i\geq 0}$ must become periodic.
\end{proof}

It is worth noticing that if $x$ is positive then it is recurrent, as its left-denominator is trivial and then $\phi(x)=x$. The same happens if $x$ is negative, as in this case the left-numerator of $x$ is trivial and then $\phi(x)=x$.

The following result is just a simple observation:

\begin{proposition}\label{P:conjugate_to_recurrent_in_G_X}
Let $(G,G^+,\Delta)$ be a Garside structure, and let $G_{\delta}$ be the standard parabolic subgroup of $G$ associated to $\delta$. If $x\in G_{\delta}$, there is $\alpha\in G_{\delta}$ such that $x^\alpha$ is recurrent.
\end{proposition}

\begin{proof}
Let $x=a^{-1}b$ be the reduced left-fraction decomposition of $x$ in $G$. We know that $a^{-1}b$ is also the reduced left-fraction decomposition of $x$ in $G_{\delta}$, hence $a,b\in G_{\delta}^+$. But then $\phi(x)=axa^{-1}\in G_{\delta}$.

Applying the same reasoning to each iterated swap of $x$, it follows that all elements in the sequence $\{\phi^i(x)\}_{i\geq 0}$ belong to $G_{\delta}$, and one can take a conjugating element in $G_{\delta}$ when conjugating $\phi^i(x)$ to $\phi^{i+1}(x)$ for every $i\geq 0$. By \autoref{P:swap_to_get_recurrent} some element in the sequence is recurrent, so the result follows.
\end{proof}

Notice that, given a standard parabolic subgroup $G_{\delta}\subset G$, the definition of swap of an element in $G_{\delta}$ does not depend on the Garside structure (either $(G,G^+,\Delta)$ or$(G_{\delta},G_{\delta}^+,\delta)$) that one is considering, since the left-fraction decompositions in $G_{\delta}$ and in $G$ coincide. Hence, an element of $G_{\delta}$ is recurrent with respect to $(G,G^+,\Delta)$ if and only if it is recurrent with respect to $(G_{\delta},G_{\delta}^+,\delta)$. We can then just say that the element is recurrent, without mentioning any Garside structure.

We finish this section by relating recurrent elements with positive (or negative) elements in a conjugacy class.

\begin{lemma}\label{L:phi_conjugates_to_positive}
Let $x\in G$. If $x$ is conjugate to a positive element, then $\phi^m(x)\in G^+$ for some $m\geq 0$.
\end{lemma}

\begin{proof}
For those familiar with Garside theory, we will parallel the usual proof showing that {\em iterated cycling} takes $x$ to a positive element~\cite{ELRIFAIMORTON,BIRMANKOLEE}.

Let $x=a^{-1}b$ be the reduced left-fraction decomposition of $x$. We will use the Garside structure $(G,G^+,\Delta^N)$ such that $a$ and $b$ are simple elements. Let us denote $\mathbf \Delta = \Delta^N$ to avoid confusion. Let $a'$ be the simple element such that $a'a=\mathbf \Delta$. We have $x=\mathbf \Delta^{-1}a'b$. Recall that $a\wedge b=1$, hence $\mathbf \Delta\wedge a'b= a'a\wedge a'b = a'(a\wedge b)=a'$. So $a'b$ is a left-weighted decomposition.

We are assuming that $x$ is conjugate to a positive element, so there exists $c\in G$ such that $cxc^{-1}\in G^+$. By \autoref{P:turn_into_positive} we can assume that $c$ is positive, multiplying it by a suitable central element if necessary. We then have $c,d\in G^+$ such that $cxc^{-1}=d$. Hence $c\mathbf \Delta^{-1}a'bc^{-1}=d$, and therefore $c \mathbf \Delta^{-1}a'b=dc\in G^+$.  If we denote $\tilde c=\mathbf \Delta c \mathbf \Delta^{-1}\in G^+$, we finally obtain
$$
   \mathbf \Delta^{-1}\tilde c a'b \in G^+.
$$
This means $1\preccurlyeq  \mathbf \Delta^{-1}\tilde c a'b$, hence $\mathbf \Delta\preccurlyeq \tilde c a'b$. Now let us see (a usual procedure in Garside theory) that we can remove $b$ from the above expression, as $a'b$ is left weighted: We have that $\mathbf \Delta \preccurlyeq \tilde c\mathbf \Delta$ (as $\mathbf \Delta^{-1}\tilde c\mathbf \Delta$ is positive). Hence
$$
\mathbf \Delta \preccurlyeq (\tilde c\mathbf \Delta) \wedge (\tilde c a' b) = \tilde c (\mathbf \Delta\wedge a'b) = \tilde c a',
$$
the last equality holding as $a'b$ is left-weighted. Now recall that $\mathbf \Delta$ is a prefix of an element if and only if it is a suffix of the same element (both assertions mean that the infimum of the element is at least 1). Hence $\tilde c a' \succcurlyeq \mathbf \Delta$. That is,
$$
\mathbf \Delta c = \tilde c \mathbf \Delta= \tilde c a'a \succcurlyeq \mathbf \Delta a.
$$
Conjugating by $\mathbf \Delta$ (which preserves the Garside structure), we get $c \mathbf \Delta\succcurlyeq a \mathbf \Delta$, and finally $c\succcurlyeq a$.

Therefore, we have shown that if $c$ is a positive element such that $cxc^{-1}$ is positive, then $c\succcurlyeq a$.  We can then decompose $c=c_1a$, and we obtain $cxc^{-1}=c_1a(a^{-1}b)a^{-1}c_1^{-1} = c_1(ba^{-1})c_1^{-1} = c_1 \phi(x) c_1^{-1}$, where $1\preccurlyeq c_1\prec c$.

If $\phi(x)$ is positive, we are done. Otherwise, we repeat the process and we will find $1\preccurlyeq c_2\prec c_1 \prec c$ such that $c_2 \phi^2(x) c_2^{-1}$ is positive, so we can repeat the process again with $\phi^2(x)$. This process must stop, as it cannot exist an infinite descending chain of positive elements in $G$, hence $\phi^m(x)\in G^+$ for some $m\geq 0$, as we wanted to show.
\end{proof}

\begin{remark}\label{R:minimal_to_recurrent} It is interesting to point out a consequence of the above proof: If $x$ is conjugate to a positive element, and $\phi^i(x)=a_i^{-1}b_i$ is the reduced left-fraction decomposition of $\phi^i(x)$ for each $i\geq 0$, then $c=a_{m-1}\cdots a_1a_0$ is the minimal positive element, with respect to $\succcurlyeq$ such that $cxc^{-1}$ is positive (provided $m$ is the smallest nonnegative integer such that $\phi^m(x)$ is positive). In other words, iterated swaps conjugate $x$ to a positive element in the fastest possible way (conjugating by positive elements on the left). If one prefers to use $\preccurlyeq$ and to conjugate by positive elements on the right, one can parallel the same arguments just by using reduced right-fractions and right-swaps, which are defined in the symmetric way.
\end{remark}

We can now characterize the sets of recurrent elements in a conjugacy class, in two particular cases. Given $x\in G$, let $\RC(x)$ be the set of recurrent elements conjugate to $x$, which coincides with the set of circuits for swap in the conjugacy class of $x$. Let $C^+(x)$ be the set of positive elements conjugate to $x$ and let $C^-(x)$ be the set of negative elements conjugate to $x$ (any of the two latter sets could be empty).

\begin{proposition}\label{P:recurrent=positive}
Let $(G,G^+,\Delta)$ be a Garside structure and let $x\in G$. One has:
\begin{enumerate}

 \item If $x$ is conjugate to a positive element, then $\RC(x)=C^+(x)$.

 \item If $x$ is conjugate to a negative element, then $\RC(x)=C^-(x)=\left(C^+(x^{-1})\right)^{-1}$.

\end{enumerate}
\end{proposition}

\begin{proof}
\begin{enumerate}

\item Suppose that $x$ is conjugate to a positive element. If $y\in C^+(x)$ then $y$ is conjugate to $x$ and $\phi(y)=y$, hence $y\in \RC(x)$.  Conversely, let $y\in \RC(x)$. Since $y$ is conjugate to $x$, it is conjugate to a positive element. By \autoref{L:phi_conjugates_to_positive}, $\phi^m(y)$ is positive for some $m\geq 0$, and then $\phi^{m+k}(y)=\phi^m(y)$ for every $k\geq 0$. But $y$ is recurrent, so $\phi^n(y)=y$ for some $n>0$, and this implies that $y=\phi^{nm}(y)=\phi^m(y)$. Hence $y$ is positive, and the orbit of $y$ under $\phi$ consists of just one element, namely $y$.

\item If $x$ is conjugate to a negative element, $x^{-1}$ is conjugate to a positive element, hence $\RC(x^{-1})=C^+(x^{-1})$ by the previous property. Now recall that for every $z\in G$ with reduced left-fraction decomposition $z=c^{-1}d$, the reduced left-fraction decomposition of $z^{-1}$ is $d^{-1}c$. Hence $\phi(z^{-1})=cd^{-1}=(dc^{-1})^{-1}=\phi(z)^{-1}$. Therefore, taking inverses commutes with $\phi$ and also preserves conjugations and conjugating elements. It follows that $\RC(x^{-1})=\RC(x)^{-1}$, that $\left(C^-(x)\right)^{-1}=C^+(x^{-1})$, and then $\RC(x)=\left(\RC(x^{-1})\right)^{-1}=\left(C^+(x^{-1})\right)^{-1}=C^{-}(x)$.

\end{enumerate}
\end{proof}

\subsection{Transport for swap and convexity}

Let $(G,G^+,\Delta)$ be a Garside structure and let $x\in G$. In this section we will explain some important properties of $\RC(x)$. Unfortunately, we do not know whether the set $\RC(x)$ is always finite, even in the case in which it consists only of the positive conjugates of $x$. In this later case, the set is finite if $G$ has homogenous relations, like in braid groups or, more generally, in Artin-Tits groups of spherical type. But in a general Garside group this set could a priori be infinite.

In any case, we will see that $\RC(x)$ satisfies the same properties as similar sets, which are usually introduced to solve the conjugacy problem like super summit sets~\cite{ELRIFAIMORTON}, ultra summit sets~\cite{GEBHARDT}, or sets of sliding circuits~\cite{GEBHARDTGM,GEBHARDTGM2}. One of the basic properties is that their elements are connected through conjugations by simple elements. Although this can be shown by comparing the set $\RC(x)$ with other sets like the ultra summit set of $x$ (see~\cite{GEBHARDT}), we will proceed to show all details avoiding the use of Garside normal forms.

We start with a very basic property.

\begin{lemma}\label{L:Delta_preserves_recurrent}
Let $(G,G^+,\Delta)$ be a Garside structure and let $x\in G$. If $y\in \RC(x)$ then $y^\Delta\in \RC(x)$.
\end{lemma}

\begin{proof}
Since conjugation by $\Delta$ preserves the lattice structure of $G$, it follows that if the reduced left-fraction decomposition of some element $z$ is $a^{-1}b$ then the reduced left-fraction decomposition of $z^\Delta$ is $(a^\Delta)^{-1}(b^\Delta)$. Hence, $\phi(z^\Delta)=(b^\Delta)(a^\Delta)^{-1} = (ba^{-1})^\Delta=\phi(z)^\Delta$. That is, applying a left swap commutes with conjugation by $\Delta$.

Given $y\in \RC(x)$, there is some $m>0$ such that $\phi^m(y)=y$. By the argument in the previous paragraph, we have $\phi^m(y^\Delta)=(\phi^m(y))^\Delta = y^\Delta$. Therefore $y^\Delta \in \RC(x)$.
\end{proof}

We will consider positive conjugates joining elements of $G$. One important concept, taken from~\cite{GEBHARDT}, is the \emph{transport} of a conjugating element.

\begin{proposition}\label{P:transport_definition}
Let $(G,G^+,\Delta)$ be a Garside structure and let $y,z\in G$ which are conjugate by a positive element $u\in G^+$, that is, $y^u=z$. Consider the reduced left-fraction decompositions $y=a^{-1}b$ and $z=c^{-1}d$. Then we define the \emph{transport of $u$ at $y$} as $u^{(1)}=auc^{-1}$, and the following conditions hold:
\begin{enumerate}

\item $u^{(1)}=auc^{-1}=bud^{-1}= au\wedge bu$.

\item $\phi(y)^{\left(u^{(1)}\right)}=\phi(z)$.

\end{enumerate}
\end{proposition}

\begin{proof}
We know that $y^u=z$, hence $u^{-1}(a^{-1}b)u=c^{-1}d$, which is equivalent to $auc^{-1}=bud^{-1}$.

Now let us represent the positive elements $au$ and $bu$:
$$
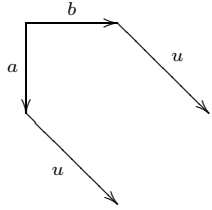
\xymatrix@C=12mm@R=12mm{
*=0{} \ar[d]_{a} \ar[r]-0+0^{b} & *=0{} \ar[rd]-0+0^{u} \\
*=0{} \ar[rd]-0+0_{u} & *=0{}  & *=0{} \\
*=0{} & *=0{} & *=0{}
}
$$
Recall that $z$ is precisely $(au)^{-1}(bu)$. From \autoref{P:fraction_simplification}, cancelling the element $\alpha=au\wedge bu$ in the middle of the above expression, one obtains the reduced left-fraction decomposition of $z$, that is, $c^{-1}d$. In other words, $au=\alpha c$ and $bu=\alpha d$, which means that $au\wedge bu=\alpha=auc^{-1}=bud^{-1}= u^{(1)}$:
$$
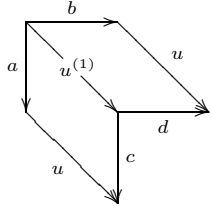
\xymatrix@C=12mm@R=12mm{
*=0{} \ar[d]_{a} \ar[r]-0+0^{b} \ar[rd]|{\phantom{x}u^{(1)}} & *=0{} \ar[rd]-0+0^{u} \\
*=0{} \ar[rd]-0+0_{u} & *=0{} \ar[d]-0+0^{c} \ar[r]-0+0_{d} & *=0{} \\
*=0{} & *=0{} & *=0{}
}
$$
This shows the first condition.

The second condition is an easy observation. Since $u^{(1)}=auc^{-1}=bud^{-1}$, one has:
$$
\phi(y)^{\left(u^{(1)}\right)}=(ba^{-1})^{\left(u^{(1)}\right)}=(bud^{-1})^{-1}(ba^{-1})(auc^{-1})= dc^{-1} = \phi(z).
$$
Hence, if $u$ conjugates $y$ to $z$, then $u^{(1)}$ conjugates $\phi(y)$ to $\phi(z)$.
\end{proof}

The transport satisfies very useful properties:

\begin{proposition}\label{P:transport_properties}
Let $(G,G^+,\Delta)$ be a Garside structure and let $y\in G$. Suppose that $u,v\in G^+$ are positive elements and let $u^{(1)}$, $v^{(1)}$, $(u\wedge v)^{(1)}$ and $(\Delta^m)^{(1)}$ be the transports of $u$, $v$, $u\wedge v$ and $\Delta^m$ at $y$, respectively, for $m\geq 0$. Then one has:
\begin{enumerate}

\item $(\Delta^m)^{(1)}=\Delta^m$ for every $m\geq 0$.

\item $(u\wedge v)^{(1)}=u^{(1)}\wedge v^{(1)}$.

\item If $u\preccurlyeq v$ then $u^{(1)}\preccurlyeq v^{(1)}$.

\end{enumerate}
\end{proposition}

\begin{proof}
Let $a^{-1}b$ be the reduced left-fraction decomposition of $y$. We know from \autoref{P:transport_definition} that for every positive element $w\in G^+$, its transport at $y$ is $w^{(1)}=aw\wedge bw$. Then the first condition is shown as follows:
$$
  (\Delta^m)^{(1)}=a\Delta^m \wedge b\Delta^m = \Delta^m a^{\Delta^m} \wedge \Delta^m b^{\Delta^m} = \Delta^m (a^{\Delta^m}\wedge b^{\Delta^m})=\Delta^m (a\wedge b)^{\Delta^m}=\Delta^m,
$$
where we have used that conjugation by $\Delta$ preserves the lattice structure, and that $a\wedge b=1$.

The second condition also follows easily from the definition of transport:
$$
 (u\wedge v)^{(1)}=a(u\wedge v)\wedge b(u\wedge v)= au \wedge av \wedge bu \wedge bv = (au\wedge bu) \wedge (av \wedge bv)= u^{(1)}\wedge v^{(1)}.
$$

Finally, $u\preccurlyeq v$ if and only if $u\wedge v=u$. If this is the case, $u^{(1)}\wedge v^{(1)} = (u\wedge v)^{(1)}=u^{(1)}$, which is equivalent to $u^{(1)}\preccurlyeq v^{(1)}$. This shows the third condition.
\end{proof}

Notice that one can iterate the transport map, provided each transport is performed successively at $y,\phi(y), \phi^2(y),\ldots$ As usual, one denotes $u^{(2)}=(u^{(1)})^{(1)}$ and, in general, $u^{(k)}=(u^{(k-1)})^{(1)}$ for $k\geq 2$. Let us see that the transport map behaves well when applied to conjugating elements between recurrent elements.

\begin{lemma}\label{L:transport_repeats}
Let $(G,G^+,\Delta)$ be a Garside structure and let $x\in G$. Given two elements $y,z\in \RC(x)$ and a positive element $u\in G^+$ such that $y^u=z$, then there exists $k>0$ such that $\phi^k(y)=y$ and $u^{(k)}=u$.
\end{lemma}

\begin{proof}
Let $m=\sup(u)$ and recall that we have $1\preccurlyeq u \preccurlyeq \Delta^m$. By \autoref{P:transport_properties} we have $1\preccurlyeq u^{(k)}\preccurlyeq \Delta^m$ for every $k>0$. This implies that the set of iterated transports of $u$ is finite.

Now let $p,q$ be positive integers such that $\phi^p(y)=y$ and $\phi^q(z)=z$. These integers exist as $y$ and $z$ are left recurrent elements. Taking $n=\textrm{lcm}(p,q)$ we have $\phi^n(y)=y$ and $\phi^n(z)=z$. Since the element $u^{(k)}$ conjugates $\phi^k(y)$ to $\phi^k(z)$ for every $z\geq 0$, it follows that $u^{(in)}$ conjugates $y$ to $z$, for every $i\geq 0$. Since the set $\{u^{(in)}\}_{i\geq 0}$ is finite, there are some $0\leq i_1<i_2$ such that $u^{(i_1n)}=u^{(i_2n)}$. We take $i_1$ as small as possible.

Suppose that $i_1>0$ and notice that, by definition, the transport at a given element is an injective map. This implies, taking the preimages of $u^{(i_1n)}=u^{(i_2n)}$ under transport (based at the elements in the orbit of $y$) $n$ times, that $u^{((i_1-1)n)}=u^{((i_2-1)n)}$. This contradicts the minimality of $i_1$. Hence $i_1=0$ and then $u=u^{(k)}$ where $k=i_2n>0$, as we wanted to show.
\end{proof}

We can then show a very important property, satisfied by all sets which are used in general to solve the conjugacy problem in Garside groups:

\begin{proposition}\label{P:convexity}
Let $(G,G^+,\Delta)$ be a Garside structure and let $x\in G$. Given $y\in \RC(x)$, if $\alpha,\beta\in G$ are such that $y^\alpha, y^\beta\in \RC(x)$, then $y^{\alpha\wedge \beta}\in \RC(x)$.
\end{proposition}

\begin{proof}
Suppose first that $\alpha$ and $\beta$ are positive. By \autoref{L:transport_repeats} we know that there are $k,m\geq 0$ such that $\phi^k(y)=y$ and $\alpha^{(k)}=\alpha$, and $\phi^m(y)=y$ and $\beta^{(m)}=\beta$. Taking $N$ to be a multiple of $k$ and $m$, it follows that $\alpha^{(N)}=\alpha$ and $\beta^{(N)}=\beta$, and both are conjugating elements starting at $y$. By \autoref{P:transport_properties}, $(\alpha\wedge \beta)^{(N)}=\alpha^{(N)}\wedge \beta^{(N)}=\alpha\wedge \beta$. Since $(\alpha\wedge \beta)^{(N)}$ conjugates $y$ to $\phi^N(y^{\alpha\wedge \beta})$ by definition of transport, and $\alpha\wedge \beta$ obviously conjugates $y$ to $y^{\alpha\wedge \beta}$, it follows that $\phi^N(y^{\alpha\wedge \beta})=y^{\alpha\wedge \beta}$, that is, $y^{\alpha\wedge \beta}\in \RC(x)$, as we wanted to show.

Now, if $\alpha$ and $\beta$ are arbitrary elements, let $M>0$ such that $\Delta^M$ is central, and $\Delta^M \alpha$ and $\Delta^M \beta$ are positive. Then $y^{\Delta^M\alpha}=y^{\alpha}\in \RC(x)$ and $y^{\Delta^M\beta}=y^{\beta}\in \RC(x)$. Hence, by the above paragraph, $y^{\Delta^{M}\alpha\wedge \Delta^{M}\beta}\in \RC(x)$. The proof finishes by noticing that $y^{\Delta^{M}\alpha\wedge \Delta^{M}\beta} = y^{\Delta^{M}(\alpha\wedge \beta)}= y^{\alpha\wedge \beta}$.
\end{proof}

As usual, the above property allows us to show that the elements of $\RC(x)$ are connected by simple conjugating elements.

\begin{corollary}\label{C:connected_by_simple}
Let $(G,G^+,\Delta)$ be a Garside structure and let $x\in G$. For every pair of distinct elements $y,z\in \RC(x)$, there is a sequence of simple elements $u_1,\ldots,u_m\in S$ such that $y^{u_1\cdots u_k}\in \RC(x)$ for $k=1,\ldots,m$, and $y^{u_1\cdots u_m}=z$.
\end{corollary}

\begin{proof}
Since $y,z\in \RC(x)$, they are conjugate. Let $u\in G^+$ be a nontrivial positive element such that $y^u=z$. By \autoref{L:Delta_preserves_recurrent}, $y^{\Delta}\in \RC(x)$. Hence, by \autoref{P:convexity}, $y^{u\wedge \Delta}\in \RC(x)$.

Let us denote $u_1=u\wedge \Delta$. Since $u$ is nontrivial and the set of atoms generate $G^+$ as a monoid, there must exist an atom $a\in \mathcal A$ such that $a\preccurlyeq u$. Since every atom is a prefix of $\Delta$, it follows that $a\preccurlyeq u_1$, hence $u_1$ is nontrivial. We can then write $u=u_1v_1$ for some $v_1\in G^+$. Then $v_1$ is a positive element which conjugates $y^{u_1}\in \RC(x)$ to $z\in \RC(x)$. If $v_1$ is nontrivial, we can apply the same reasoning to find a nontrivial simple element $u_2\preccurlyeq v_1$ such that $y^{u_1u_2}\in\RC(x)$. Then $u=u_1u_2v_2$ for some $v_2\in G_+$.

We can continue this process, but we can make at most $||u||$ steps before we obtain that $v_k$ is trivial, as $u$ cannot be decomposed in more than $||u||$ nontrivial positive elements. So there is some $m\leq ||u||$ such that $u=u_1\cdots u_m$, where $y^{u_1\cdots u_k}\in \RC(x)$ for $k=1,\ldots,m$ by construction, and $y^{u_1\cdots u_m}=y^u=z$.
\end{proof}

The above result can be used to compute elements in $\RC(x)$ starting with a single element $y\in \RC(x)$. One just needs to conjugate $y$ by all simple elements (a finite set), and keep the new conjugates of $y$ which belong to $\RC(x)$. \autoref{C:connected_by_simple} implies that every element in $\RC(x)$ can be obtained in this way. The problem is that we are not sure whether $\RC(x)$ is finite or not, so the process may never terminate.

To avoid the above problem, we can restrict our attention to finite subsets of $\RC(x)$, which also satisfy the property analogous to \autoref{P:convexity}, and therefore it will be possible to compute them completely using the above procedure.

\begin{definition}
Let $(G,G^+,\Delta)$ be a Garside structure and let $x\in G$. For every $m>0$, let $\RC^{m}(x)$ be the following subset of the conjugacy class:
$$
   \RC^{m}(x)=\{y\in \RC(x); \ sup(N_L(y))\leq m \textrm{ and } \sup(D_L(y))\leq m\}.
$$
\end{definition}

It is clear that $\RC^m(x)$ is finite, and that it is nonempty for some $m>0$. One can also show that $\RC^m(x)$ satisfies the property analogous to \autoref{P:convexity}, and then one can solve the conjugacy problem in $G$ using this set. But this is not the goal of this paper, and the solution is not computationally better than the one using sets of sliding circuits~\cite{GEBHARDTGM,GEBHARDTGM2}.

\subsection{Support of an element and parabolic closure}

In this section we will show that, if $(G,G^+,\Delta)$ is an LCM-Garside structure, and furthermore it is what we will call {\it support-preserving}, then every element $x\in G$ admits a parabolic closure, that is, a unique parabolic subgroup $\PC_G(x)$ which is minimal (with respect to inclusion) among all parabolic subgroups of $G$ containing $x$. We will sometimes write $\PC(x)$ instead of $\PC_G(x)$ if the group $G$ is clear by the context.

For this purpose we need to generalize the notion of {\it support} to the case of an arbitrary element of $G$, not necessarily a balanced, positive element.

\begin{definition}
Let $(G,G^+,\Delta)$ be an LCM-Garside structure. Let $x\in G$ whose reduced left-fraction decomposition is $x=a^{-1}b$. Write $a=x_1\cdots x_r$ and $b=x_{r+1}\cdots x_{m}$ as products of (not necessarily distinct) atoms. If we denote $X=\{x_1,\ldots,x_m\}$, we define the {\em support} of $x$ as $\textrm{Supp}(x)=\overline{X}$.
\end{definition}

Notice that, if $x$ is a positive balanced element, this definition coincides with the one given in \autoref{D:support_balanced}.

There are some Garside groups, for instance Artin groups of spherical type, in which all words representing a positive element involve the same set of atoms. In those cases, the support of $x=a^{-1}b$ (where this is its reduced left-fraction decomposition) is the set of atoms which appear in the word $v^{-1}w$, where $v$ is any representative of $a$ and $w$ is any representative of $b$. But some other Garside groups, like the braid group of $G(e,e,n)$, do not satisfy the mentioned property, as we will later see. This forces us to show the following:

\begin{proposition}\label{P:support_well_defined}
Let $(G,G^+,\Delta)$ be an LCM-Garside structure. The support of an element $x\in G$ is well defined.
\end{proposition}

\begin{proof}
Consider first a positive element $a\in G^+$. We will show that if one can write $a=x_1\cdots x_r=y_1\cdots y_s$ with $X=\{x_1,\ldots,x_r\}\subset \mathcal A$ and $Y=\{y_1,\ldots,y_s\}\subset \mathcal A$, then $\overline{X}=\overline{Y}$.

Since $(G,G^+,\Delta)$ is an LCM-Garside structure, $a$ is also a positive element in the standard parabolic subgroups $G_X=G_{\overline{X}}$ and $G_Y=G_{\overline{Y}}$.  This means that $x_1\cdots x_r$ is positive in $G_{\overline{Y}}$. Since $G_{\overline{Y}}^+$ is closed under suffixes and prefixes, it follows that $x_i\cdots x_r\in G_{\overline{Y}}^+$, and then $x_i\in G_{\overline{Y}}^+$, for $i=1,\ldots,r$. Since the atoms in $G_{\overline{Y}}$ are precisely the elements in $\overline{Y}$, we have shown that $X\subset \overline{Y}$, and then $\overline{X}\subset \overline{Y}$. Exchanging the roles of $X$ and $Y$ we get $\overline{Y}\subset \overline{X}$, thus $\overline{X}=\overline{Y}$, as we wanted to show. So the support is well defined for positive elements.

Let now $x\in G$ be an arbitrary element whose reduced left-fraction decomposition is $x=a^{-1}b$. Write $a=x_1\cdots x_r$ and $b=x_{r+1}\cdots x_m$ as product of atoms, and let $X_1=\{x_1,\ldots,x_r\}$ and $X_2=\{x_{r+1},\ldots,x_m\}$. Suppose that we can also write $a=y_1\cdots y_k$ and $b=y_{k+1}\cdots y_l$, and denote $Y_1=\{y_1,\ldots,y_k\}$ and $Y_2=\{y_{k+1},\ldots,y_l\}$. We have already shown that $\overline{X_1}=\overline{Y_1}$, and that $\overline{X_2}=\overline{Y_2}$.

Notice that $\Delta_{Z}=\Delta_{\overline{Z}}$ for every set $Z$ of atoms. Also, by definition, $\Delta_{Z_1\cup Z_2}= \Delta_{Z_1}\vee \Delta_{Z_2}$ for all sets of atoms $Z_1$ and $Z_2$. Hence:
$$
  \Delta_{X_1\cup X_2}= \Delta_{X_1}\vee \Delta_{X_2} = \Delta_{\overline{X_1}}\vee \Delta_{\overline{X_2}} = \Delta_{\overline{Y_1}}\vee \Delta_{\overline{Y_2}} = \Delta_{Y_1}\vee \Delta_{Y_2} = \Delta_{Y_1\cup Y_2}
$$
Therefore:
$$
 \overline{X_1\cup X_2} = \operatorname{Div}(\Delta_{X_1\cup X_2})\cap \mathcal A  = \operatorname{Div}(\Delta_{Y_1\cup Y_2})\cap \mathcal A  = \overline{Y_1\cup Y_2}.
$$
So the support is well defined for every element in $G$.
\end{proof}

Now we introduce another property which we will require for a Garside group $G$ to satisfy.

\begin{definition}\label{D:support_preserved_under_conjugation}
Let $(G,G^+,\Delta)$ be an LCM-Garside structure. We say that this structure is {\em support-preserving} if, for every pair of conjugate positive elements $y,z\in G^+$, and every $\alpha\in G$ such that $y^\alpha=z$, one has:
$$
    \left(G_{\operatorname{Supp}(y)}\right)^\alpha=G_{\operatorname{Supp}(z)}.
$$
\end{definition}

\begin{proposition}\label{P:support_preserved_for_recurrent}
Let $(G,G^+,\Delta)$ be a support-preserving LCM-Garside structure, and let $x\in G$. For every $y,z\in \RC(x)$ and every $\alpha\in G$ such that $y^\alpha=z$, one has:
$$
    \left(G_{\operatorname{Supp}(y)}\right)^\alpha=G_{\operatorname{Supp}(z)}.
$$
\end{proposition}

\begin{proof}
If $x$ is conjugate to a positive element, then $y,z\in \RC(x)=C^+(x)$ by \autoref{P:recurrent=positive}. Hence the result holds as $(G,G^+,\Delta)$ is support-preserving.

If $x$ is conjugate to a negative element, then $y^{-1},z^{-1}\in \RC(x)^{-1}=C^+(x^{-1})$ by \autoref{P:recurrent=positive}. Hence $\left(G_{\operatorname{Supp}(y^{-1})}\right)^\alpha=G_{\operatorname{Supp}(z^{-1})}$. Now notice that $\operatorname{Supp}(\beta^{-1})=\operatorname{Supp}(\beta)$ for every $\beta\in G$, hence $\left(G_{\operatorname{Supp}(y)}\right)^\alpha=G_{\operatorname{Supp}(z)}$.

Finally, suppose that $x$ is conjugate to neither a positive nor a negative element. For every $i\geq 0$, let $\phi^i(y)=a_i^{-1}b_i$ be the reduced left-fraction decomposition of $\phi^i(y)$, and let $\phi^i(z)=c_i^{-1}d_i$ be the reduced left-fraction decomposition of $\phi^i(z)$.

We have $y^\alpha=z$. By \autoref{P:turn_into_positive} we can assume that $\alpha$ is positive, multiplying it by a central element if necessary. We then have
$$
(a_0\alpha)^{-1}(b_0\alpha)= \alpha^{-1}y\alpha = z = c_0^{-1}d_0.
$$
Consider the transport
$$
   \alpha^{(1)}= a_0\alpha \wedge b_0\alpha=a_0\alpha c_0^{-1} = b_0\alpha d_0^{-1}.
$$
Notice that we have the commutative diagrams of conjugations:
$$
\xymatrix@C=12mm@R=12mm{
z  \ar[r]^{c_0^{-1}}  & \phi(z)
\\
y \ar[u]^{\alpha} \ar[r]_{a_0^{-1}} & \phi(y) \ar[u]_{\alpha^{(1)}}
}
\hspace{2cm}
\xymatrix@C=12mm@R=12mm{
z  \ar[r]^{d_0^{-1}}  & \phi(z)
\\
y \ar[u]^{\alpha} \ar[r]_{b_0^{-1}} & \phi(y) \ar[u]_{\alpha^{(1)}}
}
$$
Now recall that both $y$ and $z$ are recurrent, so there exists some $k>0$ such that $\phi^{k}(y)=y$, $\phi^{k}(z)=z$ and $\alpha^{(k)}=\alpha$ by \autoref{L:transport_repeats}. We obtain the following commutative diagram of conjugations:
$$
\xymatrix@C=12mm@R=12mm{
z  \ar[r]^{c_0^{-1}}  & \phi(z) \ar[r]^{c_1^{-1}} & \phi^2(z) \ar@{.>}[r] & \phi^{k-1}(z) \ar[r]^{c_{k-1}^{-1}} & z
\\
y \ar[u]^{\alpha} \ar[r]_{a_0^{-1}} & \phi(y) \ar[u]_{\alpha^{(1)}} \ar[r]_{a_1^{-1}} & \phi^2(y) \ar[u]_{\alpha^{(2)}} \ar@{.>}[r] & \phi^{k-1}(y) \ar[u]_{\alpha^{(k-1)}} \ar[r]_{a_{k-1}^{-1}} & y \ar[u]_{\alpha^{(k)}=\alpha}
}
$$
Simplifying the diagram, we have:
$$
\xymatrix@C=12mm@R=12mm{
z  \ar[rr]^{(c_{k-1}\cdots c_0)^{-1}}  & & z
\\
y \ar[u]^{\alpha} \ar[rr]_{(a_{k-1}\cdots a_0)^{-1}} & & y \ar[u]_{\alpha}
}
$$
Let us denote $g_1=a_{k-1}\cdots a_0$ and $h_1=c_{k-1}\cdots c_0$. Both elements are positive, and we have $\alpha=g_1\alpha h_1^{-1}$. 

Now notice that we also had $\alpha^{(1)}=b_0\alpha d_0^{-1}$. Repeating the above arguments, if we define $g_2=b_{k-1}\cdots b_0$ and $h_2=d_{k-1}\cdots d_0$, we have $\alpha=\alpha^{(k)}=g_2 \alpha h_2^{-1}$. Therefore $\alpha= g_1g_2 \alpha h_2^{-1}h_1^{-1}$.
That is:
$$
    (g_1g_2)^\alpha = h_1h_2.
$$
Since $g_1g_2$ and $h_1h_2$ are positive and $(G,G^+,\Delta)$ is support-preserving, it follows that
$$
\left(G_{\operatorname{Supp}(g_1g_2)}\right)^\alpha = G_{\operatorname{Supp}(h_1h_2)}.
$$
The proof will then finish by showing that $\operatorname{Supp}(g_1g_2)=\operatorname{Supp}(y)$ and that $\operatorname{Supp}(h_1h_2)=\operatorname{Supp}(z)$.

Recall that $g_1g_2=a_{k-1}\cdots a_0b_{k-1}\cdots b_0$ where all factors in this expression are positive elements. Hence $\operatorname{Supp}(g_1g_2)\supset \overline{\operatorname{Supp}(a_0)\cup \operatorname{Supp}(b_0)}=\operatorname{Supp}(y)$. On the other hand, since $y\in G_{\operatorname{Supp}(y)}$ and $(G,G^+,\Delta)$ is an LCM-Garside structure, it follows from \autoref{P:conjugate_to_recurrent_in_G_X} that all elements in $\{\phi^i(y)\}_{i\geq 0}$ belong to $G_{\operatorname{Supp}(y)}$. Hence all positive elements $a_{k-1},\ldots, a_0,b_{k-1},\ldots, b_0$ belong to $G_{\operatorname{Supp}(y)}$. Therefore $\operatorname{Supp}(g_1g_2)\subset \operatorname{Supp}(y)$, and hence $\operatorname{Supp}(g_1g_2) = \operatorname{Supp}(y)$.

The same argument shows that $\operatorname{Supp}(h_1h_2) =\operatorname{Supp}(z)$, and this finally implies that $\left(G_{\operatorname{Supp}(y)}\right)^\alpha\linebreak[1] =G_{\operatorname{Supp}(z)}$, as we wanted to show.
\end{proof}

We recall that the parabolic closure of an element is the unique minimal (for inclusion) parabolic subgroup containing it. We want to show that such a parabolic closure exists for every element in any of the groups we are interested in. Let us first show this for recurrent elements.

\begin{theorem}\label{T:parabolic_closure_for_recurrent}
If $(G,G^+,\Delta)$ is a support-preserving LCM-Garside structure, then every recurrent element $x\in G$ admits a parabolic closure, namely $\PC(x)= G_{\operatorname{Supp}(x)}$.
\end{theorem}

\begin{proof}
Let $x$ be a recurrent element, and let $X=\operatorname{Supp}(x)$. We claim that $G_X$ is the parabolic closure of $x$. It is clear that $G_X$ is parabolic and that $x\in G_X$, hence we only need to show its minimality with respect to inclusion.

Let $H$ be a parabolic subgroup such that $x\in H$. We can assume that $H\neq G$, otherwise it is clear that $G_X\subset H$. Since $H$ is parabolic, it is conjugate by some element $\alpha\in G$ to a proper standard parabolic subgroup $G_Y$, where $\overline{Y}=Y$. It follows that $x^\alpha \in H^\alpha =G_Y$. Since $(G,G^+,\Delta)$ is an LCM-Garside structure, \autoref{P:conjugate_to_recurrent_in_G_X} tells us that $(x^\alpha)^\beta$ is recurrent for some $\beta\in G_Y$. Since $x^\alpha\in G_Y$ it follows that $x^{\alpha\beta}\in G_Y$. If its reduced left-fraction decomposition is $x^{\alpha\beta}=a^{-1}b$, we have that $a,b\in G_Y$, so one can write $a$ and $b$ using atoms from $Y$. Since $Y$ is a saturated set of atoms, it follows that $\operatorname{Supp}(x^{\alpha\beta})\subset Y$. Hence $G_{\operatorname{Supp}(x^{\alpha\beta})}\subset G_Y$.

Now recall that $(G,G^+,\Delta)$ is support-preserving, and that $x$ and $x^{\alpha\beta}$ are recurrent. Hence, by \autoref{P:support_preserved_for_recurrent}:
$$
 (G_X)^{\alpha\beta}=(G_{\operatorname{Supp}(x)})^{\alpha\beta} = G_{\operatorname{Supp}(x^{\alpha\beta})}\subset G_Y.
$$
Therefore:
$$
  G_X\subset (G_Y)^{\beta^{-1}\alpha^{-1}}=(G_Y)^{\alpha^{-1}}=H.
$$
This shows the minimality of $G_X$, hence $G_X=G_{\operatorname{Supp}(x)}$ is the parabolic closure of $x$.
\end{proof}

Now the existence of the parabolic closure of an arbitrary element will be a consequence of the following result:

\begin{lemma}\label{L:parabolic_closure_for_conjugates}
Let $x,c\in G$. If $x$ admits a parabolic closure, so does its conjugate $x^c$. Namely, $\PC(x^c)=\PC(x)^c$.
\end{lemma}

\begin{proof}
Let $\PC(x)$ be the parabolic closure of $x$. We need to show that $\PC(x)^c$ is the parabolic closure of $x^c$.

First notice that $\PC(x)^c$ is parabolic, as it is the conjugate of a parabolic subgroup. Notice also that $x^c\in \PC(x)^c$, as $x\in \PC(x)$. Finally, suppose that $H$ is a parabolic subgroup containing $x^c$. Then $x^c\in H$ implies $x\in H^{c^{-1}}$, where $H^{c^{-1}}$ is a parabolic subgroup (being conjugate to $H$). Hence, by minimality of $\PC(x)$ we obtain $\PC(x)\subset H^{c^{-1}}$, and then $\PC(x)^c\subset H$. Therefore $\PC(x)^c$ is the parabolic closure of $x^c$, that is, $\PC(x^c)=\PC(x)^c$.
\end{proof}

\begin{theorem}\label{T:parabolic_closure_exists}
If $(G,G^+,\Delta)$ is a support-preserving LCM-Garside structure, then every element of $G$ admits a parabolic closure.
\end{theorem}

\begin{proof}
Let $x\in G$. Applying iterated swaps, one can conjugate $x$ to a recurrent element $y$. By \autoref{T:parabolic_closure_for_recurrent} $y$ admits a parabolic closure, and by \autoref{L:parabolic_closure_for_conjugates} so does $x$.
\end{proof}

The hypotheses of \autoref{T:parabolic_closure_exists} are satisfied by some well-known Garside groups, namely Artin groups of spherical type, but the existence of parabolic closures in these groups was already shown in~\cite{CGGW}.

We can also extend the results in~\cite{CGGW} concerning parabolic closures of powers of elements:

\begin{theorem}\label{T:parabolic_closure_of_powers}
If $(G,G^+,\Delta)$ is a support-preserving LCM-Garside structure, given $x\in G$ and $m$ a nonzero integer, the parabolic closures of $x$ and $x^m$ coincide.
\end{theorem}

\begin{proof}
The proof is basically the same as that of Theorem~8.2 in~\cite{CGGW}, replacing $RSSS_{\infty}(x)$ with $\RC(x)$.

The idea is to assume, up to conjugation, that $x\in \RC(x)$. Then consider the pair $(x,x^m)$. The conjugating element which applies a swap to $x^m$ is the inverse of its left-denominator, which is precisely $x^m\wedge 1$. Since both $x^m$ and $1$ conjugate $x$ to itself, it follows by \autoref{P:convexity} that $x^m\wedge 1$ conjugates $x$ to an element in $\RC(x)$. Hence we can conjugate the pair $(x,x^m)$ by $x^m\wedge 1$, to apply $\phi$ to the second coordinate while keeping the first coordinate inside $\RC(x)$. Iterating, we can assume that both $x$ and $x^m$ are recurrent elements.

Now, if $x$ is either positive or negative, one checks immediately that $G_{\operatorname{Supp}(x)}=G_{\operatorname{Supp}(x^m)}$ is the parabolic closure of both $x$ and $x^m$. If $x$ is neither positive nor negative, one considers the Garside structure $(G,G^+,\Delta^N)$, with $N$ big enough so that the left-numerator and the left-denominator of $x$ are simple. In this case, the swap operation is equivalent to the twisted cycling operation, implying that $x$ belongs to its ultra summit set. Then one can follow the arguments in~\cite[Theorem 8.2]{CGGW} to conclude that $G_{\operatorname{Supp}(x)}=G_{\operatorname{Supp}(x^m)}$ also in this case, hence the parabolic closures of $x$ and $x^m$ coincide.
\end{proof}

In the same way as it is done in~\cite{CGGW} for Artin-Tits groups of spherical type, we can also conclude that all roots of an element in a parabolic subgroup belong to the parabolic subgroup:

\begin{corollary}
Let $(G,G^+,\Delta)$ be a support-preserving LCM-Garside structure. If $y$ belongs to a parabolic subgroup $H$, and $x\in G$ is such that $x^m=y$ for some nonzero integer $m$, then $x\in H$.
\end{corollary}

\begin{proof}
This is the same proof as in~\cite[Corollary 8.3]{CGGW}. By \autoref{T:parabolic_closure_of_powers}, we have $\PC(x)=\PC(y)$. Since $y\in H$ and $H$ is parabolic, it follows that $\PC(y)\subset H$. But then $x\in \PC(x)=\PC(y)\subset H$.
\end{proof}

Later we will show that other known Garside groups, apart from Artin-Tits groups of spherical type, also satisfy the hypotheses of \autoref{T:parabolic_closure_exists} and \autoref{T:parabolic_closure_of_powers}. Before that, in the next section, we will provide the technical tools that will allow us to achieve this goal.

\subsection{Checking properties of a Garside structure}\label{S:checking_Garside_structure}

In this section we will explain how one can check whether a Garside structure is LCM, and whether an LCM-Garside structure is support-preserving.

In order to show that a Garside structure is LCM, there are three properties to prove (recall \autoref{D:LCM-Garside_structure}). The first two properties have a finite number of checkings (check whether $\Delta=\Delta_{\mathcal A}$ and, on the other hand, compute the elements of the form $\Delta_X$ for $X\subset \mathcal A$ and check whether their prefixes and their suffixes coincide). The third property can be verified thanks to a result from~\cite{GODELLE2007}, which can be applied to the elements of the form $\delta=\Delta_X$ with $X\subset \mathcal A$.

\begin{proposition}\label{P:characterization_standard_parabolic}\cite[Proposition 1.16]{GODELLE2007}
Let $(G,G^+,\Delta)$ be a Garside structure, and let $\delta\in \operatorname{Div}(\Delta)$ be a balanced element. Then $G_{\delta}$ is a standard parabolic subgroup (hence $(G_{\delta},G_{\delta}^+,\delta)$ is a Garside structure) if and only if, for every $x,y\in \operatorname{Div}(\delta)$, one has:
$$
   xy\wedge \Delta \in \operatorname{Div}(\delta) \qquad \mbox{and} \qquad xy\wedge^{\Lsh} \Delta \in \operatorname{Div}(\delta).
$$
\end{proposition}

Now let us assume that $(G,G^+,\Delta)$ is an LCM-Garside structure. We will explain how to determine whether the structure is support-preserving.

Given an element $x\in G$ which is conjugate to a positive element, one can compute the set $C^+(x)$ in two steps. First one finds one element $y\in C^+(x)$ by applying iterated swaps to $x$. Then, starting with $y$, one computes the directed graph $\mathcal G_{C^+(x)}$, defined as follows:
\begin{itemize}

\item The vertices of $\mathcal G_{C^+(x)}$ correspond to the elements of $C^+(x)$.

\item There is an arrow labeled $g$ with source $u$ and target $v$ if and only if $g$ is non-trivial and positive, $u^g=v$ and $u^h\notin C^+(x)$ whenever $1\prec h \prec g$.
\end{itemize}

The arrows in $\mathcal G_{C^+(x)}$ are called {\em minimal positive conjugators}. If the arrow starts at a vertex $u$, it is called a {\em minimal positive conjugator for $u$}. Since the graph $\mathcal G_{C^+(x)}$ is connected, one can compute the whole graph, starting with a single element $y\in C^+(x)$, provided that one knows how to compute the minimal positive conjugators for any given element.

A crucial property of these conjugating elements is that we just need to check that they are support-preserving, in order to show that an LCM-Garside structure is support-preserving:

\begin{proposition}\label{P:support_preserved_by_minimal_positive_conjugators}
An LCM-Garside structure $(G,G^+,\Delta)$ is support-preserving if and only if for every positive element $x\in G^+$, and every minimal positive conjugator $c$ for $x$ in $\mathcal G_{C^+(x)}$, one has:
$$
    \left(G_{\operatorname{Supp}(x)}\right)^c=G_{\operatorname{Supp}(x^c)}.
$$
\end{proposition}

\begin{proof}
The only if statement is trivial, so let us assume that the support in $G$ is preserved under conjugation by minimal positive conjugators. Let $x,y$ be a pair of conjugate positive elements and let $\alpha\in G$ such that $x^\alpha=y$.

If $\alpha$ is not positive, we know from \autoref{P:turn_into_positive} that $\beta\alpha$ is positive for some central element $\beta\in G$. Then
$$
\left(G_{\operatorname{Supp}(x)}\right)^\alpha =  \left(\left(G_{\operatorname{Supp}(x)}\right)^{\beta^{-1}}\right)^{\beta \alpha} = \left(G_{\operatorname{Supp}(x)}\right)^{\beta \alpha},
$$
and $x^{\beta \alpha}=x^\alpha=y$. Hence, replacing $\alpha$ with $\beta\alpha$ if necessary, we can assume that $\alpha$ is positive.

Now we claim that every positive element $\alpha$ such that $x^\alpha$ is positive can be decomposed as a product of minimal positive conjugators, meaning that there is a path in $\mathcal G_{C^+(x)}$, starting at $x$, corresponding to a sequence of minimal positive conjugators whose product is $\alpha$.

If $\alpha$ is either trivial or a minimal positive conjugator, the claim clearly holds. Otherwise, since $\alpha$ is not minimal, it can be decomposed as $\alpha=a_1b_1$, where $a_1$ and $b_1$ are positive (so $1\prec a_1\prec \alpha$) and $x^{a_1}\in C^+(x)$. If $a_1$ is not a minimal positive conjugator, we keep going and find $a_2$ such that $1\prec a_2\prec a_1 \prec \alpha$ and $x^{a_2}\in C^+(x)$. Since we cannot have an infinite descending chain of positive elements in a Garside group, this process must stop, and we will obtain a minimal positive conjugator $\alpha_1$ for $x$, with $\alpha=\alpha_1\beta_1$ and $\beta_1\in G^+$. If $\beta_1$ is not trivial, we apply the same reasoning to $\beta_1$ and find a minimal positive conjugator $\alpha_2$ for $x^{\alpha_1}$, such that $\alpha=\alpha_1\alpha_2\beta_2$ and $\beta_2\in G^+$. This process must also terminate, since a positive element cannot be decomposed as an arbitrarily large product of positive elements. So some $\beta_r$ will be trivial, and $\alpha=\alpha_1\cdots \alpha_r$ will be a product of minimal positive conjugators, starting at $x$. This shows the claim.

Finally, since we are assuming that the support is preserved under conjugation by minimal positive conjugators, it follows that the support will be preserved by $\alpha$:
$$
 \left(G_{\operatorname{Supp}(x)}\right)^\alpha = \left(G_{\operatorname{Supp}(x)}\right)^{\alpha_1\alpha_2\cdots \alpha_r} =
 \left(G_{\operatorname{Supp}(x^{\alpha_1})}\right)^{\alpha_2\cdots \alpha_r} = \cdots = G_{\operatorname{Supp}(x^{\alpha_1\cdots \alpha_r})} = G_{\operatorname{Supp}(y)}
$$
\end{proof}

We will therefore be interested in describing minimal simple elements in detail. Given $x\in G^+$, we want to compute the arrows of the graph $\mathcal G_{C^+(x)}$ starting at the vertex $x$. This is explained in~\cite{FRANCOGM}, and we will also describe here how to do it, as we will perform this computation in the following subsections.

Let us start by recalling the following result, which is a particular case of \autoref{P:convexity}, and also holds for other sets like super summit sets, ultra summit sets, sliding circuits set and the like, although we will only use it in this case.

\begin{lemma}\label{L:convexity}\cite[Proposition 4.8]{FRANCOGM} Let $G$ be a Garside group, let $x\in G^+$. For every $u,v\in G$, if $x^u,x^v\in C^+(x)$ then $x^{u\wedge v}\in C^+(x)$.
\end{lemma}

Now recall that $x\in G^+$, and that we want to compute the minimal positive conjugators for $x$. Take an atom $a\in G$, that is, a positive element which admits no nontrivial positive prefix. The set
$$
\mathcal C_a(x)=\{\alpha\in G;\ a\preccurlyeq \alpha \textrm{ and } x^\alpha\in C^+(x) \}=\{\alpha\in G;\ a\preccurlyeq \alpha \textrm{ and } x^\alpha\in G^+\}
$$
is nonempty ($\Delta\in \mathcal C_a(x)$), and is closed under $\wedge$ by \autoref{L:convexity}. Hence, as it is formed by positive elements, and every $\preccurlyeq$-chain of positive elements must have a minimal element, it follows that $\mathcal C_a(x)$ has a unique $\preccurlyeq$-minimal element, that we denote $\rho_{a}(x)$.

Every minimal positive conjugator starting at $x$ must have an atom as a prefix, so it must be equal to $\rho_a(x)$ for some atom $a$. Actually, the set of minimal positive conjugators starting at $x$ is precisely the set of $\preccurlyeq$-minimal elements in the set $\mathcal M(x)=\{\rho_a(x); \ a\in \mathcal A\}$, where $\mathcal A$ is the set of atoms of~$G$. Since the set $\mathcal M(x)$ is finite, we can compute the arrows starting at $x$ if we are able to compute $\rho_a(x)$ for every atom $a$. Let us see how we can do it.

We will make extensive use of diagrams of the following kind, that we will call LCM-diagrams:
$$
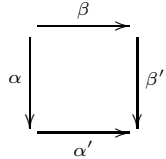
\xymatrix@C=12mm@R=12mm{
  \ar[d]_{\alpha} \ar[r]^{\beta}
& \ar[d]^{\beta'}
\\
  \ar[r]_{\alpha'}
&
}
$$
In this diagram, $\alpha$ and $\beta$ are positive elements, $\alpha \alpha'=\beta\beta'$, and this is precisely the least common multiple $\alpha\vee\beta$. Concatenating several LCM-diagrams produce a new LCM-diagram, if one only reads the product of the arrows in the perimeter of the diagram (we will see many examples below). The element $\alpha'$ is usually denoted $\alpha\backslash \beta$, and it is called the {\em right-complement of $\alpha$ in $\beta$}. So we have $\alpha \backslash \beta = \alpha^{-1}(\alpha\vee \beta)$ and, similarly, $\beta\backslash \alpha = \beta^{-1} (\beta\vee \alpha)$. We can then express every LCM-diagram as follows:
$$
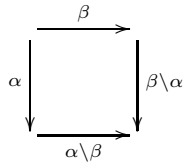
\xymatrix@C=12mm@R=12mm{
  \ar[d]_{\alpha} \ar[r]^{\beta}
& \ar[d]^{\beta\backslash \alpha}
\\
  \ar[r]_{\alpha\backslash \beta}
&
}
$$

Let us then explain how to compute the element $\rho_a(x)$ for some atom $a$ and some positive element $x$. The element $\rho_a(x)$ is a positive element $c$ such that $x^c$ is positive, that is, $c^{-1}xc=y\in G^+$ or, equivalently, $xc=cy$ for some $y\in G^+$. Since $a\preccurlyeq c \preccurlyeq cy$, this implies that $a\preccurlyeq xc$. At the same time, $xa\preccurlyeq xc$ since $a\preccurlyeq c$. Therefore $a\vee xa \preccurlyeq xc$. Let us compute $a\vee xa$ and write it as $xc_1$ for some positive element $c_1$. We have $xc_1 = xa \vee a = xa \vee x \vee a$. Left-multiplying by $x^{-1}$ we obtain $c_1=a \vee x^{-1}(x\vee a)$, that is, $c_1= a \vee x\backslash a$. Notice that we have $xa\preccurlyeq xc_1 \preccurlyeq xc$, hence $a\preccurlyeq c_1\preccurlyeq c$. Then we can apply the same reasoning, computing $c_2=c_1\vee x\backslash c_1$ and obtaining $a\preccurlyeq c_1\preccurlyeq c_2\preccurlyeq c$. And so on.

In this way we get a sequence
$$
 a=c_0\preccurlyeq c_1\preccurlyeq c_2\preccurlyeq \cdots
$$
of positive prefixes of $c=\rho_a(x)$, where $c_{i+1}=c_i\vee x\backslash c_i$ for every $i\geq 0$. We will call them the {\em converging prefixes} of $\rho_a(x)$. The sequence of converging prefixes must stabilize, and it is shown in~\cite{FRANCOGM} that $\rho_a(x)=c_m$ for the smallest $m$ such that $c_m=c_{m+1}$, that is, we obtain $\rho_a(x)$ at the place in which the chain stabilizes.

Notice that all elements in the above chain are simple, since $\rho_a(x)\preccurlyeq \Delta$ (as $\Delta\in \mathcal C_a(x)$). Hence, in order to compute $\rho_a(x)$ we just need to know how to compute $s\vee \alpha$ for some simple element $s$ and some positive element $\alpha$. If we write $\alpha=a_1\cdots a_r$ as a product of atoms, this computation is performed by starting with the following diagram:
$$
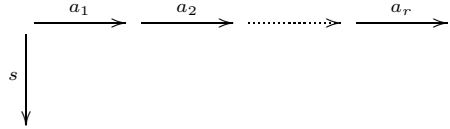
\xymatrix@C=12mm@R=12mm{
  \ar[r]^{a_1} \ar[d]_{s}
& \ar[r]^{a_2}
& \ar@{.>}[r]
& \ar[r]^{a_r}
&
\\
& & &
&
}
$$
and filling the squares, from left to right, to obtain a concatenation of LCM-diagrams:
$$
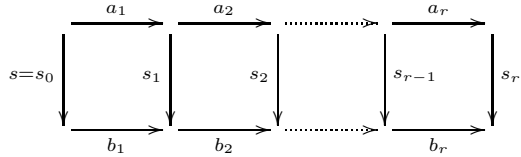
\xymatrix@C=12mm@R=12mm{
  \ar[r]^{a_1} \ar[d]_{s=s_0}
& \ar[r]^{a_2} \ar[d]_{s_1}
& \ar@{.>}[r]  \ar[d]_{s_2}
& \ar[r]^{a_r} \ar[d]^{s_{r-1}}
& \ar[d]^{s_r}
\\
  \ar[r]_{b_1}
& \ar[r]_{b_2}
& \ar@{.>}[r]
& \ar[r]_{b_r}
&
}
$$

The top row represents the element $\alpha=a_1\cdots a_r$. Let us denote $\beta=b_1\cdots b_r$. Each square is an LCM-diagram, so $s_{i-1}\vee a_{i} = s_{i-1}b_i=a_is_i$ for $i=1,\ldots,r$. Since the concatenation of LCM-diagrams is an LCM-diagram, we obtain that $s\vee \alpha = s\beta = \alpha s_r$. In other words, $s_r=\alpha\backslash s$.

This means that in order to compute $s\vee \alpha$, we just need to know how to compute $t\vee a$ for any simple element $t$ and any atom $a$. This is something that we should know how to compute, when working with a Garside group $G$. Since the set of simple elements and the set of atoms are finite, we may even have that information precomputed and stored.

Therefore, in order to compute the element $\rho_a(x)$ as explained above, and assuming that $x=a_1\cdots a_r$ as a product of atoms, one starts with the initial converging prefix $c_0=a$ and, for every $j\geq 0$, one computes a sequence of right-complements, from $c_j$ to $x\backslash c_j$, by filling the following LCM-diagram:
$$
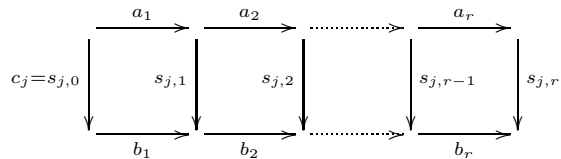
\xymatrix@C=12mm@R=12mm{
  \ar[r]^{a_1} \ar[d]_{c_j=s_{j,0}}
& \ar[r]^{a_2} \ar[d]_{s_{j,1}}
& \ar@{.>}[r]  \ar[d]_{s_{j,2}}
& \ar[r]^{a_r} \ar[d]^{s_{j,r-1}}
& \ar[d]^{s_{j,r}}
\\
  \ar[r]_{b_1}
& \ar[r]_{b_2}
& \ar@{.>}[r]
& \ar[r]_{b_r}
&
}
$$
By construction, for $i=1,\ldots,r$ we have $s_{j,i}=(a_1\cdots a_i)\backslash c_j$, so $s_{j,r}=x\backslash c_j$. Hence, since $c_{j+1}=c_j \vee x\backslash c_j = c_j\vee s_{j,r}$, one can add one more square to the above LCM-diagram, as follows:
$$
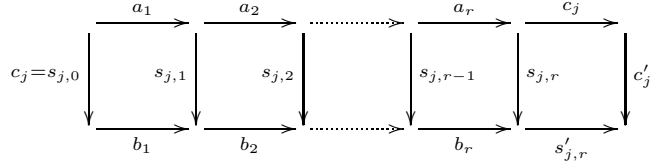
\xymatrix@C=12mm@R=12mm{
  \ar[r]^{a_1} \ar[d]_{c_j=s_{j,0}}
& \ar[r]^{a_2} \ar[d]_{s_{j,1}}
& \ar@{.>}[r]  \ar[d]_{s_{j,2}}
& \ar[r]^{a_r} \ar[d]^{s_{j,r-1}}
& \ar[r]^{c_j} \ar[d]^{s_{j,r}}
& \ar[d]^{c_j'}
\\
  \ar[r]_{b_1}
& \ar[r]_{b_2}
& \ar@{.>}[r]
& \ar[r]_{b_r}
& \ar[r]_{s_{j,r}'}
&
}
$$
Then we have that $c_{j+1}=c_j c_j'$, which is the product of the top and right arrows in the last squared diagram.

This is repeated until one finds $c_m=c_{m+1}$ and then $\rho_a(x)=c_m$. For every $j=0,\ldots,m-1$ and every $i=0,\ldots,r$, the elements $s_{j,i}=(a_1\cdots a_i)\backslash c_j$ will be called the {\em pre-minimal conjugators} for $a$ and $x$. Notice that the pre-minimal conjugators are not necessarily prefixes of $\rho_a(x)$, but the converging prefixes $c_0,c_1,\ldots,c_m$ are.

Therefore, by computing the pre-minimal conjugators and the converging prefixes, one can compute the element $\rho_a(x)$ for every atom $a$ and every positive element $x$.

In the forthcoming subsections we will be able to describe the elements $\rho_a(x)$, in some cases, by performing a detailed study of the pre-minimal conjugators and the converging prefixes which are computed during the process. In other cases we will show that the element $\rho_a(x)$ that one would obtain is not a minimal positive conjugator, so there is no need to compute it.

There is a situation which occurs with all the monoids under study. Let us see that, in order to show that an LCM-Garside structure is support-preserving, we just need to care about the elements $\rho_a(x)$, where $a\notin \operatorname{Supp}(x)$.

\begin{proposition}\label{P:atom_in_X_support-preserving}
Let $(G,G^+,\Delta)$ be an LCM-Garside structure. Suppose that for every $u\in G^+$ and every $b\notin \operatorname{Supp}(u)$ such that $\rho_b(u)$ is a minimal positive conjugator, $\#(\operatorname{Supp}(u^{\rho_b(u)}))\leq \#(\operatorname{Supp}(u))$. Then, for every $x\in G^+$ and every $a\in \operatorname{Supp}(x)$, one has $\operatorname{Supp}(x^{\rho_a(x)})=\operatorname{Supp}(x)$ and 
$$
\left(G_{\operatorname{Supp}(x)}\right)^{\rho_a(x)}=G_{\operatorname{Supp}(x)}
=G_{\operatorname{Supp}(x^{\rho_a(x)})}.
$$ 
\end{proposition}

\begin{proof}
Let $X=\operatorname{Supp}(x)$ and let $y=x^{\rho_a(x)}$ and $Y=\operatorname{Supp}(y)$. We want to prove that $X=Y$ and that $(G_X)^{\rho_a(x)}=G_X=G_Y$. We first notice that $\Delta_X$ is a positive element which admits $a$ as a prefix, and which conjugates $x$ to a positive element. Hence $\Delta_X\in \mathcal C_a(x)$. Since $\rho_a(x)$ is the minimal element in this set, it follows that $\rho_a(x)\preccurlyeq \Delta_X$. Hence $\rho_a(x)\in G_X^+$. This implies that $y=x^{\rho_a(x)}\in G_X^+$, hence $Y=\operatorname{Supp}(y) \subset X$.

From the above paragraph and the hypothesis, it follows that for every $u\in G^+$ and every atom $b\in \mathcal A$ such that $\rho_b(u)$ is a minimal positive conjugator,  $\#(\operatorname{Supp}(u^{\rho_b(u)}))\leq \#(\operatorname{Supp}(u))$.

Let us go back to $x$ and $y$, such that $x^{\rho_a(x)}=y$ with $a\in X$. We know that there is some positive power $(\Delta_X)^e$ which is central in $G_X$. If we denote $\alpha=\rho_a(x)^{-1}(\Delta_X)^e$, we see that $\alpha$ is a positive element such that $y^{\alpha}=x$. Now $\alpha$ can be decomposed as a product of minimal positive conjugators, and the conjugation by each one cannot increase the number of elements in the corresponding support. This implies that $\#(X)\leq \#(Y)$. Since $X$ and $Y$ are finite sets, and $Y\subset X$, it follows that $X=Y$.

Finally, since $\rho_a(x)\in G_X^+$, conjugation by $\rho_a(x)$ is an inner automorphism of $G_X$, so it conjugates $G_X$ to itself. In other words, $\rho_a(x)$ preserves the support.
\end{proof}

We finish this subsection with some useful results for computing the set of minimal positive conjugators for some $y\in G^+$.

\begin{lemma}\label{L:pre-minimal_prefixes}
Let $x\in G^+$ and let $a$ be an atom. Write $x=a_1\cdots a_r$ as a product of atoms. If for every $j=0,\ldots,m-1$ and every $i=0,\ldots,r-1$ we have $s_{j,i}\preccurlyeq s_{j,i+1}$, that is, if every pre-minimal conjugator is a prefix of the next one, then $c_{j+1}=s_{j,r}$ for every $j=0,\ldots,m-1$, and $\rho_a(x)=x^{m}\backslash a$.
\end{lemma}

\begin{proof}
By hypothesis, for every $j=0,\ldots,m-1$ we have $c_j=s_{j,0}\preccurlyeq s_{j,1} \preccurlyeq \cdots \preccurlyeq s_{j,r} = x\backslash c_j$. Hence, as $c_{j+1}=c_j\vee x\backslash c_j$, it follows that $c_{j+1}=x\backslash c_j=s_{j,r}$.

The above property implies that, in order to compute new pre-minimal conjugators, one does not need to compute $c_j\vee x\backslash c_j$. Hence one can concatenate LCM-diagrams as follows, to obtain all converging prefixes for $\rho_a(y)$:
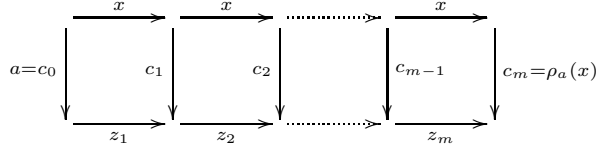
$$
\xymatrix@C=12mm@R=12mm{
  \ar[d]_{a=c_0} \ar[r]^{x}
& \ar[d]_{c_1} \ar[r]^{x}
& \ar[d]_{c_2} \ar@{.>}[r]
& \ar[d]^{c_{m-1}} \ar[r]^{x}
& \ar[d]^{c_m=\rho_a(x)}
\\
  \ar[r]_{z_1}
& \ar[r]_{z_2}
& \ar@{.>}[r]
& \ar[r]_{z_m}
&
}
$$
Hence $\rho_a(x)=x^m\backslash a$, as we wanted to show.
\end{proof}

\begin{lemma}\label{L:unnecessary_arrow}
Let $x\in G^+$ and let $a$ and $b$ be atoms of $G$. Suppose that $a\not\preccurlyeq \rho_b(x)$. If there is some converging prefix $c_i$ for $\rho_a(x)$ such that $b\preccurlyeq c_i$, then $\rho_a(x)$ is not a minimal positive conjugator for $x$.
\end{lemma}

\begin{proof}
We know from~\cite{FRANCOGM} that $\rho_a(x)=c_m$, where $m$ is the place in which the chain of converging prefixes stabilizes. This means that $c_i\preccurlyeq c_m=\rho_a(x)$, and hence $b\preccurlyeq \rho_a(x)$.

Since $x^{\rho_a(x)}$ is positive by definition of $\rho_a(x)$, it follows that $\rho_a(x)\in \mathcal C_b(x)$. But $\rho_b(x)$ is the $\preccurlyeq$-minimal element in this set, so $\rho_b(x)\preccurlyeq \rho_a(x)$.

Now notice that $\rho_b(x)\neq \rho_a(x)$, since $a\preccurlyeq \rho_a(x)$ by definition and $a\not\preccurlyeq \rho_b(x)$ by hypothesis. Therefore $1\prec \rho_b(x)\prec \rho_a(x)$. Since $x^{\rho_b(x)}$ is positive, this implies that $\rho_a(x)$ is not a minimal positive conjugator for $x$.
\end{proof}

\begin{lemma}\label{L:unnecessary_pre-minimal_prefixes}
Let $x\in G^+$ and let $a$ and $b$ be atoms of $G$. Suppose that $a\not\preccurlyeq \rho_b(x)$. Write $x=a_1\cdots a_r$ as a product of atoms, and suppose that for every $j=0,\ldots,m-1$ and every $i=0,\ldots,r-1$ we have $s_{j,i}\preccurlyeq s_{j,i+1}$. If there is some pre-minimal conjugator $s_{j,i}$ for $a$ and $x$ such that $b\preccurlyeq s_{j,i}$, then $\rho_a(x)$ is not a minimal positive conjugator for $x$.
\end{lemma}

\begin{proof}
If $b\preccurlyeq s_{j,i}$ for some pre-minimal conjugator $s_{j,i}$, by hypothesis we know that $s_{j,i}\preccurlyeq s_{j,i+1}\preccurlyeq \cdots \preccurlyeq s_{j,r}$, so $b\preccurlyeq s_{j,r}$. Also, since every pre-minimal conjugator is a prefix of the next one, from \autoref{L:pre-minimal_prefixes} we know that $s_{j,r}=c_{j+1}$. Hence $b\preccurlyeq c_{j+1}$. Therefore, by \autoref{L:unnecessary_arrow}, $\rho_a(x)$ is not a minimal positive conjugator for $x$.
\end{proof}

\section{Parabolic closures for complex braid groups}
\label{sect:parabclosurecbg}

\subsection{The classical braid group}

We will start our study of particular complex braid groups by the case of the classical braid group $\mathcal{B}_{n+1}$
on $n+1$ strands. Since it is the same as the Artin group of type $A_n$,
the results are already known. We nevertheless expose our new
methods in detail for this case,
as it will serve as
a guide for the other groups, in which the same strategies will be used.

Let $G = \mathcal{B}_{n+1}$ endowed with classical Garside structure, for which the monoid of positive elements is given by the presentation with generators $\mathcal A=\{s_1,\ldots,s_n\}$ (which are the atoms of this structure) and relations
\begin{itemize}
  \item $s_is_js_i=s_js_is_j$ \quad if $|i-j|=1$,
  \item $s_is_j=s_js_i$ \qquad if $|i-j|>1$.
\end{itemize}

As usual, one can represent the braids in $G$ as a collection of $n+1$ disjoint strands, up to homotopy fixing the endpoints, and the generator $s_i$ corresponds to a braid in which the strands $i$ and $i+1$ cross once (positively), and there is no other crossing. The simple elements in this structure corresponds to braids in which every two strands cross (positively) at most once, and the Garside element $\Delta=\Delta_{\mathcal A}$ is the braid in which every two strands cross (possitively) exactly once.

For every $X\subset \mathcal A$, the subgroup generated by $X$ is a direct product of classical braid groups. The element $\Delta_X$ can be seen geometrically as well. Namely, if all elements in $X$ are consecutive generators, $X=\{s_i,s_{i+1},\ldots,s_{j}\}$, then $\Delta_X$ is the braid which every two strands in $\{i,i+1,\ldots,j+1\}$ cross (positively) exactly once:
$$
  \Delta_X=s_j (s_{j-1}s_j)\cdots (s_is_{i+1}\cdots s_j).
$$
Otherwise, if one considers that consecutive generators in $\{s_1,\ldots,s_n\}$ are connected, and one decomposes $X$ in connected components $X=X_1\cup\cdots \cup X_r$, where $X_k=\{s_{i_k},s_{i_k+1},\ldots,s_{j_k}\}$ for some $i_k\leq j_k$, then $\Delta_X=\Delta_{X_1}\cdots \Delta_{X_r}$, where the factors are pairwise commuting. It is well known that in every case $X\subset \mathcal A$ is saturated, as the set of atoms which are prefixes of $\Delta_X$ is precisely $X$. Moreover, $G_X$ is a standard parabolic subgroup of $G$. Therefore, the classical Garside structure of $G$ is an LCM-Garside structure.

Now we need to introduce some special elements, which are analogous to the \emph{elementary ribbons} defined by Godelle in~\cite{GODELLE2003}. We will do it for an arbitrary Garside group, as in~\cite{CGGW}.

\begin{definition}\label{D:ribbon}
Let $X\subset \mathcal A$ be a subset of atoms in a Garside group, and let $u\in \mathcal A$. Then we define $r_{X,u}=\Delta_X^{-1}\Delta_{X\cup \{u\}}$.
\end{definition}

\begin{lemma}\label{L:ribbons_are_simple}
Let $X\subset \mathcal A$ be a subset of atoms in a Garside group, and let $u\in \mathcal A$. If $u\in \overline{X}$ then $r_{X,u}=1$. Otherwise $r_{X,u}$ is a nontrivial simple element.
\end{lemma}

\begin{proof}
If $u\in \overline{X}$ then $\overline{X}=\overline{X\cup\{u\}}$. Hence $\Delta_{X\cup\{u\}}=\Delta_{\overline{X\cup\{u\}}}=\Delta_{\overline{X}}=\Delta_X$, and then $r_{X,u}=1$.

Let us suppose that $u\notin \overline{X}$. Recall that $\Delta_{X\cup \{u\}}$ is the join of all elements in $X\cup \{u\}$, hence all elements in $X$ are prefixes of $\Delta_{X\cup \{u\}}$, which implies that $\Delta_X\preccurlyeq \Delta_{X\cup \{u\}}$. Hence $r_{X,u}=\Delta_{X}^{-1}\Delta_{X\cup \{u\}}$ is a positive element. Since $r_{X,u}$ is a suffix of $\Delta_{X\cup \{u\}}$, it is a suffix of the Garside element $\Delta$, and hence it is a simple element. Now, since $\Delta_X=\Delta_{\overline{X}}$, an atom is a prefix of $\Delta_X$ if and only if it belongs to $\overline{X}$. Hence $u\not\preccurlyeq \Delta_X$. On the other hand $u\preccurlyeq \Delta_{X\cup \{u\}}$ by definition. Therefore $\Delta_X\neq \Delta_{X\cup\{u\}}$ and $r_{X,u}$ is nontrivial.
\end{proof}

We will show that if $x,y\in G^+$ are two positive braids such that $\operatorname{Supp}(x)=X\neq Y=\operatorname{Supp}(y)$, and that $x^\alpha=y$ for some minimal positive conjugator $\alpha$, then $\alpha$ is a ribbon, namely $\alpha=r_{X,a}$ for some atom $a$. This is actually shown in~\cite[Proposition 6.3]{CGGW} for every Artin group of spherical type, but we will show it here in a different way, which will be helpful in the following sections, when we will study other Garside groups.

Given $1\leq i\leq k\leq n$, we denote:
$$
   S_{i,k}=s_is_{i+1}\cdots s_k, \qquad S_{k,i}=s_ks_{k-1}\cdots s_i.
$$
The former braid can be drawn as the $i$-th strand crossing positively the strands $i+1,i+2,\ldots,k+1$ in that precise order. The latter, as the $(k+1)-$st strand crossing positively the strands $k,k-1,\ldots,i$.

Given $1\leq i\leq j\leq k\leq n$, we can define the braid $\sigma_{i,j,k}$ in which every strand from $\{i,\ldots, j\}$ crosses positively every strand in $\{j+1,\ldots,k+1\}$ exactly once, and there is no other crossing. Hence, it is a simple braid, and it is easy to see that:
$$
   \sigma_{i,j,k}=S_{j,k}S_{j-1,k-1}\cdots S_{i,k-j+i} = S_{j,i}S_{j+1,i+1}\cdots S_{k,i+k-j}.
$$
We will need the following:

\begin{lemma}\label{L:ribbon_prefixes_in_braids}
Given $1\leq i_1\leq i_0 \leq j \leq k_0 \leq k_1\leq n$, one has $\sigma_{i_0,j,k_0}\preccurlyeq \sigma_{i_1,j,k_1}$.
\end{lemma}

\begin{proof}
It is well known that a simple braid $s_0$ is a prefix of a simple braid $s_1$ if and only if the pairs of strands that cross in $s_0$ also cross in $s_1$ (see~\cite{EPSTEINETAL}). The result follows immediately, as every $\sigma_{i,j,k}$ is a simple braid, and we know exactly the pairs of strands that cross in such an element.
\end{proof}

Recall also that one has:
$$
  \Delta_{\{s_i,\ldots,s_{j-1}\}\cup \{s_{j+1},\ldots,s_k\}}= \Delta_{\{s_i,\ldots,s_{j-1}\}}\Delta_{\{s_{j+1},\ldots,s_k\}}.
$$
This also holds when either $i=j$ or $j=k$, that is, when one of the above factors is trivial.

Now, by checking the strands that cross in the corresponding braids, one can easily see that:
$$
 \Delta_{\{s_i,\ldots,s_{j-1}\}\cup \{s_{j+1},\ldots,s_k\}}\sigma_{i,j,k}= \Delta_{\{s_i,\ldots,s_k\}}.
$$
Therefore, we have shown the following:

\begin{lemma}\label{L:ribbons_in_braid_groups}
Let $1\leq i\leq j \leq k\leq n$, and let $X=\{s_i,\cdots s_{j-1}\}\cup \{s_{j+1},\ldots,s_k\}$. Then one has:
$$
   r_{X,s_{j}}=\sigma_{i,j,k}.
$$
\end{lemma}

From this, we have a complete description of nontrivial ribbon elements:

\begin{lemma}\label{L:ribbons_in_braid_groups_complete}
Let $X\subset \{s_1,\ldots,s_n\}$ and $s_j\notin X$. Decompose $X=X_0\sqcup X_1$, where $X_0\cup\{s_j\}=\{s_{i_0},\ldots,s_{k_0}\}$ is the connected component of $X\cup \{s_j\}$ containing $s_j$. Then
$$
   r_{X,s_j}=r_{X_0,s_j}=\sigma_{i_0,j,k_0}.
$$
\end{lemma}

\begin{proof}
Since every atom of $X_0\cup\{s_j\}$ commutes with every atom in $X_1$, we have $\Delta_X=\Delta_{X_0}\Delta_{X_1}= \Delta_{X_1}\Delta_{X_0}$ and $\Delta_{X\cup \{s_j\}}=\Delta_{X_0\cup \{s_j\}}\Delta_{X_1} = \Delta_{X_1}\Delta_{X_0\cup \{s_j\}}$. Hence
$$
  r_{X,s_j}=\Delta_{X}^{-1}\Delta_{X\cup \{s_j\}} = \Delta_{X_0}^{-1}\Delta_{X_1}^{-1} \Delta_{X_1} \Delta_{X_0\cup\{s_j\}} = \Delta_{X_0}^{-1}\Delta_{X_0\cup\{s_j\}} = r_{X_0,s_j}.
$$
Now we can apply \autoref{L:ribbons_in_braid_groups} to $X_0$ and we get $r_{X_0,s_j}=\sigma_{i_0,j,k_0}$.
\end{proof}

We have then obtained an explicit description of the ribbon elements in $G = \mathcal{B}_{n+1}$  as products of atoms. Let us show some technical but important results, which are well-known for specialists in braid groups.

\begin{lemma}\label{L:Sjk_vee_atom}
Let $1\leq j\leq k\leq n$ and suppose that $i \notin \{j-1,\ldots,k+1\}$ and that $m\in \{j+1,\ldots,k\}$. The following are LCM-diagrams in
the classical monoid for $\mathcal{B}_{n+1}$:
$$
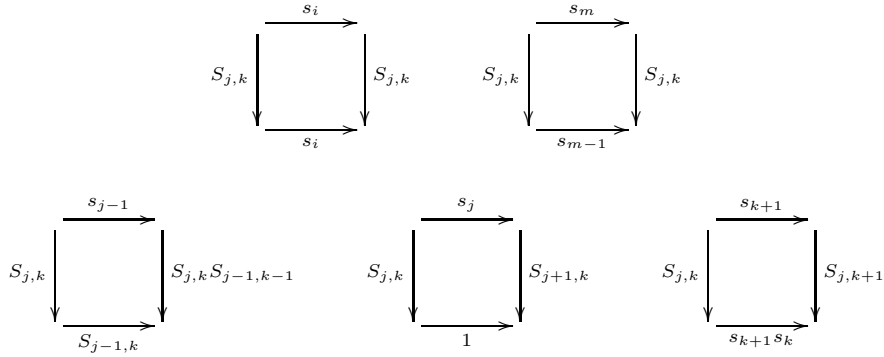
\xymatrix@C=12mm@R=12mm{
  \ar[r]^{s_i} \ar[d]_{S_{j,k}}
& \ar[d]^{S_{j,k}}
\\
  \ar[r]_{s_i}
&
}
\qquad
\xymatrix@C=12mm@R=12mm{
  \ar[r]^{s_m} \ar[d]_{S_{j,k}}
& \ar[d]^{S_{j,k}}
\\
  \ar[r]_{s_{m-1}}
&
}
$$
$$
\xymatrix@C=12mm@R=12mm{
  \ar[r]^{s_{j-1}} \ar[d]_{S_{j,k}}
& \ar[d]^{S_{j,k}S_{j-1,k-1}}
\\
  \ar[r]_{S_{j-1,k}}
&
}
\qquad
\xymatrix@C=12mm@R=12mm{
  \ar[r]^{s_{j}} \ar[d]_{S_{j,k}}
& \ar[d]^{S_{j+1,k}}
\\
  \ar[r]_{1}
&
}
\qquad
\xymatrix@C=12mm@R=12mm{
  \ar[r]^{s_{k+1}} \ar[d]_{S_{j,k}}
& \ar[d]^{S_{j,k+1}}
\\
  \ar[r]_{s_{k+1}s_k}
&
}
$$
\end{lemma}

\begin{proof}
The first diagram is clear, as $S_{j,k}$ cannot start with $s_i$, so the length of $S_{j,k}\vee s_i$ must be at least one letter bigger than $S_{j,k}$. Since $S_{j,k}s_i=s_i S_{j,k}$, this element admits $S_{j,k}$ and $s_i$ as prefixes and has the minimal possible length, so $s_i\vee S_{j,k}=S_{j,k}s_i=s_i S_{j,k}$, as stated in the diagram.

The same argument holds for the second diagram, taking into account that $s_m\not\preccurlyeq S_{j,k}$ (as the only possible initial letter of $S_{j,k}$ is $s_j$, since $\{j,j+1\}$ is the only pair of consecutive strands that cross in $S_{j,k}$). This time we have:
$$
 s_m S_{j,k}  = s_m S_{j,m-2} s_{m-1} s_m S_{m+1,k} = S_{j,m-2} s_m s_{m-1} s_m S_{m+1,k}
$$
$$
=  S_{j,m-2} s_{m-1} s_m s_{m-1} S_{m+1,k} = S_{j,m-2} s_{m-1} s_m S_{m+1,k} s_{m-1} = S_{j,k} s_{m-1}.
$$
So this element is $s_m\vee S_{j,k}$, as we wanted to show.

The fourth diagram is evident, as $S_{j,k}=s_j S_{j+1,k}$.

The fifth diagram can be shown by concatenating several known LCM-diagrams as follows:
$$
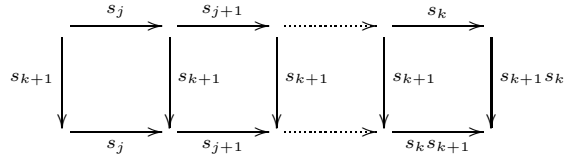
\xymatrix@C=12mm@R=12mm{
  \ar[r]^{s_j} \ar[d]_{s_{k+1}}
& \ar[r]^{s_{j+1}} \ar[d]^{s_{k+1}}
& \ar@{.>}[r]  \ar[d]^{s_{k+1}}
& \ar[r]^{s_k} \ar[d]^{s_{k+1}}
& \ar[d]^{s_{k+1}s_k}
\\
  \ar[r]_{s_j}
& \ar[r]_{s_{j+1}}
& \ar@{.>}[r]
& \ar[r]_{s_ks_{k+1}}
&
}
$$
The outermost paths in this diagram coincide with the elements of the fifth diagram in the statement (transposed).

It only remains to show that the third diagram is correct. If $j=k$ the result is clear, as $s_{j-1}\vee s_{j}= s_{j-1}s_js_{j-1}= s_js_{j-1}s_j$. So we assume that $j<k$ and that the claim holds for smaller values of $k$. Then we have, as $S_{j,k}=S_{j,k-1}s_{k}$, the following LCM-diagram:
$$
\xymatrix@C=16mm@R=12mm{
  \ar[r]^{s_{j-1}} \ar[d]_{S_{j,k-1}}
& \ar[d]^{S_{j,k-1}S_{j-1,k-2}}
\\
  \ar[r]^{S_{j-1,k-1}} \ar[d]_{s_{k}}
& \ar[d]^{s_{k}s_{k-1}}
\\
  \ar[r]_{S_{j-1,k}}
&
}
$$
The top square holds by induction hypothesis, and the bottom one comes from the fifth diagram in the statement. This finishes the proof, as the vertical path on the right is:
$$
 S_{j,k-1}S_{j-1,k-2}s_{k}s_{k-1}= S_{j,k-1}s_{k}S_{j-1,k-2}s_{k-1}= S_{j,k}S_{j-1,k-1}.
$$
\end{proof}

\begin{lemma}\label{L:Skj_vee_atom}
Let $1\leq j\leq k\leq n$ and suppose that $i \notin \{j-1,\ldots,k+1\}$ and that $m\in \{j,\ldots,k-1\}$. The following are LCM-diagrams in the classical monoid for $\mathcal{B}_{n+1}$:
$$
\xymatrix@C=12mm@R=12mm{
  \ar[r]^{s_i} \ar[d]_{S_{k,j}}
& \ar[d]^{S_{k,j}}
\\
  \ar[r]_{s_i}
&
}
\qquad
\xymatrix@C=12mm@R=12mm{
  \ar[r]^{s_m} \ar[d]_{S_{k,j}}
& \ar[d]^{S_{k,j}}
\\
  \ar[r]_{s_{m+1}}
&
}
$$
$$
\xymatrix@C=12mm@R=12mm{
  \ar[r]^{s_{k+1}} \ar[d]_{S_{k,j}}
& \ar[d]^{S_{k,j}S_{k+1,j+1}}
\\
  \ar[r]_{S_{k+1,j}}
&
}
\qquad
\xymatrix@C=12mm@R=12mm{
  \ar[r]^{s_{k}} \ar[d]_{S_{k,j}}
& \ar[d]^{S_{k-1,j}}
\\
  \ar[r]_{1}
&
}
\qquad
\xymatrix@C=12mm@R=12mm{
  \ar[r]^{s_{j-1}} \ar[d]_{S_{k,j}}
& \ar[d]^{S_{k,j-1}}
\\
  \ar[r]_{s_{j-1}s_j}
&
}
$$
\end{lemma}

\begin{proof}
These diagrams are obtained by conjugating those in \autoref{L:Sjk_vee_atom} by $\Delta$, an operation that preserves least common multiples, and sends $s_i$ to $s_{n-i}$ for every $i$.
\end{proof}

Using the above two results, we can derive the least common multiples of a ribbon element and a suitable atom:

\begin{lemma}\label{L:sigmaijk_vee_atom}
Given $1\leq i\leq j \leq k \leq n$ and $m\notin\{i-1,j,k+1\}$, the following are LCM-diagrams in the classical monoid for $\mathcal{B}_{n+1}$, whenever the elements involved are defined:
$$
\xymatrix@C=12mm@R=12mm{
  \ar[d]_{\sigma_{i,j,k}} \ar[r]^{s_{m}}
& \ar[d]^{\sigma_{i,j,k}}
\\
  \ar[r]_{s_t}
&
}
\qquad
\xymatrix@C=12mm@R=12mm{
  \ar[d]_{\sigma_{i,j,k}} \ar[r]^{s_{i-1}}
& \ar[d]^{\sigma_{i-1,j,k}}
\\
  \ar[r]_{S_{i-1,k+i-j}}
&
}
\qquad
\xymatrix@C=12mm@R=12mm{
  \ar[d]_{\sigma_{i,j,k}} \ar[r]^{s_{k+1}}
& \ar[d]^{\sigma_{i,j,k+1}}
\\
  \ar[r]_{S_{k+1,k+i-j}}
&
}
$$
In the first diagram, $s_t$ is an atom (conjugate of $s_m$ by $\sigma_{i,j,k}$).
\end{lemma}

\begin{proof}
The first diagram holds since the only possible initial letter of $\sigma_{i,j,k}$ is $\sigma_j$ (as $j$ and $j+1$ are the only consecutive strands which cross in the simple element $\sigma_{i,j,k}$), hence the least common multiple of $s_m$ and $\sigma_{i,j,k}$ must be bigger than $\sigma_{i,j,k}$. There is a common multiple one letter bigger, namely:
$$
  s_m\sigma_{i,j,k} = s_m \Delta_{\{i,\ldots,k\}\setminus \{j\}}^{-1} \Delta_{\{i,\ldots,k\}} = \Delta_{\{i,\ldots,k\}\setminus \{j\}}^{-1} \Delta_{\{i,\ldots,k\}} s_t = \sigma_{i,j,k}s_t.
$$
Notice that $\Delta_X$ conjugates every atom in $X$ to another atom in $X$, and commutes with every atom not adjacent to $X$, so the central equality above holds for some atom $s_t$. This element is thus the searched least common multiple, and the first diagram holds.

The second diagram is obtained from the following concatenation of LCM-diagrams, coming from \autoref{L:Sjk_vee_atom}:
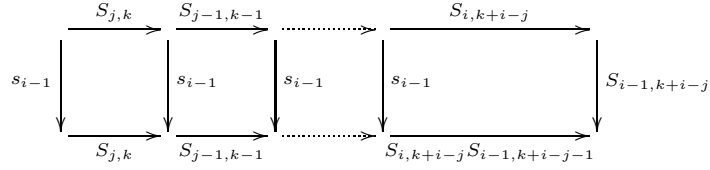
$$
\xymatrix@C=12mm@R=12mm{
  \ar[r]^{S_{j,k}} \ar[d]_{s_{i-1}}
& \ar[r]^{S_{j-1,k-1}} \ar[d]^{s_{i-1}}
& \ar@{.>}[r]  \ar[d]^{s_{i-1}}
& \ar[rr]^{S_{i,k+i-j}} \ar[d]^{s_{i-1}}
& & \ar[d]^{S_{i-1,k+i-j}}
\\
  \ar[r]_{S_{j,k}}
& \ar[r]_{S_{j-1,k-1}}
& \ar@{.>}[r]
& \ar[rr]_{S_{i,k+i-j}S_{i-1,k+i-j-1}}
& &
}
$$
The product of the arrows in the top row is $\sigma_{i,j,k}$, while the product of the arrows in the bottom row is $\sigma_{i-1,j,k}$. Hence the second diagram holds.

Finally, the third one is obtained from the second one after conjugation by $\Delta$.
\end{proof}

We can now determine the minimal positive conjugators in some suitable cases:

\begin{proposition}\label{P:braid_minimal_conjugators_are_ribbons}(see~\cite{CGGW})
Let $G = \mathcal{B}_{n+1}$ be the braid group on $n+1$ strands. Let $x\in G^+$ and let $X=\operatorname{supp}(x)$. For every $s_j\in \mathcal A \setminus X$ one has $\rho_{s_j}(x)=r_{X,s_j}$.
\end{proposition}

\begin{proof}
Recall that in order to find $\rho_{s_j}(x)$ one just needs to compute the pre-minimal conjugators and converging prefixes for $s_j$ and $x$. The initial one is $c_0=s_j=\sigma_{j,j,j}$. Then, by \autoref{L:sigmaijk_vee_atom}, if some pre-minimal conjugator is $\sigma_{i,j,k}$, the next one will be either $\sigma_{i,j,k}$, or $\sigma_{i-1,j,k}$, or $\sigma_{i,j,k+1}$ (here we are using that $s_j\notin X$).  From \autoref{L:ribbon_prefixes_in_braids}, it follows that each pre-minimal conjugator is a prefix of the next one. Therefore, by \autoref{L:pre-minimal_prefixes}, $\rho_{s_j}(x)=x^m\backslash s_j$ for some $m\geq 0$. That is, in order to obtain $\rho_{s_j}(x)$, we just need to compute pre-minimal conjugators until the sequence stabilizes.

Notice also that a pre-minimal conjugator $\sigma_{i,j,k}$ is different to the next one if and only if the corresponding atom in the word representing $x$ is adjacent to $\{i,\ldots,k\}$. This means that the sequence of pre-minimal conjugators will stabilize exactly when it reaches $\sigma_{i_0,j,k_0}$, where $\{i_0,\ldots,k_0\}$ is the connected component of $X\cup\{s_j\}$ which contains $s_j$. Hence $\rho_{s_j}(x)=\sigma_{i_0,j,k_0}$ and, by \autoref{L:ribbons_in_braid_groups_complete}, $\sigma_{i_0,j,k_0}=r_{X,\sigma_j}$, as we wanted to show.
\end{proof}

The above result yields immediately that the classical Garside structure of $G$ is support-preserving. This is shown in~\cite[Corollary 6.5]{CGGW} for every Artin group of spherical type, but we state it here for completeness.

\begin{proposition}(see \cite[Corollary 6.5]{CGGW})\label{P:braid_support_preserving} The classical Garside structure for
the classical braid group
 is support-preserving.
\end{proposition}

\begin{proof}
Let $x,y$ be positive braids which are connected by a minimal positive conjugator, that is, $x^\rho=y$ where $\rho=\rho_{s_j}(x)$ for some atom $s_j$. Let $X=\operatorname{supp}(x)$ and $Y=\operatorname{supp}(y)$. We need to show that $\rho$ conjugates $G_X$ to $G_Y$.

If $s_j\notin X$ then, by \autoref{P:braid_minimal_conjugators_are_ribbons}, $\rho=r_{X,s_j}$. But conjugation by $r_{X,s_j}$ sends the set of atoms $X$ to another set of atoms $Z$ of the same size. So the conjugation of $x$ by $\rho$ acts letter by letter, sending each letter of $x$ to the corresponding letter of $Z$. It follows that $Z$ is precisely the support of the obtained element $y$, so $Z=Y$ and conjugation by $\rho$ sends the atoms of $X$ to the atoms of $Y$, hence it sends $G_X$ to $G_Y$, as we wanted to show.

Since conjugation by a minimal positive conjugator $\rho_{s_j}(x)$ where $s_j\notin X$ preserves the sizes of the supports, we can apply \autoref{P:atom_in_X_support-preserving} to conclude that every $\rho_{s_i}(x)$ also preserves the support when $s_i\in X$. Therefore, the classical Garside structure if $B$ is support-preserving.
\end{proof}

In~\cite[Proposition 7.2]{CGGW}, the above result, together with some technical arguments involving a set denoted $RSSS_{\infty}(x)$, was used to show the existence of the parabolic closure of an element in an Artin group of spherical type. From the definition and properties of recurrent elements shown in this paper, we already have this result immediately from \autoref{T:parabolic_closure_exists}.

\begin{corollary}
Every element of the classical braid group
admits a parabolic closure.
\end{corollary}

Notice that, after~\cite[Proposition 7.2]{CGGW}, we can also deduce the following from \autoref{T:parabolic_closure_exists}:

\begin{corollary}
Every element in an Artin group of spherical type admits a parabolic closure.
\end{corollary}

As we said, this result was already shown in~\cite{CGGW}, but we gave a complete proof here, using recurrent elements.

\subsection{The group $G(e,e,n)$}

In this section we will study the complex braid group $B$ of type $G(e,e,n)$, endowed with the standard monoid structure of \autoref{sect:G(e,e,n)} with set of atoms $\mathcal A=\{t_0,t_1,\ldots,t_{e-1},s_3,s_4,\ldots,s_n\}$, with $e\geq 1$ and $n\geq 2$.
We will show that the standard parabolic subgroups associated to the standard Garside structure are the `parabolic subgroups' introduced in \cite{CALMAR} \S 6.3 -- that we already called `standard parabolic subgroups' in Corollary \ref{cor:parabsCP} --
 and we will show that it is a support-preserving LCM-Garside structure, so that every element admits a parabolic closure.

We remark that we modified the order of the elements $t_i$ with respect to~\cite{CORPIC}, using the permutation $t_i\mapsto t_{e-1-i}$.  We now provide some important information about this Garside structure, taken from~\cite{CORPIC}:

If we denote $\tau=t_it_{i+1}$ (which is the same element for any $i$), we have $\tau=t_0\vee \cdots \vee t_{e-1}$, and also $\tau=t_i\vee t_j$ for every $i\neq j$.

For $k=2,\ldots,n$, let $\Lambda_k=s_ks_{k-1}\cdots s_3 \:\tau\: s_3\cdots s_{k-1}s_k$ (in particular, $\Lambda_2=\tau$). Then the Garside element $\Delta$ is the LCM of the atoms, $\Delta=\Delta_{\mathcal A}$, which we can write in four different ways as follows:
$$
   \Delta=\Lambda_2\Lambda_3\cdots \Lambda_n = \Lambda_n\Lambda_{n-1}\cdots \Lambda_2 = (\tau s_3\cdots s_n)^{n-1}=(s_n\cdots s_3\tau)^{n-1}
$$
One has, for every $i$:
$$
   \Delta\tau = \tau\Delta,\qquad  \Delta s_i=s_i\Delta, \qquad \Delta t_i = t_{i-n} \Delta,
$$
where the subindices of the $t_i$ are taken modulo $e$. This shows the permutation of the atoms induced by conjugation by $\Delta$. We can also see how conjugation by an element $\Lambda_k$ affects some atoms:
$$
   \Lambda_k s_i = s_i\Lambda_k\ (i\neq k,k+1),\qquad \Lambda_2 t_{i} = t_{i-2}\Lambda_2, \qquad \Lambda_k t_{i} = t_{i-1} \Lambda_k\ (k>2).
$$

Finally, the {\bf simple elements} in this structure are the elements of the form $p_2\cdots p_r$, where each $p_k$ is a prefix of $\Lambda_k$, taking into account that each $\Lambda_k$ only admits the evident prefixes: those which correspond to a prefix of the word $s_ks_{k-1}\cdots s_3 t_it_{i+1} s_3\cdots s_{k-1}s_k$ for some $i$~\cite[Theorem 3.7]{CORPIC}.

With the above information, we will show that $(B,B^+,\Delta)$ is a support-preserving LCM-Garside structure. Let us denote $T_e=\{t_0,\ldots,t_{e-1}\}$ and $S_n=\{s_3,\ldots,s_n\}$, so the set of atoms of $B$ is $\mathcal A=T_e\cup S_n$. We have:

\begin{proposition}\label{P:standard_set_of_atoms}
The standard Garside structure $(B,B^+,\Delta)$ for $G(e,e,n)$
is an LCM-Garside structure. Moreover, if $X\subset \mathcal A$ is a set of atoms, then $X$ is saturated if and only if $\#(X\cap T_e)\in\{0,1,e\}$.
\end{proposition}

\begin{proof}
We already know that $\Delta=\Delta_{\mathcal A}$, the least common multiple of all the atoms. We need to study the least common multiple $\Delta_X$ of subsets $X\subset \mathcal A$. The fact that $\#(X\cap T_e)\in\{0,1,e\}$
implies that $X$ is saturated is an immediate consequence of \cite{CALMAR} \S 6.3. If $1<\#(X\cap T_e)<e$, let $t_i,t_j\in X$ and let $t_k\notin X$. We have $t_i\vee t_j=\tau$, hence $\tau\preccurlyeq \Delta_X= x_1\vee \cdots \vee x_m$, where $X=\{x_1,\ldots,x_m\}$. But $t_k\preccurlyeq \tau$, hence $t_k\preccurlyeq \Delta_X$, implying that $X$ is not saturated. Hence, we only need to consider the elements $\Delta_X$ with $\#(X\cap T_e)\in\{0,1,e\}$, and it follows from~\cite{CALMAR} that such an element is balanced and that $G_{\Delta_X}$ is a standard parabolic subgroup.

\end{proof}
We have then shown that the standard parabolic subgroups of $B$ are those generated by some set $X\subset \mathcal A$ which either contains no $t_i$, or contains exactly one, or contains them all. From the proof of the above result it follows that if $X$ is not saturated (that is, if $1<\#(X\cap T_e)<e$), then $\overline{X}=X\cup T_e$. Recall that we always have $\Delta_X=\Delta_{\overline{X}}$.

The rest of this section is devoted to show that $(B,B^+,\Delta)$ is a support-preserving Garside structure for $B$. By \autoref{P:support_preserved_by_minimal_positive_conjugators}, we just need to show that for every positive element $x\in B^+$ and every minimal positive conjugator $\rho$ for $x$ such that $x^\rho=y$, then $(G_{X})^{\rho}=G_{Y}$, where $X=\operatorname{Supp}(x)$ and $Y=\operatorname{Supp}(y)$.

We will then need to study in detail the minimal positive conjugators for an element in $B^+$, and to check that conjugation by such an element `preserves the support' in the natural way.

Recall the concept of {\em ribbon element} given in \autoref{D:ribbon}. As we did with braid groups in the previous section, we will show that the minimal conjugating elements that connect positive conjugates
 with different supports, are all {\em ribbons} of the form $r_{X,u}$, for some $X$ and $u$.

When $\#(X\cap T_e)=0$ and $u$ is an atom, we already know that $X\cup \{u\}$ is contained into a submonoid of $B^+$ isomorphic to a braid monoid (with the same Garside structure), so we know exactly how the ribbon $r_{X,u}$ looks like. The same happens when $\#(X\cap T_e)=1$ and $u\in S_n$. Hence we only need to study ribbons in the remaining cases.

Given $i,k\in \{2,\ldots,n\}$, we denote:
$$
   \Lambda_{i,k}=s_is_{i-1}\cdots s_3 \:\tau\: s_3\cdots s_{k-1}s_k.
$$
We consider that $\Lambda_{2,k}=\tau\: s_3\cdots s_{k-1}s_k$, $\Lambda_{i,2}=s_is_{i-1}\cdots s_3 \:\tau$, and $\Lambda_{2,2}=\tau$. Notice that $\Lambda_{k,k}=\Lambda_k$ for every $k=2,\ldots,n$.

We will also denote $S_{i,k}=s_is_{i+1}\cdots s_k$ and $S_{k,i}=s_ks_{k-1}\cdots s_i$ whenever $3\leq i\leq k\leq n$. We will extend this notation to $i=2$, by denoting $s_2=t_a$ for some given $a\in\{0,\ldots,e-1\}$ which must be specified. Hence, we will denote $S_{2,k}^{(a)}=t_a s_3\cdots s_k$ and $S_{k,2}^{(a)}=s_k\cdots s_3 t_a$. We will sometimes use the superscript even if $i\neq 2$: in that case $S_{i,k}^{(a)}=S_{i,k}$ and $S_{k,i}^{(a)}=S_{k,i}$.

Let $X\subset \mathcal A$ such that $X\cap T_e$ has more than one element. We will use the concept of {\em adjacent generators} coming from Artin groups, as follows. We will say that $s_i$ and $s_j$ are adjacent if and only if $i$ and $j$ are consecutive numbers, that $s_3$ is adjacent to all atoms in $T_e$, and that all atoms in $T_e$ are adjacent to each other. We can imagine a graph whose vertices are the atoms in $\mathcal A$, where adjacent atoms are connected by an edge. In this way we can talk about {\em connected} subsets of $\mathcal A$, or about the {\em connected components} of a subset of $\mathcal A$.

We just need to determine the ribbon elements in two cases, given by the following two Lemmas.

\begin{lemma}\label{L:ribbons_in_B(e,e,n)}
Let $3\leq j\leq k\leq n$. Denote $X=\{t_{i_1},\ldots,t_{i_p}\}\cup \{s_3,\ldots,s_{j-1}\}\cup \{s_{j+1},\ldots,s_k\}$ for some $p>1$. Then one has:
$$
   r_{X,s_j}=\Lambda_{j,k}\Lambda_{j,k-1}\cdots \Lambda_{j,j} = \Lambda_{j,j}\Lambda_{j+1,j}\cdots \Lambda_{k,j}
$$
\end{lemma}

\begin{proof}
Let $Y=X\cup \{s_j\} = \{t_{i_1},\ldots,t_{i_p}\}\cup \{s_3,\ldots,s_{k}\}$. Notice that $Y$ is not necessarily saturated but, as it contains more than one element from $T_e$, we have that $\overline{Y}=T_e\cup \{s_3,\ldots,s_k\}$.

Suppose that $k-j=0$. Then $\overline{X}=T_e\cup \{s_3,\ldots,s_{j-1}\}$ and $\overline{X\cup \{s_j\}}=\overline{Y}=T_e\cup \{s_3,\ldots,s_{j}\}$. This implies that $\Delta_X=\Lambda_2\cdots \Lambda_{j-1}$ and that $\Delta_{X\cup \{s_j\}}=\Lambda_2\cdots \Lambda_{j-1}\Lambda_j$. Hence $r_{X,s_j}=\Lambda_j = \Lambda_{j,j}$, as we wanted to show.

Now suppose that $k-j>0$ and that the result holds for smaller values. Let $Z=X\setminus \{s_k\}$. Since $s_j\neq s_k$, we can apply the induction hypothesis to $Z$ and $s_j$, to obtain $r_{Z,s_j}=\Lambda_{j,k-1}\Lambda_{j,k-2}\cdots \Lambda_{j,j} = \Lambda_{j,j} \Lambda_{j+1,j}\cdots \Lambda_{k-1,j}$.

We need to find $r_{X,s_j}$ such that $\Delta_{X}r_{X,s_j}=\Delta_{X\cup \{s_j\}}$. We have:
$$
   \Delta_{X\cup\{s_j\}}=\Delta_{Z\cup\{s_j\}}\Lambda_{k}= \Delta_{Z}r_{Z,s_j} \Lambda_k = \Delta_Z \left(\prod_{m=1}^{k-j}{\Lambda_{j,k-m}}\right)\Lambda_k
$$
Now let $X_2=\{s_{j+1},\ldots,s_k\}$ and $X_1=X\setminus X_2$ be the two connected components of $X$, and let $Z_2=\{s_{j+1},\ldots,s_{k-1}\}$ and $Z_1=Z\setminus Z_2$ be the two connected components of $Z$. We have $X_1=Z_1$ and
$$
  \Delta_X= \Delta_{X_1}\Delta_{X_2} = \Delta_{Z_1}\Delta_{Z_2}r_{Z_2,s_k} = \Delta_{Z}r_{Z_2,s_k} = \Delta_Z (s_k s_{k-1}\cdots s_{j+1}),
$$
where the last equality comes from \autoref{L:ribbons_in_braid_groups}, since $Z_2\cup \{s_j\}\subset S_n$ and it is a connected set.

Now recall how conjugation by $\Lambda_k$ affects the atoms. It follows that $s_i\Lambda_k=\Lambda_k s_i$ for $i\neq k,k+1$, and that $\tau\Lambda_k = \Lambda_k\tau$ for $k\geq 2$. Hence $\Lambda_{j,p}\Lambda_k=\Lambda_k\Lambda_{j,p}$ whenever $j\leq p<k$. Therefore:
$$
   \Delta_{X\cup\{s_j\}} = \Delta_Z \left(\prod_{m=1}^{k-j}{\Lambda_{j,k-m}}\right)\Lambda_k =\Delta_Z \Lambda_k \left(\prod_{m=1}^{k-j}{\Lambda_{j,k-m}}\right)
$$
$$
   =\Delta_Z (s_k s_{k-1}\cdots s_{j+1}) \Lambda_{j,k} \left(\prod_{m=1}^{k-j}{\Lambda_{j,k-m}}\right)
   =\Delta_X \left(\prod_{m=0}^{k-j}{\Lambda_{j,k-m}}\right),
$$
showing that $r_{X,s_j}=\Lambda_{j,k}\Lambda_{j,k-1}\cdots \Lambda_{j,j}$.

Finally, we have:
$$
  r_{X,s_j}=\Lambda_{j,k}(\Lambda_{j,k-1}\cdots \Lambda_{j,j}) = \Lambda_{j,k}(\Lambda_{j,j}\cdots \Lambda_{k-1,j})=
$$
$$
 = \Lambda_{j,j}\sigma_{j+1}\cdots \sigma_{k} (\Lambda_{j,j}\cdots \Lambda_{k-1,j})
 = \Lambda_{j,j} (\sigma_{j+1}\Lambda_{j,j})\cdots (\sigma_k \Lambda_{k-1,j})
 = \Lambda_{j,j} \Lambda_{j+1,j}\cdots \Lambda_{k,j},
$$
where we have used the induction hypothesis and commutation relations only.
\end{proof}

\begin{lemma}\label{L:r_(X_k^((b-1)),t_b)}
For every $k\in \{2,\ldots,n\}$ and every $b\in \{0,\ldots,e-1\}$, let $X=\{t_{b-1},s_3,\ldots,s_k\}$. Then one has:
$$
  r_{X,t_b} = (t_bs_3\cdots s_k)(t_{b+1}s_3\cdots s_{k-1})\cdots (t_{b+k-3} s_3) t_{b+k-2}.
$$
\end{lemma}

\begin{proof}
For every $j\in \{2,\ldots,n\}$ and every $a\in\{0,\ldots,e-1\}$, let us denote $X_{2,j}^{(a)}=\{t_a,s_3,\ldots,s_j\}$.

We will show the result by induction in $k$. When $k=2$, we need to show that $r_{X,t_b}=t_b$. But this is clear, since in this case $X=\{t_{b-1}\}$ and $X\cup \{t_b\}=\{t_{b-1},t_b\}$, hence
 $$
    \Delta_{\{t_{b-1}\}} t_b = t_{b-1} t_b = \tau = \Delta_{\{t_{b-1},t_b\}}.
 $$

Let us then assume that $k>2$, and that the result is true for smaller values of $k$. We know that
$$
   \Delta_{X_{2,k}^{(b-1)}} = \Delta_{X_{2,k-1}^{(b-1)}} (s_k\cdots s_3 t_{b-1})
$$
by \autoref{L:ribbons_in_braid_groups}, since this is a relation in a braid monoid and $X=X_{2,k}^{(b-1)}$ is connected. Also, by induction hypothesis we have:
$$
   (t_bs_3\cdots s_k)(t_{b+1}s_3\cdots s_{k-1})\cdots (t_{b+k-3} s_3) t_{b+k-2} = (t_bs_3\cdots s_k) r_{X_{2,k-1}^{(b)},t_{b+1}}.
$$
Finally, recalling how the element $\Lambda_k$ conjugates the atoms, we see that
$$
   \Lambda_k r_{X_{2,k-1}^{(b)},t_{b+1}} = r_{X_{2,k-1}^{(b-1)},t_{b}} \Lambda_k.
$$
Therefore:
\begin{eqnarray*}
& & \Delta_{X_{2,k}^{(b-1)}}  (t_bs_3\cdots s_k)(t_{b+1}s_3\cdots s_{k-1})\cdots (t_{b+k-3} s_3) t_{b+k-2}
\\
& = & \Delta_{X_{2,k-1}^{(b-1)}} (s_k\cdots s_3 t_{b-1}) (t_bs_3\cdots s_k) r_{X_{2,k-1}^{(b)},t_{b+1}}
\\
& = & \Delta_{X_{2,k-1}^{(b-1)}} \Lambda_k r_{X_{2,k-1}^{(b)},t_{b+1}}
\\
& = & \Delta_{X_{2,k-1}^{(b-1)}} r_{X_{2,k-1}^{(b-1)},t_{b}} \Lambda_k
\\
& = & \Delta_{X_{2,k-1}^{(b-1)}\cup \{t_b\}} \Lambda_k
\\
& = & \Lambda_2\cdots \Lambda_{k-1} \Lambda_k
\\
& = & \Delta_{X_{2,k}^{(b-1)}\cup \{t_b\}}
\end{eqnarray*}
This shows the result.
\end{proof}

Now that we have the explicit description of the elements $r_{X,u}$ which will be relevant for our purposes, we can describe the minimal conjugators joining positive elements. There is one case which is already clear, as it follows from the results in the previous section.

\begin{proposition}\label{P:B(e,e,n)_minimal_conjugators_are_ribbons}
Let $x\in B^+$ and let $X=\operatorname{supp}(x)$. Suppose that $X\cap T_e$ has at most one element. For every $u\in \mathcal A\setminus X$ such that $\#((X\cup \{u\})\cap T_e)\leq 1$ one has $\rho_{u}(x)=r_{X,u}$.
\end{proposition}

\begin{proof}
This follows immediately from \autoref{P:braid_minimal_conjugators_are_ribbons}, since in this case $X\cup \{u\}$ is contained in a monoid isomorphic to a braid monoid, with the same Garside structure.
\end{proof}

Now we need to see what happens when the atoms involved in the computations are not included into a braid monoid. We will need the following results:

\begin{lemma}\label{L:Two_squares}
Given $3\leq p\leq q \leq n$ and $a, b\in \{0,\ldots,e-1\}$ with $a\neq b$, the following are LCM-diagrams:
$$
\xymatrix@C=12mm@R=12mm{
  \ar[r]^{S_{q,p}} \ar[d]_{S_{2,p}^{(a)}}
&  \ar[d]^{S_{2,p+1}^{(a)}}
\\ \ar[r]_{S_{q,p-1}^{(a)}}
&
}
\qquad
\xymatrix@C=12mm@R=12mm{
  \ar[r]^{S_{q,2}^{(a)}} \ar[d]_{t_b}
&  \ar[d]^{{{{\scriptstyle t_{a+1} \atop \phantom{x}} \atop \scriptstyle s_3} \atop \phantom{x}} \atop \scriptstyle t_{b+1}}
\\ \ar[r]_{\Lambda_{q,3}}
&
}
$$
\end{lemma}

\begin{proof}
The first calculation is done in the monoid $\langle t_a,s_3,\ldots,s_q\rangle$, which is isomorphic to a braid monoid on $q$ strands $\langle \sigma_1,\ldots,\sigma_{q-1}\rangle$. The element  $S_{2,p}^{(a)}$ corresponds to a braid in which the first strand crosses once the strands $2,3,\ldots,p$, and the element $S_{q,p}$ corresponds to the braid in which the $q$-th strand crosses once the strands $q-1, q-2, \ldots, p-1$. Both are simple braids, and their least common multiple will also be simple, thus determined by the crossings of its strands. The least common multiple must contain all the mentioned crossings (which are all different), but one readily sees that no simple element can contain exactly that set of crossings. Adding the crossing of strands $1$ and $q$ will produce a simple braid, which will then be the least common multiple. It corresponds to the element $S_{2,p}^{(a)} S_{q,p-1}^{(a)} = S_{q,p} S_{2,p+1}^{(a)}$.

For the second claim, we write $S_{q,2}^{(a)}$ as $S_{q,4} s_3 t_a$ (where the first factor would be trivial if $q=3$). Then we have the following LCM-diagram, formed joining known LCM-diagrams:
$$
\xymatrix@C=12mm@R=12mm{
   \ar[r]^{S_{q,4}} \ar[dd]_{t_b}
&  \ar[r]^{s_3} \ar[dd]_{t_b}
&  \ar[r]^{t_a} \ar[d]_{t_b}
&  \ar[d]^{t_{a+1}}
\\
&
&  \ar[r]_{t_{b+1}} \ar[d]_{s_3}
&  \ar[d]^{\scriptsize s_3 \atop \scriptsize t_{b+1}}
\\
   \ar[r]_{S_{q,4}}
&  \ar[r]_{s_3t_b}
&  \ar[r]_{t_{b+1}s_3}
&
}
$$
The product of the arrows in the bottom row is precisely $\Lambda_{q,3}$ (as $t_bt_{b+1}=\tau$), hence the result is shown.
\end{proof}

Now let us define the elements which will correspond to pre-minimal conjugators in one of the cases we need to study. Recall the elements $\sigma_{i,j,k}$ defined in the previous section. In the case of
standard Garside structure of $G(e,e,n)$,
 we can define the same elements for $2\leq i\leq j \leq k$, taking into account that if $i=2$, a superscript $(a)$ is needed, meaning that $s_2=t_a$. Hence:
$$
  \sigma_{2,j,k}^{(a)} = S_{j,k}S_{j-1,k-1}\cdots S_{2,k-j+2}^{(a)} = S_{j,2}^{(a)}S_{j+1,3}\cdots S_{k,k-j+2}.
$$
Also, given $3\leq i\leq j\leq k\leq n$, let
$$
    \overline{\Lambda}_{i,j,k} = \Lambda_{j,i}\Lambda_{j+1,i}\cdots \Lambda_{k,i}.
$$
And, for any $p\in \{2,\ldots, n\}$ and every $a,b\in \{0,\ldots,e-1\}$ with $a\neq b$, we define:
$$
   \Omega_{p}^{(a,b)} = S_{2,p}^{(a+1)}S_{2,p-1}^{(a+2)}\cdots S_{2,3}^{(a+p-2)} S_{2,2}^{(b+p-2)}.
$$
To avoid confusion, we notice that $\Omega_{2}^{(a,b)}=t_b$, that $\Omega_{3}^{(a,b)}=t_{a+1} s_3 t_{b+1}$, and that $\Omega_{4}^{(a,b)}=t_{a+1} s_3 s_4 t_{a+2} s_3 t_{b+2}$.

We already know, from \autoref{L:sigmaijk_vee_atom}, which is the least common multiple of $\sigma_{i,j,k}$ and an atom in a braid monoid. We need to show the only missing case:

\begin{lemma}~\label{L:sigma2jk_vee_tb}
Let $3\leq j \leq k\leq n$ and $a,b\in \{0,\ldots,e-1\}$ with $a\neq b$. The following is an LCM-diagram:
$$
\xymatrix@C=12mm@R=12mm{
  \ar[d]_{\sigma_{2,j,k}^{(a)}} \ar[r]^{t_b}
&  \ar[d]^{\overline{\Lambda}_{3,j,k}}
\\ \ar[r]_{\Omega_{k-j+3}^{(a,b)}}
&
}
$$
Moreover, $\sigma_{2,j,k}^{(a)}\preccurlyeq \overline{\Lambda}_{3,j,k}$.
\end{lemma}

\begin{proof}
First, we claim that the following is an LCM-diagram, where $2\leq p<q\leq n$:
$$
\xymatrix@C=12mm@R=12mm{
  \ar[r]^{S_{q,p}^{(a)}} \ar[d]_{\Omega_{p}^{(a,b)}}
&  \ar[d]^{\Omega_{p+1}^{(a,b)}}
\\ \ar[r]_{\Lambda_{q,3}}
&
}
$$
The claim is shown by the following concatenation of LCM-diagrams:
$$
\xymatrix@C=18mm@R=12mm{
  \ar[r]^{S_{q,p+2}} \ar[d]_{S_{2,p}^{(a+1)}}
& \ar[r]^{S_{p+1,p}} \ar[d]_{S_{2,p}^{(a+1)}}
& \ar[d]^{S_{2,p+1}^{(a+1)}}
\\  \ar[d]_{S_{2,p-1}^{(a+2)}}
& \ar[r]^{S_{p+1,p-1}} \ar[d]_{S_{2,p-1}^{(a+2)}}
& \ar[d]^{S_{2,p}^{(a+2)}}
\\  \ar@{.>}[d]
& \ar[r]^{S_{p+1,p-2}} \ar@{.>}[d]
& \ar@{.>}[d]
\\  \ar[d]_{S_{2,3}^{(a+p-2)}}
& \ar[r]^{S_{p+1,3}} \ar[d]_{S_{2,3}^{(a+p-2)}}
& \ar[d]^{S_{2,4}^{(a+p-2)}}
\\  \ar[d]_{t_{b+p-2}}
& \ar[r]^{S_{p+1,2}^{(a+p-2)}} \ar[d]_{t_{b+p-2}}
& \ar[d]^{{{{\scriptstyle t_{a+p-1} \atop \phantom{x}} \atop \scriptstyle s_3} \atop \phantom{x}} \atop \scriptstyle t_{b+p-1}}
\\ \ar[r]_{S_{q,p+2}}
& \ar[r]_{\Lambda_{p+1,3}}
&
}
$$
The first column holds by commutativity (since the involved atoms are not consecutive), and it only appears if $q\geq p+2$. The squares on the second column hold from \autoref{L:Two_squares}. Then one just needs to notice that the products of the arrows at each edge of this diagram are precisely the elements of our claim, so the claim holds.

Now we finish the proof thanks to the following LCM-diagram:
$$
\xymatrix@C=18mm@R=12mm{
   \ar[d]_{\Omega_2^{(a,b)}} \ar[r]^{S_{j,2}^{(a)}}
&  \ar[d]_{\Omega_3^{(a,b)}} \ar[r]^{S_{j+1,3}}
&  \ar[d]_{\Omega_4^{(a,b)}} \ar@{.>}[r]
&  \ar[d]_{\Omega_{k-j+2}^{(a,b)}} \ar[r]^{S_{k,k-j+2}}
&  \ar[d]^{\Omega_{k-j+3}^{(a,b)}}
\\ \ar[r]_{\Lambda_{j,3}}
&  \ar[r]_{\Lambda_{j+1,3}}
&  \ar@{.>}[r]
&  \ar[r]_{\Lambda_{k,3}}
&
}
$$
Each square is an instance of the above claim. Since product of the arrows in the top row is $\sigma_{2,j,k}^{(a)}$, and the leftmost vertical arrow is $\Omega_2^{(a,b)}=t_b$, this corresponds to the LCM-diagram of the statement.

Now let us show that $\sigma_{2,j,k}^{(a)}\preccurlyeq \overline{\Lambda}_{3,j,k}$. We just proved that the LCM-diagram of the statement holds, hence $\sigma_{2,j,k}^{(a)}\Omega_{k-j+3}^{(a,b)} = t_b \overline{\Lambda}_{3,j,k}$. Now, for every $m\in\{j,\ldots,k\}$, we know that $t_b\Lambda_{m,3}=\Lambda_{m,3}t_{b+1}$. Hence:
$$
t_b \overline{\Lambda}_{3,j,k} = t_b (\Lambda_{j,3}\Lambda_{j+1,3}\cdots \Lambda_{k,3}) =
(\Lambda_{j,3}\Lambda_{j+1,3}\cdots \Lambda_{k,3}) t_{b+k-j+1} = \overline{\Lambda}_{3,j,k} t_{b+k-j+1}.
$$
On the other hand, by definition: $\Omega_{k-j+3}^{(a,b)}\succcurlyeq t_{b+k-j+1}$, so $\Omega_{k-j+3}^{(a,b)}=P t_{b+k-j+1}$ for some positive element $P$. Then
$$
\sigma_{2,j,k}^{(a)}P t_{b+k-j+1}= \sigma_{2,j,k}^{(a)}\Omega_{k-j+3}^{(a,b)} = t_b \overline{\Lambda}_{3,j,k} = \overline{\Lambda}_{3,j,k} t_{b+k-j+1}.
$$
Cancelling $t_{b+k-j+1}$, we obtain that $\sigma_{2,j,k}^{(a)}\preccurlyeq \overline{\Lambda}_{3,j,k}$, as we wanted to show.
\end{proof}

The elements $\sigma_{i,j,k}$ and $\overline{\Lambda}_{i,j,k}$ will be the pre-minimal conjugators in the case we are about to study, so we need to know the least common multiples of the latter with suitable atoms:

\begin{lemma}\label{L:atom_vee_overline_Lambda}
Let $3\leq m\leq j \leq d\leq n$. The following are LCM-diagrams, where in the first one $m\neq j$, in the second one $d+1\neq j$, and in the third one $i\neq m,j,d+1$.
$$
\xymatrix@C=18mm@R=12mm{
  \ar[d]_{\overline{\Lambda}_{m,j,d}} \ar[r]^{s_m}
&  \ar[d]^{\overline{\Lambda}_{m+1,j,d}}
\\ \ar[r]_{S_{m+d-j+1,m}}
&
}
\qquad
\xymatrix@C=18mm@R=12mm{
  \ar[d]_{\overline{\Lambda}_{m,j,d}} \ar[r]^{s_{d+1}}
&  \ar[d]^{\overline{\Lambda}_{m,j,d+1}}
\\ \ar[r]_{\Lambda_{d+1,m+1}}
&
}
\qquad
\xymatrix@C=12mm@R=12mm{
  \ar[d]_{\overline{\Lambda}_{m,j,d}} \ar[r]^{s_i}
&  \ar[d]^{\overline{\Lambda}_{m,j,d}}
\\ \ar[r]_{s_c}
&
}
$$
In the last diagram, $s_c$ is some atom in $S_n$. Moreover, in each diagram the vertical arrow on the left is a prefix of the vertical arrow on the right.
\end{lemma}

\begin{proof}
First we observe that the following is an LCM-diagram whenever $3\leq m\leq p < q$:
$$
\xymatrix@C=18mm@R=12mm{
   \ar[d]_{S_{p,m}} \ar[r]^{S_{q,3}}
&  \ar[d]^{S_{p+1,m+1}} \ar[r]^{\tau}
&  \ar[d]^{S_{p+1,m+1}} \ar[r]^{S_{3,m-1}}
&  \ar[d]^{S_{p+1,m+1}} \ar[r]^{s_m}
&  \ar[d]^{S_{p+1,m}}
\\ \ar[r]_{S_{q,3}}
&  \ar[r]_{\tau}
&  \ar[r]_{S_{3,m-1}}
&  \ar[r]_{s_m s_{m+1}}
&
}
$$
The top row represents $\Lambda_{p,m}$ and the bottom row represents $\Lambda_{p,m+1}$. Now we can concatenate the above diagram for different values of $p$ and $q$ (taking into account that $m<j$), to obtain:
$$
\xymatrix@C=20mm@R=12mm{
   \ar[d]_{S_{m,m}} \ar[r]^{\Lambda_{j,m}}
&  \ar[d]_{S_{m+1,m}} \ar[r]^{\Lambda_{j+1,m}}
&  \ar[d]_{S_{m+2,m}} \ar@{.>}[r]
&  \ar[d]_{S_{m+d-j,m}} \ar[r]^{\Lambda_{d,m}}
&  \ar[d]^{S_{m+d-j+1,m}}
\\ \ar[r]_{\Lambda_{j,m+1}}
&  \ar[r]_{\Lambda_{j+1,m+1}}
&  \ar@{.>}[r]
&  \ar[r]_{\Lambda_{d,m+1}}
&
}
$$
This shows the first diagram in the statement. Let us see that the vertical arrow on the left (of the first diagram in the statement) is a prefix of the vertical arrow on the right, that is, $\overline{\Lambda}_{m,j,d}\preccurlyeq \overline{\Lambda}_{m+1,j,d}$. From the diagram, we have that $\overline{\Lambda}_{m,j,d}S_{m+d-j+1,m} =  s_m\overline{\Lambda}_{m+1,j,d}$. Since $m<j$ in this case, we know that $s_m\Lambda_{k,m+1}=\Lambda_{k,m+1}s_m$ for every $k\in\{j,\ldots,d\}$. Then $s_m \overline{\Lambda}_{m+1,j,d} = \overline{\Lambda}_{m+1,j,d} s_m$. On the other hand, $S_{m+d-j+1,m}= S_{m+d-j+1,m+1}s_m$. Therefore $\overline{\Lambda}_{m,j,d}S_{m+d-j+1,m+1} s_m= \overline{\Lambda}_{m+1,j,d}s_m$ and then, cancelling $s_m$, we obtain that $\overline{\Lambda}_{m,j,d}\preccurlyeq \overline{\Lambda}_{m+1,j,d}$.

Now suppose that $d+1\neq j$. It is clear that for $p=j,\ldots,d-1$ one has $s_{d+1} \vee \Lambda_{p,m} = s_{d+1}\Lambda_{p,m} = \Lambda_{p,m} s_{d+1}$, since $s_{d+1}$ commutes with all atoms involving $\Lambda_{p,m}$. We then need to calculate the least common multiple of $s_{d+1}$ and $\Lambda_{d,m}$. This is done in the following LCM-diagram:
$$
  
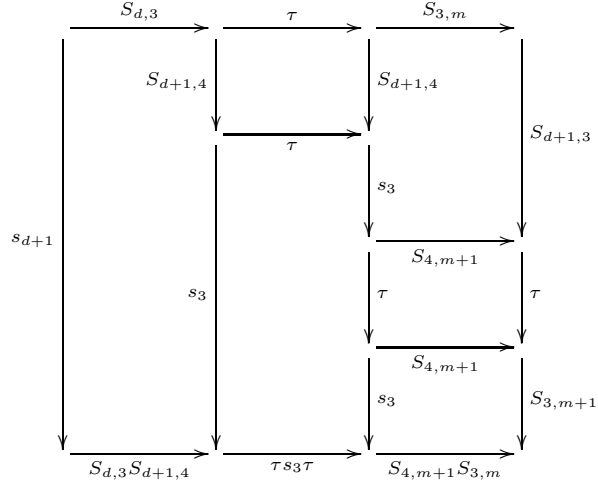
\xymatrix@C=18mm@R=12mm{
   \ar[dddd]_{s_{d+1}} \ar[r]^{S_{d,3}}
&  \ar[d]_{S_{d+1,4}} \ar[r]^{\tau}
&  \ar[d]^{S_{d+1,4}} \ar[r]^{S_{3,m}}
&  \ar[dd]^{S_{d+1,3}}
\\
&
   \ar[ddd]_{s_3} \ar[r]_{\tau}
&  \ar[d]^{s_3}
&
\\
&
& \ar[d]^{\tau} \ar[r]_{S_{4,m+1}}
& \ar[d]^{\tau}
\\
&
& \ar[d]^{s_3} \ar[r]_{S_{4,m+1}}
& \ar[d]^{S_{3,m+1}}
\\ \ar[r]_{S_{d,3}S_{d+1,4}}
& \ar[r]_{\tau s_3 \tau}
& \ar[r]_{S_{4,m+1}S_{3,m}}
&
}
$$
The squares in the above diagram either have already been considered, or they correspond to commuting atoms, or can be easily verified in the braid monoid, with the exception of the lower central one: $s_3\vee \tau = s_3\tau s_3 \tau = \tau s_3 \tau s_3$. These equalities can also be easily checked by decomposing them into suitable LCM-diagrams.

Notice that the lowest row corresponds to the following element:
\begin{eqnarray*}
S_{d,3} \underline{S_{d+1,4} \tau} s_3 \underline{\tau S_{4,m+1}} S_{3,m}
& = & S_{d,3} \tau \underline{S_{d+1,3} S_{4,m+1}} \tau S_{3,m}
\\ & = & S_{d,3} \tau S_{3,m} S_{d+1,3} \tau S_{3,m}
\\ & = & \Lambda_{d,m} \Lambda_{d+1,m}.
\end{eqnarray*}
Therefore, we have:
$$
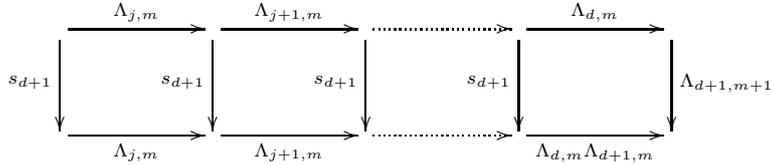
\xymatrix@C=18mm@R=12mm{
   \ar[d]_{s_{d+1}} \ar[r]^{\Lambda_{j,m}}
&  \ar[d]_{s_{d+1}} \ar[r]^{\Lambda_{j+1,m}}
&  \ar[d]_{s_{d+1}} \ar@{.>}[r]
&  \ar[d]_{s_{d+1}} \ar[r]^{\Lambda_{d,m}}
&  \ar[d]^{\Lambda_{d+1,m+1}}
\\ \ar[r]_{\Lambda_{j,m}}
&  \ar[r]_{\Lambda_{j+1,m}}
&  \ar@{.>}[r]
&  \ar[r]_{\Lambda_{d,m}\Lambda_{d+1,m}}
&
}
$$
This shows the second LCM-diagram in the statement. In this case it is trivial that the vertical arrow on the left (of the second diagram in the statement) is a prefix of the vertical arrow on the right, since $\overline{\Lambda}_{m,j,d} \Lambda_{d+1,m}=\overline{\Lambda}_{m,j,d+1}$ by definition.

Notice that we also have shown that $s_{d+1}\vee \Lambda_{d,m} = \Lambda_{d,m} \Lambda_{d+1,m}s_{m+1}$. Since this element admits $\Lambda_{d,m} \Lambda_{d+1,m}$ as a prefix, it follows that:
\begin{equation}\label{E:Lambdas}
   s_{d+1}\vee \Lambda_{d,m} \Lambda_{d+1,m} = s_{d+1}\Lambda_{d,m} \Lambda_{d+1,m} = \Lambda_{d,m} \Lambda_{d+1,m}s_{m+1}
\end{equation}

We will now suppose that either $i<m$ or $i>d+1$. It is easily seen that, in this case, $s_i\vee \Lambda_{p,m}=s_i\Lambda_{p,m}=\Lambda_{p,m}s_i$ for every $p\in \{j,\ldots,d\}$. Therefore $s_i\vee \overline{\Lambda}_{m,j,d}=s_i\overline{\Lambda}_{m,j,d} =  \overline{\Lambda}_{m,j,d}s_i$. Hence, the third diagram in the statement holds in this case, with $c=i$

If $m<i<j$, one has that $s_i\vee \Lambda_{p,m}=s_i\Lambda_{p,m}=\Lambda_{p,m}s_{i+1}$ for every $p\in \{j,\ldots,d\}$. This implies that $s_i\vee \overline{\Lambda}_{m,j,d}=s_i\overline{\Lambda}_{m,j,d} =  \overline{\Lambda}_{m,j,d}s_{i+j-d+1}$. So this case also holds, with $c=i+j-d+1$.

Finally, we suppose that $m\leq j<i<d+1$. The statement in this case is shown by the following LCM-diagram:
$$
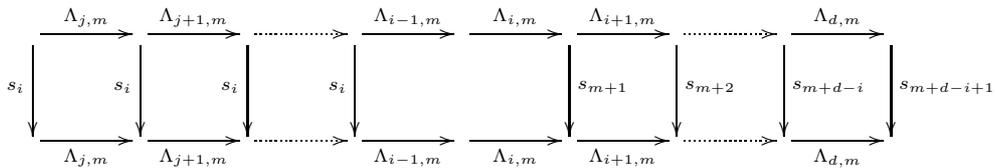
\xymatrix@C=12mm@R=12mm{
   \ar[d]_{s_i} \ar[r]^{\Lambda_{j,m}}
&  \ar[d]_{s_i} \ar[r]^{\Lambda_{j+1,m}}
&  \ar[d]_{s_i} \ar@{.>}[r]
&  \ar[d]_{s_i} \ar[r]^{\Lambda_{i-1,m}}
&  \ar[r]^{\Lambda_{i,m}}
&  \ar[d]^{s_{m+1}} \ar[r]^{\Lambda_{i+1,m}}
&  \ar[d]^{s_{m+2}} \ar@{.>}[r]
&  \ar[d]^{s_{m+d-i}} \ar[r]^{\Lambda_{d,m}}
&  \ar[d]^{s_{m+d-i+1}}
\\ \ar[r]_{\Lambda_{j,m}}
&  \ar[r]_{\Lambda_{j+1,m}}
&  \ar@{.>}[r]
&  \ar[r]_{\Lambda_{i-1,m}}
&  \ar[r]_{\Lambda_{i,m}}
&  \ar[r]_{\Lambda_{i+1,m}}
&  \ar@{.>}[r]
&  \ar[r]_{\Lambda_{d,m}}
&
}
$$
The central rectangle comes from Equation~(\ref{E:Lambdas}), and all other squares have already been considered. This shows the final case, with $c=m+d-i+1$.

Since the vertical arrows on the left and on the right of the third diagram in the statement coincide, the former is a prefix of the latter.
\end{proof}

The last technical result that we will need concerns the following elements, for $3\leq j \leq k\leq n$ and $b\in \{0,\ldots,e-1\}$:
$$
   \Phi_{k,j}^{(b)}=S_{2,k}^{(b)}S_{2,k-1}^{(b+1)}\cdots S_{2,j}^{(b+k-j)}.
$$
These elements will be the pre-minimal conjugators in the final case we will consider. Hence, we need to study their least common multiples with suitable atoms.

\begin{lemma}\label{L:Phi}
For $3\leq j \leq k\leq n$ and $b\in \{0,\ldots,e-1\}$, if we denote $s_2^{(b-1)}=t_{b-1}$ and $s_c^{(b-1)}=s_c$ for $c\geq 3$, the following are LCM-diagrams, where $i\neq k-j+2,k+1$:
$$
\xymatrix@C=18mm@R=12mm{
  \ar[d]_{\Phi_{k,j}^{(b)}} \ar[r]^{s_{k-j+2}^{(b-1)}}
&  \ar[d]^{\Phi_{k,j-1}^{(b)}}
\\ \ar[r]_{S_{2,j}^{(b+k-j+1)}}
&
}
\qquad
\xymatrix@C=18mm@R=12mm{
  \ar[d]_{\Phi_{k,j}^{(b)}} \ar[r]^{s_{k+1}^{(b-1)}}
&  \ar[d]^{\Phi_{k+1,j+1}^{(b)}}
\\ \ar[r]_{S_{k+1,j}}
&
}
\qquad
\xymatrix@C=12mm@R=12mm{
  \ar[d]_{\Phi_{k,j}^{(b)}} \ar[r]^{s_i^{(b-1)}}
&  \ar[d]^{\Phi_{k,j}^{(b)}}
\\ \ar[r]_{s_c}
&
}
$$
In the last diagram, $s_c$ is some atom in $S_n$. Moreover, in each diagram the vertical arrow on the left is a prefix of the vertical arrow on the right.
\end{lemma}

\begin{proof}
Let us suppose first that $k-j+2=2$. In this case $k=j$, so $s_{k-j+2}^{(b-1)}=t_{b-1}$ and, if $j>3$, $\Phi_{k,j}^{(b)}=\Phi_{k,k}^{(b)}=S_{2,k}^{(b)}=t_b s_3 S_{4,k}$. Hence we have:
$$
\xymatrix@C=12mm@R=12mm{
   \ar[dd]_{t_{b-1}} \ar[r]^{t_b}
&  \ar[dd]_{t_{b+1}} \ar[r]^{s_3}
&  \ar[d]_{t_{b+1}} \ar[r]^{S_{4,k}}
&  \ar[d]^{t_{b+1}}
\\
& &  \ar[d]_{s_{3}} \ar[r]^{S_{4,k}}
& \ar[d]^{S_{3,k}}
\\
\ar[r]_{t_b}
&  \ar[r]_{s_3 t_{b+1}}
&  \ar[r]_{S_{4,k}S_{3,k-1}}
&
}
$$
We see that the bottom row is: $t_b s_3t_{b+1}S_{4,k}S_{3,k-1} = t_bs_3 S_{4,k} t_{b+1} S_{3,k-1} = S_{2,k}^{(b)}S_{2,k-1}^{(b+1)} = \Phi_{k,k-1}^{(b)}= \Phi_{k,j-1}^{(b)}$. And the rightmost column is $t_{b+1}S_{3,k}= S_{2,k}^{(b+1)}$. Hence, the first diagram in the statement holds in this case. If $j=k=3$ we only have the two leftmost rectangles, but the result also holds, as $t_bs_3t_{b+1}=S_{2,3}^{(b)}S_{2,2}^{(b-1)}=\Phi_{3,2}^{(b)}=\Phi_{k,j-1}^{(b)}$. And the rightmost column would be $t_{b+1}s_3=S_{2,3}^{(b+1)}$, as desired.

Now suppose that $k-j+2=3$. In this case we have
$$
\xymatrix@C=12mm@R=12mm{
   \ar[d]_{s_3} \ar[r]^{t_bs_3}
&  \ar[d]_{t_{b}} \ar[r]^{S_{4,k}}
&  \ar[d]_{t_{b}} \ar[rr]^{t_{b+1}S_{3,k-1}}
& & \ar[d]^{S_{2,k-1}^{(b+2)}}
\\
\ar[r]_{t_bs_3}
&  \ar[r]_{S_{4,k}}
&  \ar[rr]_{\Phi_{k-1,k-2}^{(b+1)}}
& &
}
$$
The square on the right hand side comes from the previous case, and the bottom row equals $\Phi_{k,k-2}^{(b)}$, so in this case the first diagram in the statement also holds.

Now, if $i=k-j+2>3$, we have:
$$
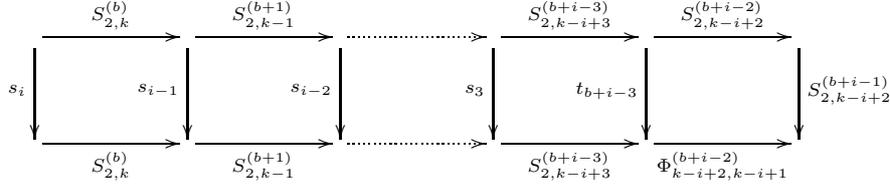
\xymatrix@C=18mm@R=12mm{
   \ar[d]_{s_i} \ar[r]^{S_{2,k}^{(b)}}
&  \ar[d]_{s_{i-1}} \ar[r]^{S_{2,k-1}^{(b+1)}}
&  \ar[d]_{s_{i-2}} \ar@{.>}[r]
&  \ar[d]_{s_3} \ar[r]^{S_{2,k-i+3}^{(b+i-3)}}
&  \ar[d]_{t_{b+i-3}} \ar[r]^{S_{2,k-i+2}^{(b+i-2)}}
&  \ar[d]^{S_{2,k-i+2}^{(b+i-1)}}
\\
   \ar[r]_{S_{2,k}^{(b)}}
&  \ar[r]_{S_{2,k-1}^{(b+1)}}
&  \ar@{.>}[r]
&  \ar[r]_{S_{2,k-i+3}^{(b+i-3)}}
&  \ar[r]_{\Phi_{k-i+2,k-i+1}^{(b+i-2)}}
&
}
$$
Notice that the top row is $\Phi_{k,j}^{(b)}$ since $k-i+2=j$. The bottom row is then $\Phi_{k,j-1}^{(b)}$, and the rightmost column is $S_{2,j}^{(b+k-j+1)}$.  Hence, the first diagram in the statement holds in every case.

We see that the vertical arrow on the left hand side of the first diagram is a prefix of the vertical arrow on the right since, by definition, $\Phi_{k,j}^{(b)} S_{2,j-1}^{(b+k-j+1)} = \Phi_{k,j-1}^{(b)}$.

When $i<k-j+2$, we have a diagram similar to the previous one, but where the vertical arrow becomes an element of $T_e$ before arriving to the last square:
$$
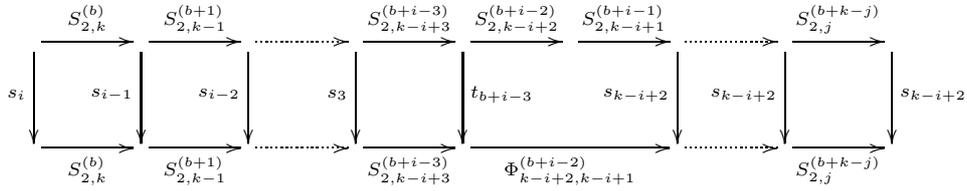
\xymatrix@C=12mm@R=12mm{
   \ar[d]_{s_i} \ar[r]^{S_{2,k}^{(b)}}
&  \ar[d]_{s_{i-1}} \ar[r]^{S_{2,k-1}^{(b+1)}}
&  \ar[d]_{s_{i-2}} \ar@{.>}[r]
&  \ar[d]_{s_3} \ar[r]^{S_{2,k-i+3}^{(b+i-3)}}
&  \ar[d]^{t_{b+i-3}} \ar[r]^{S_{2,k-i+2}^{(b+i-2)}}
&  \ar[r]^{S_{2,k-i+1}^{(b+i-1)}}
&  \ar[d]_{s_{k-i+2}} \ar@{.>}[r]
&  \ar[d]_{s_{k-i+2}} \ar[r]^{S_{2,j}^{(b+k-j)}}
&  \ar[d]^{s_{k-i+2}}
\\
   \ar[r]_{S_{2,k}^{(b)}}
&  \ar[r]_{S_{2,k-1}^{(b+1)}}
&  \ar@{.>}[r]
&  \ar[r]_{S_{2,k-i+3}^{(b+i-3)}}
&  \ar[rr]_{\Phi_{k-i+2,k-i+1}^{(b+i-2)}}
& & \ar@{.>}[r]
&  \ar[r]_{S_{2,j}^{(b+k-j)}}
&
}
$$
In this case, the top and the bottom rows are both equal to $\Phi_{k,j}^{(b)}$, so the third diagram in the statement holds, with $c=k-i+2$.

If $k-j+2<i<k+1$, the diagram is easier:
$$
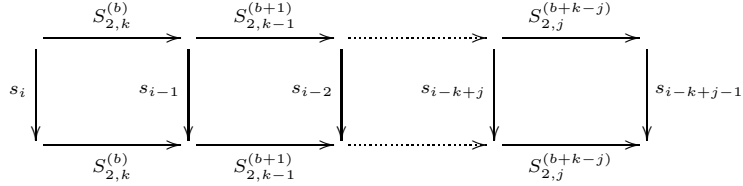
\xymatrix@C=18mm@R=12mm{
   \ar[d]_{s_i} \ar[r]^{S_{2,k}^{(b)}}
&  \ar[d]_{s_{i-1}} \ar[r]^{S_{2,k-1}^{(b+1)}}
&  \ar[d]_{s_{i-2}} \ar@{.>}[r]
&  \ar[d]_{s_{i-k+j}} \ar[r]^{S_{2,j}^{(b+k-j)}}
&  \ar[d]^{s_{i-k+j-1}}
\\
   \ar[r]_{S_{2,k}^{(b)}}
&  \ar[r]_{S_{2,k-1}^{(b+1)}}
&  \ar@{.>}[r]
&  \ar[r]_{S_{2,j}^{(b+k-j)}}
&
}
$$
Hence, the third diagram in the statement holds for $c=i-k+j-1$, in this case.

If $i>k+1$ the element $s_i$ commutes with every atom in $\Phi_{k,j}^{(b)}$, hence $s_i\vee \Phi_{k,j}^{(b)} = s_i \Phi_{k,j}^{(b)}= \Phi_{k,j}^{(b)} s_i$, and the result is true also in this case, with $c=1$. Hence, the third diagram of the statement holds in every case, and it is trivial that the vertical arrow on the left is a prefix of the vertical arrow on the right, as they coincide.

Finally, suppose that $i=k+1$. We will first see that, for every $m=3,\ldots,k$ and every $a\in \{1,\ldots,e-1\}$, we have that $S_{k+1,m+1}\vee S_{2,m}^{(a)}=S_{k+1,m+1}S_{2,m+1}^{(a)}=S_{2,m}^{(a)}S_{k+1,m}$. This is shown in the following diagram:
$$
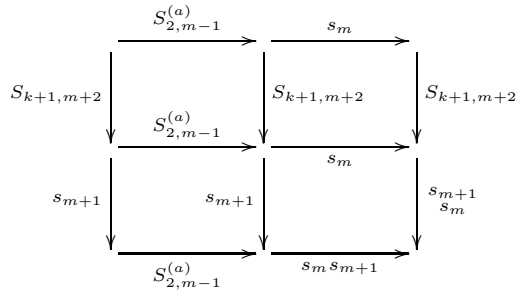
\xymatrix@C=18mm@R=12mm{
   \ar[d]_{S_{k+1,m+2}} \ar[r]^{S_{2,m-1}^{(a)}}
&  \ar[d]^{S_{k+1,m+2}} \ar[r]^{s_m}
&  \ar[d]^{S_{k+1,m+2}}
\\
   \ar[d]_{s_{m+1}} \ar[r]^{S_{2,m-1}^{(a)}}
&  \ar[d]_{s_{m+1}} \ar[r]_{s_m}
&  \ar[d]^{\scriptstyle s_{m+1}\atop \scriptstyle s_m}
\\
   \ar[r]_{S_{2,m-1}^{(a)}}
&  \ar[r]_{s_m s_{m+1}}
&
}
$$
Using this property, we finish the proof with the following:
$$
\xymatrix@C=18mm@R=12mm{
   \ar[d]_{s_{k+1}} \ar[r]^{S_{2,k}^{(b)}}
&  \ar[d]_{S_{k+1,k}} \ar[r]^{S_{2,k-1}^{(b+1)}}
&  \ar[d]_{S_{k+1,k-1}} \ar[r]^{S_{2,k-2}^{(b+2)}}
&  \ar[d]_{S_{k+1,k-2}} \ar@{.>}[r]
&  \ar[d]_{S_{k+1,j+1}} \ar[r]^{S_{2,j}^{(b+k-j)}}
&  \ar[d]^{S_{k+1,j}}
\\
   \ar[r]_{S_{2,k+1}^{(b)}}
&  \ar[r]_{S_{2,k}^{(b+1)}}
&  \ar[r]_{S_{2,k-1}^{(b+2)}}
&  \ar@{.>}[r]
&  \ar[r]_{S_{2,j+1}^{(b+k-j)}}
&
}
$$
The top row is $\Phi_{k,j}^{(b)}$, and the bottom row is $\Phi_{k+1,j+1}^{(b)}$, so the second diagram in the statement holds, as we wanted to show.

We finish by showing that the vertical arrow on the left of the second diagram in the statement is a prefix of the vertical arrow on the right. We have already shown that $\Phi_{k,j}^{(b)} S_{k+1,j} = s_{k+1} \Phi_{k+1,j+1}^{(b)}$. Now we notice that $s_t S_{2,t}^{(a)}=S_{2,t}^{(a)}s_{t+1}$ for every $t\in \{j+1,\ldots,k\}$ and every $a\in \{0,\ldots,e-1\}$. Hence
$$
  s_{k+1} \Phi_{k+1,j+1}^{(b)} = s_{k+1} (S_{2,k+1}^{(b)}S_{2,k}^{(b+1)}\cdots S_{2,j+1}^{b+k-j}) = (S_{2,k+1}^{(b)}S_{2,k}^{(b+1)}\cdots S_{2,j+1}^{b+k-j}) s_j = \Phi_{k+1,j+1}^{(b)} s_j.
$$
On the other hand, $S_{k+1,j}=S_{k+1,j+1}s_j$. Therefore
$$
\Phi_{k,j}^{(b)}S_{k+1,j+1}s_j = \Phi_{k,j}^{(b)} S_{k+1,j} = s_{k+1} \Phi_{k+1,j+1}^{(b)} = \Phi_{k+1,j+1}^{(b)} s_j.
$$
Cancelling $s_j$ we obtain that $\Phi_{k,j}^{(b)} \preccurlyeq \Phi_{k+1,j+1}^{(b)}$, as we wanted to show.

\end{proof}

We have already shown all technical lemmas, so we can finally prove the following:

\begin{proposition} The standard Garside structure of $B$ for $W=G(e,e,n)$ is support-preserving.
\end{proposition}

\begin{proof}
Let $x,y\in B^+$ be connected by a minimal positive conjugator, that is, $x^\rho=y$ where $\rho=\rho_{u}(x)$ for some atom $u$. Let $X=\operatorname{Supp}(x)$ and $Y=\operatorname{Supp}(y)$. We need to show that $\rho$ conjugates $G_X$ to $G_Y$.

Let us first suppose that $u\notin X$. The set $X\cup \{u\}$ may have several connected components, and we denote $X_1$ the subset of $X$ such that $X_1\cup \{u\}$ is the connected component containing $u$, and $X_2=X\setminus X_1$.

Notice that $\Delta_{X\cup \{u\}} = \Delta_{X_1\cup\{u\}}\Delta_{X_2}$, where the two factors commute. And that $\Delta_{X} = \Delta_{X_1}\Delta_{X_2}$, where the two factors also commute. Hence $r_{X,u}=\Delta_X^{-1}\Delta_{X\cup \{u\}}=\Delta_{X_1}^{-1}\Delta_{X_1\cup \{u\}}=r_{X_1,u}$. We will see that $\rho=r_{X_1,u}=r_{X,u}$ in all cases in which $\rho$ is a minimal simple conjugator.

Let us first suppose that $\#((X_1\cup \{u\})\cap T_e) \leq 1$. In this case we can apply \autoref{P:B(e,e,n)_minimal_conjugators_are_ribbons} to conclude that $\rho=r_{X_1,u}=r_{X,u}$. Hence, the same argument as in the proof of \autoref{P:braid_support_preserving} applies, and it follows that $\rho=r_{X,u}$ conjugates $X$ to $Y$, and then it conjugates $G_X$ to $G_Y$, as claimed.

Suppose now that $\#((X_1\cup \{u\})\cap T_e) > 1$.

We will first assume that $u=s_j\in S_n$, hence $\#(X_1\cap T_e)=\#(X\cap T_e) > 1$. Let $w=a_1a_2\cdots a_r$ be a word in $X$ representing $x$. Let $\{t_{i_1},\ldots,t_{i_p}\}=\{a_1,\ldots,a_r\}\cap T_e$. Notice that $p>1$, and that all atoms in $X\cap S_n$ must be present in $w$.

In order to compute $\rho_u(x)$, that is $\rho_{s_j}(x)$, one must compute the pre-minimal conjugators and the converging prefixes for $s_j$ and $x$. We will see that each pre-minimal conjugator is a prefix of the following one, hence the converging prefix $c_{i+1}$ will be equal to the pre-minimal conjugator $s_{i,r}$, for every $i\geq 0$, and we just need to care about the pre-minimal conjugators.

The first pre-minimal conjugator is $s_j=\sigma_{j,j,j}$. Using the relations in \autoref{L:sigmaijk_vee_atom}, we see that all pre-minimal conjugators have the form $\sigma_{i,j,k}$, and that each one is a prefix of the next one, until a pre-minimal conjugator is $\sigma_{3,j,k}$ and the next letter from $w$ is $t_a$ for some $a\in \{1,\ldots,e-1\}$. In that case, the following pre-minimal conjugator will be $\sigma_{2,j,k}^{(a)}$.

The following steps either leave the pre-minimal conjugator untouched, or increase the third index ($k$), or, when the pre-minimal conjugator has the form $\sigma_{2,j,k}^{(a)}$ and the next letter from $w$ is $t_b$ for some $b\neq a$, the following pre-minimal conjugator is $\overline{\Lambda}_{3,j,k}$ (by \autoref{L:sigma2jk_vee_tb}). Notice that, in all steps up to now, every pre-minimal conjugator is a prefix of the following one.

From this moment, we can apply \autoref{L:atom_vee_overline_Lambda}, to find that every new pre-minimal conjugator will have the form $\overline{\Lambda}_{i,j,k}$ for some $i\leq j\leq k$, each one being a prefix of the following one. Hence, all elements in the sequence $c_0\prec c_1\prec \cdots \prec c_m$ will be either $\sigma_{i,j,k}$ or $\overline{\Lambda}_{i,j,k}$ for some $i\leq j\leq k$.

It remains to notice that the pre-minimal conjugators will be modified from $c_i$ to $c_{i+1}$, unless $c_i=\overline{\Lambda}_{j,j,k}$ where $X_1\cup \{s_j\}=\{t_{i_1},\ldots,t_{i_p}\}\cup \{s_3,\ldots,s_k\}$, in which case all letters of $w$ will leave the pre-minimal conjugator untouched. Therefore $c_m=\overline{\Lambda}_{j,j,k}$, that is, $\rho_{s_j}(x)=\overline{\Lambda}_{j,j,k}$. Since $\overline{\Lambda}_{j,j,k}=r_{X_1,s_j}$ by
\autoref{L:ribbons_in_B(e,e,n)}, we have that $\rho=r_{X_1,s_j}=r_{X,s_j}$, and then $\rho$ conjugates $X$ to $Y$, so $G_X$ to $G_Y$, as we wanted to show.

Now suppose that $\#((X_1\cup \{u\})\cap T_e) > 1$ and $u=t_a\in T_e$. Notice that $X_1\cap T_e$ must have at least one element, but cannot have more than one, as in this case it has $e$ elements (as $X$ is saturated being the support of $x$), and we would have $u\in X$, a contradiction with our initial assumption. Therefore $X_1\cap T_e=\{t_{b-1}\}$ for some $b-1\in \{0,\ldots,e-1\}$. Since $X_1\cup \{u\}=X_1\cup\{t_a\}$ is connected, we must have $X_1=\{t_{b-1},s_3,\ldots,s_k\}$ for some $k\in \{2,\ldots,n\}$ (the case $k=2$ means $X_1=\{t_{b-1}\}$).

The interesting case happens when $t_a=t_{b}$. In this case we claim that $\rho=\rho_{t_b}(x) = r_{X_1,t_b}$. Hence we will have that $\rho=r_{X,t_b}$ conjugates $X$ to $Y$.

To show this claim, denote $X_j=\{s_3,\ldots,s_j\}$ and $X_j^{(b-1)}=\{t_{b-1},s_3,\cdots,s_j\}$ for $j=2,\ldots,k$. We will show that all pre-minimal conjugators for $t_b$ and $x$ have the form $\Phi_{j,i}^{(b)}$ for some $2\leq i\leq j \leq k$.

Indeed, the first pre-minimal conjugator is $t_b=\Phi_{2,2}^{(b)}$. The atoms which appear in any word representing $x$ are precisely those in $X$, which means that we can apply \autoref{L:Phi} at every step, so if some pre-minimal conjugator is $\Phi_{j,i}^{(b)}$, the following one will be either $\Phi_{j,i}^{(b)}$, or $\Phi_{j,i-1}^{(b)}$, or $\Phi_{j+1,i+1}^{(b)}$. Moreover, the pre-minimal conjugator will be modified after reading all letters from $x$, unless it is equal to $\Phi_{k,2}^{(b)}$. Since each pre-minimal conjugator is a prefix of the following one, the elements $c_0\prec c_1\prec \cdots \prec c_m$ will also have this form, and the last of these elements will be $\rho=c_m=\Phi_{k,2}^{(b)}$. Since $\Phi_{k,2}^{(b)}=r_{X_1,t_b}$ by \autoref{L:r_(X_k^((b-1)),t_b)}, it follows that $\rho=r_{X_1,t_b}=r_{X,t_b}$, so it conjugates $X$ to $Y$, as we wanted to show.

Still considering $u\notin X$, it only remains the case $X_1=\{t_{b-1},s_3,\ldots,s_k\}$ and $u=t_a$, with $a\neq b-1,b$. In this case we see that, when computing the pre-minimal conjugators as in the previous case, one obtains some pre-minimal conjugator of the form $\Phi_{2,p}^{(b)}$, and then some $c_i$ will admit $\Phi_{2,p}^{(b)}$ as a prefix. This implies that $c_i$ admits $t_{b}$ as a prefix. On the other hand, from the previous case we know that $\rho_{t_{b}}(x)=\Phi_{2,k}^{(b)}$, which does not admit $t_a$ as a prefix. Hence, it follows that $\rho_{t_a}(x)$ is not a minimal conjugating element by \autoref{L:unnecessary_arrow}.

We have then shown that, if $u\notin X$, the associated minimal positive conjugator sends the set $X$ to the set $Y$, so $\#(X)=\#(Y)$ in this case. By \autoref{P:atom_in_X_support-preserving}, when $u\in X$, the element $\rho_{u}(x)$ also sends $G_X$ to $G_Y$ (which in this case equals $G_X$). Therefore, the Garside structure is support-preserving.
\end{proof}

\subsection{Dual monoids for exceptional groups}

We fix a well-generated irreducible complex reflection group $W$ having all its reflections of order $2$,
and a Coxeter element $c \in W$. As in \autoref{sect:wellgen2refl} this defines a dual braid monoid $M(c)$ endowed with a Garside structure
with $\Delta = \mathbf{c}$. The first results we prove are valid for all these groups.

\begin{proposition}\label{P:Dual_LCM}
The dual braid monoid $M(c)$ is
LCM-Garside.
\end{proposition}
\begin{proof}
Recall that the dual monoids $M(c)$ are known to satisfy condition (\ref{eq:hypdual}) by \autoref{prop:prophypdual}.
In order to prove that they are LCM-Garside, we need to check the conditions of \autoref{D:LCM-Garside_structure}. Conditions (1) and (2) are immediate, so we need to prove (3). Let $X \subset \mathcal{A}$. Then $\delta = \Delta_X$ is
a simple element, so that $\delta = \mathbf{c}_0$ for $c_0 \in [1,c]$. By \autoref{prop:parabsreflwellgens} (3) this is
a Coxeter element, attached to some parabolic subgroup $W_0$ of $W$. Then $G_{\delta}$ is the subgroup of $B$
generated by the atoms dividing $\delta$, that is the elements $\mathbf{s}$ for $s \in W$ a reflection such
that $\ell(s) + \ell(s^{-1} c_0) =
\ell(c_0)$. By \autoref{prop:parabsreflwellgens} (3) such reflections necessarily belong to $W_0$.

On the other hand, we have an injective morphism $M(c_0) \to M(c)$ which identifies $M(c_0)$ to a submonoid
of $M(c)$ and induces an injective group homomorphism $B_0 \to B$, where $B_0$ is the group of fractions of $M(c_0)$. Thus $B_0$ is identified with $G_{\delta}$. Moreover, we have $G_{\delta}^+ = G_{\delta} \cap G^+ = B_0 \cap M(c) = M(c_0) \subset M(c)$ by \autoref{prop:dualinjectionmonoid} (3).

Therefore we need to prove that, for any
$c_0 \in [1,c]$, $\Div_{M(c_0)}(\mathbf{c_0}) = M(c_0) \cap \Div_{M(c)}(\mathbf{c})$. Clearly $\Div_{M(c_0)}(\mathbf {c_0}) \subset M(c_0) \cap \Div_{M(c)}(\mathbf c)$. Let $\mathbf m \in M(c_0) \cap \Div_{M(c)}(\mathbf c)$. We need to prove that $\mathbf m$ divides $\mathbf c_0$. For this we introduce the natural monoid morphism $\pi : M(c) \to W$, which maps $M(c_0)$ onto $W_0$,
and $\mathbf{w}$ to $w$ for every $w \in [1,c]$. From this we get that $\pi(\mathbf m) = m \in W_0\cap [1,c] = [1,c_0]$
by property (\ref{eq:hypdual}). Now, clearly the length of $\mathbf m$ in $M(c_0)$ is the same as the length of $\mathbf m$ in $M(c)$,
that is the length of $m$ with respect to set $\mathcal{R}$ of reflections of $W$ (see Proposition 3.6 (3)).
Since $m \in W_0$ this length is equal to  the length with respect to set $\mathcal{R}_0 = \mathcal{R} \cap W_0$ of reflections of $W_0$ by \autoref{prop:parabsreflwellgens} (2). So, we have $m \in [1,c_0]$
and $\ell(m) = \ell(\mathbf m)$. By \autoref{lem:caractsimplesintervalles} this implies $\mathbf m \in \Div_{M(c_0)}(\mathbf{c}_0)$ and this concludes the proof.
\end{proof}

This implies in particular that the standard parabolic subgroups in the Garside sense are exactly the ones previously introduced in Proposition \ref{prop:dualstandardparabs}, so that the two concepts of parabolic subgroups coincide in that case, too.

We now want to check whether these Garside structures are support-preserving.

The following is an easy consequence of \autoref{prop:dual_simple}.
\begin{lemma}\label{L:dual_lcm}
  Let $\mathbf a,\mathbf b$ be two distinct atoms of $M(c)$. If $\mathbf{ba}$ is simple, then $\mathbf a\vee \mathbf b=\mathbf{ba}=\mathbf{as}=\mathbf{tb}$, for some atoms $\mathbf s\neq \mathbf a$ and $\mathbf t\neq \mathbf b$.
\end{lemma}

\begin{proof}
Since $\mathbf{ba}$ is simple, it is balanced by \autoref{prop:dual_simple}~(2). It admits $\mathbf a$ as a suffix, so it also admits $\mathbf a$ as a prefix. Since $\ell(\mathbf{ba})=2$ by \autoref{prop:dual_simple}~(1), it follows that $\mathbf{ba}=\mathbf{as}$ for some atom $\mathbf s$. Recall that $\mathbf a$ and $\mathbf b$ are two distinct atoms, so $\mathbf a \vee \mathbf b$ cannot be a single atom, and its length must be greater than 1. Since we have a common multiple of $\mathbf a$ and $\mathbf b$ of length 2, it follows that $\mathbf a\vee \mathbf b=\mathbf{ba}=\mathbf{as}$. Moreover, $\mathbf s \neq \mathbf a$, since a simple element is square-free, by \autoref{prop:dual_simple}~(3).

In the same way, $\mathbf b$ is a prefix of $\mathbf{ba}$, hence it is also a suffix. By the above arguments, there exists an atom $\mathbf t\neq \mathbf b$ such that $\mathbf{ba}=\mathbf{tb}$.
\end{proof}

\begin{lemma}\label{L:dual_chain_of_subsets}
  Let $G$ be one of the 2-reflection groups $G_{24}$, $G_{27}$, $G_{29}$, $G_{33}$, $G_{34}$, equipped with the dual Garside structure $(G,M(c),\Delta)$, where $\pi(\Delta)=c$. Let $X\subsetneqq \mathcal A$ be a saturated set of atoms. Let $A_0=\{\mathbf a\in \mathcal A;\ \Delta_X \mathbf a\preccurlyeq \Delta \}$. Then, there is a sequence $A_0\subset A_1 \subset \cdots \subset A_r =\mathcal A\setminus X$ such that, for every $i>0$ and every atom $\mathbf a\in A_i$, there is a (possibly trivial) saturated subset $X'\subsetneqq X$ such that:
\begin{enumerate}
  \item For every $\mathbf b\in X'$, one has $\mathbf a\preccurlyeq \mathbf b\backslash \mathbf a$.

  \item For every $\mathbf b\in X\setminus X'$, there exists $\mathbf d\in A_{i-1}$ such that $\mathbf d\preccurlyeq \mathbf b\backslash \mathbf a$.
\end{enumerate}
\end{lemma}

\begin{proof}

For every such saturated set $X$,
one applies by computer the following algorithm, with the notation that $a \leftarrow b$
indicates that the variable $a$ receives the content of $b$, and where $\mathcal P_{sat}(\mathcal A)$ denotes de saturated subsets of $\mathcal A$.

\begin{enumerate}

\item Compute $\mathcal{A}_{X} = \{\mathbf{a} \in \mathcal{A}\ | \ \mathbf{a} \prec \Delta_X^{-1}\Delta \}$.

\item Set $\mathcal{B} = \mathcal{A} \setminus ( X \cup \mathcal{A}_{X})$, and
$\mathcal{D} = \{ (\mathbf a,\mathbf b,V)\in \mathcal{B}\times X \times \mathcal{P}_{sat}(\mathcal{A})  \ | \ \Delta_V= \mathbf b\backslash \mathbf a = \mathbf{b}^{-1} (\mathbf{a}\vee \mathbf{b}) \}$.

\item Set $\mathcal{N} = \emptyset$.

\item Repeat

\begin{enumerate}

\item $\mathcal{A}_{X} \leftarrow \mathcal{A}_{X} \cup \mathcal{N}$ 
       \hfill {\tt \# $\mathcal A_X$ will be $A_i$ for $i=0,1,2\ldots$}

\item $\mathcal{D} \leftarrow \{ (\mathbf a,\mathbf b,V) \in \mathcal{D} \ | \ \mathcal{A}_{X} \cap V = \emptyset \}$
    
     \hfill {\tt \# Remove those satisfying $\mathbf d \preccurlyeq \mathbf b \backslash \mathbf a$ for some $\mathbf d\in A_{i-1}$}

\item For every $\mathbf a \in \mathcal{A}$, do :

\begin{itemize}

\item If every $(\mathbf a,\mathbf b,V) \in \mathcal{D}$ satisfies $\mathbf a \in V$ (that is, $\mathbf a \preccurlyeq \mathbf b\backslash \mathbf a$) \emph{and}
    
    there is some saturated $X' \subsetneq X$
    such that every $(\mathbf a,\mathbf b,V)$ satisfies $\mathbf b \in X'$, then
    
    $\mathcal{D} \leftarrow \{ (\mathbf z,\mathbf b,V) \in \mathcal{D} \ | \ \mathbf z \neq \mathbf a \}$

    \hfill {\tt \# Remove if all the remaining cases satisfy condition (1).}

\end{itemize}

\item $\mathcal{N} \leftarrow \{ \mathbf z \in \mathcal{B}\setminus \mathcal{A}_{X}\ | \ \forall (\mathbf a,\mathbf b,V) \in \mathcal{D}\ \mathbf z \neq \mathbf a
\}$

   \hfill {\tt \# The new atoms to add to $\mathcal{A}_X$ are those completely removed  from $\mathcal{D}$.}

\end{enumerate}

\item until $\mathcal{N} = \emptyset$.

\item if $\mathcal{D} = \emptyset$ then return \textsf{true} else return \textsf{false}.

\end{enumerate}

The fact that this algorithm terminates in all these cases with the statement \textsf{true} proves the claim.
\end{proof}

\begin{proposition}
 Let $G$ be one of the 2-reflection groups $G_{24}$, $G_{27}$, $G_{29}$, $G_{33}$, $G_{34}$. The dual Garside structure $(G,M(c),\Delta)$ is support-preserving.
\end{proposition}

\begin{proof}
Let $\mathbf x\in M(c)$ be a positive element, and let $X=\operatorname{Supp}(\mathbf x)$. If $X=\mathcal A$, the structure is support preserving by \autoref{P:atom_in_X_support-preserving}. Hence we can assume that $X\subsetneqq \mathcal A$.

Consider the chain of subsets $A_0\subset A_1\subset \cdots \subset A_r = \mathcal A\setminus X$ described in \autoref{L:dual_chain_of_subsets}.

Let $\mathbf a\in A_0$. By definition, $\Delta_X \mathbf a$ is a simple element. Given $\mathbf b\in X$, we know that $\mathbf b$ is a suffix of $\Delta_X$, hence $\mathbf{ba}$ is a suffix of the simple element $\Delta_X \mathbf a$, so $\mathbf{ba}$ is simple. By \autoref{L:dual_lcm}, $\mathbf a\vee \mathbf b=\mathbf{ba}=\mathbf{as}$ for some atom $\mathbf s$. This holds for every $\mathbf b\in X$, therefore conjugation by $\mathbf a$ sends $X$ to some subset $Y\subset \mathcal A$. It follows that $\mathbf a^{-1}\mathbf{xa}$ is a positive element $\mathbf y$, and this implies that $\rho_{\mathbf a}(\mathbf x)=\mathbf a$. So, in order to see that this minimal positive conjugator preserves the support, we need to show that $\operatorname{Supp}(\mathbf y)=Y$.

Consider a subset $X'=\{\mathbf x_1,\ldots,\mathbf x_m\}\subset X$ such that $\mathbf x_1\vee \cdots \vee \mathbf x_m =\Delta_X$. Let us see that $\mathbf x_1 \mathbf a \vee \cdots \vee \mathbf x_m \mathbf a=\Delta_X\mathbf a$. Since $\Delta_X\mathbf a$ is simple and each $\mathbf x_i$ is a suffix of $\Delta_X$, it follows that $\mathbf x_i \mathbf a$ is a suffix, hence a prefix of $\Delta_X\mathbf a$, for $i=1,\ldots,m$. Therefore $\mathbf x_1 \mathbf a \vee \cdots \vee \mathbf x_m \mathbf a \preccurlyeq \Delta_X \mathbf a$. But then we have that $\Delta_X = \mathbf x_1 \vee \cdots \vee \mathbf x_m \preccurlyeq \mathbf x_1 \mathbf a \vee \cdots \vee \mathbf x_m \mathbf a \preccurlyeq \Delta_X\mathbf a$. Since $\ell(\Delta_X\mathbf a)=\ell(\Delta_X)+1$ (because the Garside structure is homogeneous), if an element $\alpha$  satisfies $\Delta_X\preccurlyeq \alpha \preccurlyeq \Delta_X\mathbf a$, we must have either $\ell(\alpha)=\ell(\Delta_X)$ or $\ell(\alpha)=\ell(\Delta_X)+1$. In the former case $\alpha=\Delta_X$, and in the latter $\alpha=\Delta_X\mathbf a$, since two positive elements of the same length, one being prefix of the other, must be equal. But in our case $\mathbf x_1\mathbf a$ cannot be a prefix of $\Delta_X$, since $\mathbf a \notin X$ and $X$ is saturated. Therefore $\mathbf x_1 \mathbf a \vee \cdots \vee \mathbf x_m \mathbf a \neq \Delta_X$, and then $\mathbf x_1 \mathbf a \vee \cdots \vee \mathbf x_m \mathbf a=\Delta_X\mathbf a$, as claimed.

We keep considering $X'=\{\mathbf x_1,\ldots,\mathbf x_m\}\subset X$ such that $\mathbf x_1\vee \cdots \vee \mathbf x_m =\Delta_X$. For $i=1,\ldots,m$, let $\mathbf y_i=\mathbf a^{-1}\mathbf x_i \mathbf a \in Y$, and denote $Y'=\{\mathbf y_1,\ldots,\mathbf y_m\}\subset Y$. Then we have:
$$
  \mathbf y_1\vee \cdots \vee \mathbf y_m
= \mathbf a^{-1}\mathbf x_1 \mathbf a \vee \cdots \vee \mathbf a^{-1}\mathbf x_m \mathbf a
= \mathbf a^{-1}(\mathbf x_1 \mathbf a \vee \cdots \vee \mathbf x_m \mathbf a) =\mathbf a^{-1}\Delta_X \mathbf a.
$$
Notice that the final result of the equality does not depend on $X'$, but on $X$. If $X'=X$ we have $Y'=Y$, by construction. Hence the above equality shows that $\Delta_Y=\mathbf a^{-1}\Delta_X\mathbf a$. But then, for every other subset $X'$ we obtain $\mathbf y_1\vee \cdots \vee \mathbf y_m = \Delta_Y$ as well.

Let us see that $Y$ is a saturated set of atoms.  Let $\mathbf s$ be an atom such that $\mathbf s\preccurlyeq \Delta_Y$. We need to show that $\mathbf s\in Y$. From $\mathbf s\preccurlyeq \Delta_Y$ it follows that $\mathbf a \mathbf s \preccurlyeq \mathbf a \Delta_Y = \Delta_X \mathbf a$. Since the latter is a simple element, it follows that $\mathbf a\mathbf s$ is simple. Hence, by \autoref{L:dual_lcm}, there exists an atom $\mathbf t$ such that $\mathbf{as}=\mathbf{ta}$. Now $\mathbf{ta}$ is a prefix, hence a suffix of $\Delta_X \mathbf a$, and this implies that $\mathbf t$ is a suffix of $\Delta_X$. Since $X$ is saturated, $\mathbf t\in X$. Therefore $\mathbf s = \mathbf a^{-1}\mathbf{ta} \in \mathbf a^{-1}X\mathbf a = Y$. This shows that $Y$ is saturated.

Now let $\mathbf u_1\cdots \mathbf u_k$ be a decomposition of $\mathbf x$ as a product of atoms (notice that $\mathbf u_i\in X$ for every $i$). For $i=1,\ldots,k$, let $\mathbf v_i=\mathbf a^{-1}\mathbf u_i \mathbf a$. Then $\mathbf y=\mathbf v_1\cdots \mathbf v_k$. We know that $\operatorname{Supp}(\mathbf x)=X$, that is, $\Delta_X=\mathbf u_1 \vee \cdots \vee \mathbf u_k$. From the above arguments, taking $X'=\{\mathbf u_1,\ldots,\mathbf u_k\}$ and $Y'=\{\mathbf v_1,\ldots,\mathbf v_k\}$ (it does not matter if some atoms are repeated), we have $\mathbf v_1 \vee \cdots \vee \mathbf v_k=\Delta_Y$, where $Y$ is saturated. Hence $\operatorname{Supp}(\mathbf y)=Y$.

We know that $\rho_{\mathbf a}(\mathbf x)=\mathbf a$ conjugates $X$ to $Y$, so it conjugates $G_X$ to $G_Y$. Hence $\rho_{\mathbf a}(\mathbf x)$ preserves the support if $\mathbf a\in A_0$.

Now suppose that $\mathbf a \in A_i\setminus A_0$ for some $i>0$. We claim that $\rho_{\mathbf a}(\mathbf x)$ is not a minimal positive conjugator. We show this claim by induction on $i$.

Suppose that $i=1$, and let $X'$ be the subset of $X$ described in \autoref{L:dual_chain_of_subsets}. Write $\mathbf x=\mathbf u_1\cdots \mathbf u_k$ as a product of atoms. Notice that not all these atoms can belong to $X'$, otherwise $\operatorname{Supp}(\mathbf x)\subset X'\subsetneqq X$, which is not possible. Let $j$ be the first index such that $\mathbf u_j\in X\setminus X'$. We proceed to compute the pre-minimal conjugators, starting with $\mathbf s_{0,0}=\mathbf a$. If $\mathbf a\preccurlyeq \mathbf s_{0,t}$ for some $t<j$, then $\mathbf a \preccurlyeq \mathbf u_t\backslash \mathbf a \preccurlyeq \mathbf u_t \backslash \mathbf s_{0,t} = \mathbf s_{0,t+1}$. This shows that $\mathbf a\preccurlyeq \mathbf s_{0,j-1}$. Then, as $\mathbf u_j\in X\setminus X'$, there exists $\mathbf d\in A_0$ such that $\mathbf d \preccurlyeq \mathbf u_j \backslash \mathbf a \preccurlyeq \mathbf u_j \backslash \mathbf s_{0,j-1} = \mathbf s_{0,j}$. So $\mathbf d$ is a prefix of a pre-minimal conjugator. But then we know that all subsequent pre-minimal conjugators will admit $\mathbf d$ as a prefix, since $\mathbf u \backslash \mathbf d = \mathbf d$ for every $\mathbf u\in X$ and every $\mathbf d\in A_0$. Therefore, we will have $\mathbf d \preccurlyeq \mathbf s_{0,k}$, which implies that $\mathbf d$ is a prefix of the converging prefix $c_1$. We already know from the previous case that $\rho_{\mathbf d}(\mathbf x)=\mathbf d$, so it does not admit $\mathbf a$ as a prefix. Therefore, by \autoref{L:unnecessary_arrow}, $\rho_{\mathbf a}(\mathbf x)$ is not a minimal positive conjugator for $\mathbf x$.

Now suppose that $i>1$ and that the claim holds for smaller values of $i$. Let $\mathbf a\in A_i\setminus A_0$ and, as above, let $X'$ be the subset of $X$ described in \autoref{L:dual_chain_of_subsets}. Using the above argument, we know that some pre-minimal conjugator for $\mathbf a$ will admit an atom $\mathbf b\in A_{i-1}$ as a prefix. Applying \autoref{L:dual_chain_of_subsets} again, all subsequent pre-minimal conjugators will admit an atom from $A_{i-1}$ as a prefix. Hence, a converging prefix for $\mathbf a$ and $\mathbf x$ admits some atom $\mathbf s\in A_{i-1}$ as a prefix. It follows that $\mathbf s \preccurlyeq \rho_{\mathbf a}(\mathbf x)$. If $\mathbf s\in A_0$, the argument in the previous paragraph shows that $\rho_{\mathbf a}(\mathbf x)$ is not a minimal positive conjugator. If, on the contrary, $\mathbf s\in A_{i-1}\setminus A_0$, from $\mathbf s \preccurlyeq \rho_{\mathbf a}(\mathbf x)$ we obtain $\rho_{\mathbf s}(\mathbf x)\preccurlyeq \rho_{\mathbf a}(\mathbf x)$. But we know that $\rho_{\mathbf s}(\mathbf x)$ is not a minimal positive conjugator, by induction hypothesis. Therefore $\rho_{\mathbf a}(\mathbf x)$ is not a minimal positive conjugator, and the claim is shown.

Since we have $A_r=\mathcal A \setminus X$ by \autoref{L:dual_chain_of_subsets}, we have shown that for every atom $\mathbf a \in \mathcal A \setminus (A_0\cup X)$ the element $\rho_{\mathbf a}(\mathbf x)$ is not a minimal positive conjugator, while for every $\mathbf a\in A_0$ the element $\rho_{\mathbf a}(\mathbf x)=\mathbf a$ is a positive minimal conjugator and preserves the support, sending the set $X$ to a set $Y$ of the same size. By \autoref{P:atom_in_X_support-preserving}, the atoms in $X$ also preserve the support, so the Garside structure is support-preserving.
\end{proof}

In order to better understand the above result, we will use $G_{24}$ as an example. See~\cite{MARINPFEIFFER} for a detailed explanation of this Garside structure of $G_{24}$.

There are 14 atoms in the mentioned Garside structure of $G_{24}$, named $b_1,\ldots,b_{14}$, and the Garside element is $c=b_1b_2b_3$. A presentation of the Garside monoid (and of the Garside group), using these atoms is given in \autoref{F:G24_relations} (note that these are {\bf not} LCM diagrams).

\begin{figure}[ht]\caption{Defining relations of the dual braid monoid in type $G_{24}$.}\label{F:G24_relations}
$$
\xymatrix@C=6mm@R=6mm{
   b_{1} \ar[r] & b_{2} \ar[dl]
\\ b_{4} \ar[u]
}
\quad
\xymatrix@C=6mm@R=6mm{
   b_{8} \ar[r] & b_{5} \ar[dl]
\\ b_{7} \ar[u]
}
\quad
\xymatrix@C=6mm@R=6mm{
   b_{12} \ar[r] & b_{6} \ar[dl]
\\ b_{14} \ar[u]
}
\quad
\xymatrix@C=6mm@R=6mm{
   b_{2} \ar[r] & b_{11} \ar[dl]
\\ b_{10} \ar[u]
}
\quad
\xymatrix@C=6mm@R=6mm{
   b_{5} \ar[r] & b_{1} \ar[dl]
\\ b_{3} \ar[u]
}
\quad
\xymatrix@C=6mm@R=6mm{
   b_{6} \ar[r] & b_{8} \ar[dl]
\\ b_{9} \ar[u]
}
\quad
\xymatrix@C=6mm@R=6mm{
   b_{11} \ar[r] & b_{12} \ar[dl]
\\ b_{13} \ar[u]
}
$$
$$
\xymatrix@C=6mm@R=6mm{
   b_{6} \ar[r] & b_{13} \ar[d]
\\ b_{1} \ar[u] & b_{10} \ar[l]
}
\quad
\xymatrix@C=6mm@R=6mm{
   b_{11} \ar[r] & b_{4} \ar[d]
\\ b_{8} \ar[u] & b_{3} \ar[l]
}
\quad
\xymatrix@C=6mm@R=6mm{
   b_{1} \ar[r] & b_{7} \ar[d]
\\ b_{12} \ar[u] & b_{9} \ar[l]
}
\quad
\xymatrix@C=6mm@R=6mm{
   b_{8} \ar[r] & b_{14} \ar[d]
\\ b_{2} \ar[u] & b_{13} \ar[l]
}
\quad
\xymatrix@C=6mm@R=6mm{
   b_{12} \ar[r] & b_{10} \ar[d]
\\ b_{5} \ar[u] & b_{4} \ar[l]
}
\quad
\xymatrix@C=6mm@R=6mm{
   b_{2} \ar[r] & b_{3} \ar[d]
\\ b_{6} \ar[u] & b_{7} \ar[l]
}
\quad
\xymatrix@C=6mm@R=6mm{
   b_{5} \ar[r] & b_{9} \ar[d]
\\ b_{11} \ar[u] & b_{14} \ar[l]
}
\quad
$$
\end{figure}
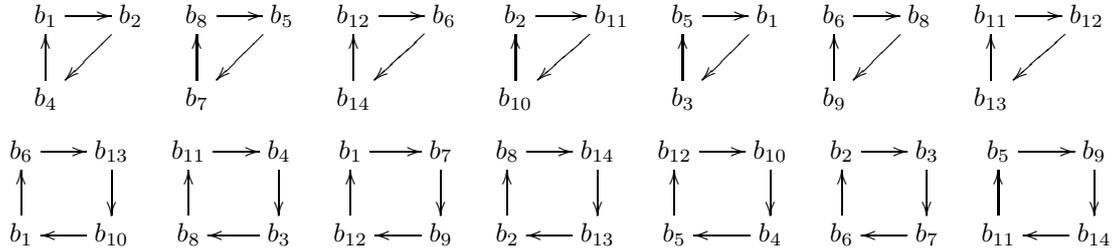

The relations are defined as follows. In every diagram of \autoref{F:G24_relations}, all possible products of two consecutive elements (following the arrows) determine the same element in the monoid (and in the group). For instance, the first diagram produces the relations $b_1b_2=b_2b_4=b_4b_1$.

Notice that all relations are homogeneous (with elements of length 2). Hence, as the Garside element ($b_1b_2b_3$) has length 3, all proper simple elements have length at most 2. If two different atoms lie in the same diagram, their least common multiple has length 2 (for example, $b_6\vee b_{10}=b_6b_{13}=b_{10}b_1$). On the other hand, if there is no diagram containing two given atoms, their least common multiple is $c$ (this happens in $G_{24}$, but it does not in other groups considered in this section).

It is worth mentioning that we have ordered the diagrams in \autoref{F:G24_relations} in such a way that, if we conjugate an element in \autoref{F:G24_relations} by $c$, we obtain the corresponding element in the diagram on its right (if the element appears in the rightmost diagram, the corresponding element is in the leftmost one). So conjugation by $c$ permutes the atoms as follows:
$$
  b_{1} \rightarrow b_{8} \rightarrow b_{12} \rightarrow b_{2} \rightarrow b_{5} \rightarrow b_{6} \rightarrow b_{11} \rightarrow b_{1},
$$
$$
  b_{4} \rightarrow b_{7} \rightarrow b_{14} \rightarrow b_{10} \rightarrow b_{3} \rightarrow b_{9} \rightarrow b_{13} \rightarrow b_{4}.
$$

Now notice that the nontrivial standard parabolic subgroups of $G_{24}$, with respect to this structure, are the following ones: cyclic subgroups generated by one atom; subgroups generated by the atoms in a diagram of \autoref{F:G24_relations}; the whole group $G_{24}$.

\begin{example}
Let $\mathbf x = b_1^4$. We have $X=\operatorname{Supp}(x)=\{b_1\}$. Then $A_0=\{b_2,b_3,b_6,b_7\}$. If $\mathbf a\in A_0$, for instance if $\mathbf a=b_2$, the computation of the pre-minimal conjugators for $\mathbf a$ and $\mathbf x$ is the following:
$$
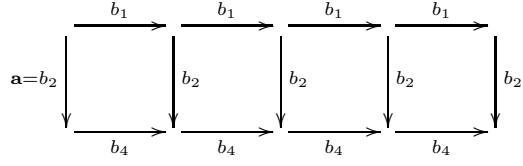
\xymatrix@C=12mm@R=12mm{
   \ar[d]_{\mathbf a= b_2} \ar[r]^{b_1}
&  \ar[d]^{b_2} \ar[r]^{b_1}
&  \ar[d]^{b_2} \ar[r]^{b_1}
&  \ar[d]^{b_2} \ar[r]^{b_1}
&  \ar[d]^{b_2}
\\ \ar[r]_{b_4}
&  \ar[r]_{b_4}
&  \ar[r]_{b_4}
&  \ar[r]_{b_4}
&
}
$$
Hence, the only converging prefix is $\mathbf a$, and $\rho_{\mathbf a}=\mathbf a=b_2$. In this case, one has $\mathbf x^{\mathbf a}=(b_1^4)^{\mathbf a}=b_4^4= \mathbf y$, and $\mathbf a$ conjugates $X=\{b_1\}$ to $Y=\{b_4\}$.

On the other hand, if $\mathbf a\notin X\cup A_0$, for instance if $\mathbf a=b_4$, we have the following:
$$
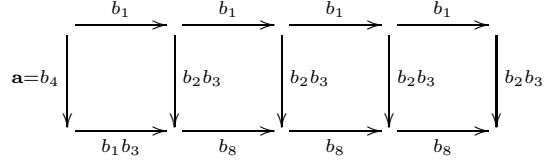
\xymatrix@C=12mm@R=12mm{
   \ar[d]_{\mathbf a= b_4} \ar[r]^{b_1}
&  \ar[d]^{b_2b_3} \ar[r]^{b_1}
&  \ar[d]^{b_2b_3} \ar[r]^{b_1}
&  \ar[d]^{b_2b_3} \ar[r]^{b_1}
&  \ar[d]^{b_2b_3}
\\ \ar[r]_{b_1b_3}
&  \ar[r]_{b_8}
&  \ar[r]_{b_8}
&  \ar[r]_{b_8}
&
}
$$
From this diagram we see that $b_2\in A_0$ is a prefix of a converging prefix for $\mathbf a$ (the converging prefix $c_1=b_4\vee b_2b_3 = \Delta$), hence $b_1=\rho_{b_2}(\mathbf x)\preccurlyeq \rho_{\mathbf a}(\mathbf x)$. Hence $\rho_{\mathbf a}(x)$ is not minimal.
\end{example}

\begin{example}
Let $\mathbf x = b_1b_1b_6b_{10}$. In this case $X=\{b_1,b_6,b_{10},b_{13}\}$, and $A_0=\{b_2\}$.
If $\mathbf a\in A_0$, that is if $\mathbf a=b_2$, the computation of the pre-minimal conjugators for $\mathbf a$ and $\mathbf x$ is:
$$
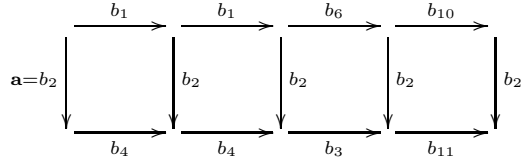
\xymatrix@C=12mm@R=12mm{
   \ar[d]_{\mathbf a= b_2} \ar[r]^{b_1}
&  \ar[d]^{b_2} \ar[r]^{b_1}
&  \ar[d]^{b_2} \ar[r]^{b_6}
&  \ar[d]^{b_2} \ar[r]^{b_{10}}
&  \ar[d]^{b_2}
\\ \ar[r]_{b_4}
&  \ar[r]_{b_4}
&  \ar[r]_{b_3}
&  \ar[r]_{b_{11}}
&
}
$$
We see that $\rho_{\mathbf a}(\mathbf x)=\mathbf a=b_2$, and that $\mathbf a$ conjugates $X=\{b_1,b_6,b_{10},b_{13}\}$ to $Y=\{b_4,b_3,b_{11}, b_8\}$. Hence it conjugates $\mathbf x$ to $\mathbf y=b_4b_4b_{3}b_{11}$, whose support is $Y$.

On the other hand, if $\mathbf a\notin X\cup A_0$, for instance if $\mathbf a=b_3$, we have the following:
$$
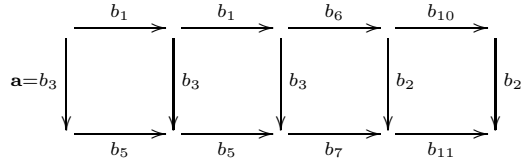
\xymatrix@C=12mm@R=12mm{
   \ar[d]_{\mathbf a= b_3} \ar[r]^{b_1}
&  \ar[d]^{b_3} \ar[r]^{b_1}
&  \ar[d]^{b_3} \ar[r]^{b_6}
&  \ar[d]^{b_2} \ar[r]^{b_{10}}
&  \ar[d]^{b_2}
\\ \ar[r]_{b_5}
&  \ar[r]_{b_5}
&  \ar[r]_{b_7}
&  \ar[r]_{b_{11}}
&
}
$$
In this case we see that $b_3$ appears as a pre-minimal conjugator as long as the letters of $\mathbf x$ equal $b_1$, but it changes to a pre-minimal conjugator starting with $b_2$ when one encounters the letter $b_6$, which does not belong to $A_0$. From that moment on, the forthcoming pre-minimal conjugators start with $b_2$, so the converging prefix $c_1=b_3\vee b_2$ admits $b_2$ as a prefix. Therefore, $\rho_{\mathbf a}(\mathbf x)$ is not a minimal positive conjugator in this case.
\end{example}

We have then shown that the dual Garside structures of the groups $G_{24}$, $G_{27}$, $G_{29}$, $G_{33}$ and $G_{34}$ are support-preserving LCM-structures. Therefore, by \autoref{T:parabolic_closure_exists}, every element admits a parabolic closure.

\section{Intersection of parabolic subgroups}
\label{sect:intersectparabs}

In this section we will adapt the arguments in~\cite{CGGW} to show that, in any irreducible complex braid group
 (except possibly in $G_{31}$), the intersection of parabolic subgroups is a parabolic subgroup.

\subsection{The main Garside-theoretic argument}\label{Garside_argument}

We will consider first the complex braid groups which admit a support-preserving LCM-Garside structure $(B,B^+,\Delta)$, as shown in the previous section. We also notice that, in all those Garside structures, we have the following two properties:
\begin{enumerate}
 \item {\em Relations are homogeneous}. This implies that all positive words representing an element $x\in B^+$ have the same length, and we can use that quantity as the length of $x$.

 \item {\em Simple elements are square-free}. This means that, for every atom $a\in \mathcal A$, the element $aa$ is not simple. This implies that $aa$ can never be a factor of a simple element. It also implies that, given two simple elements $\alpha$ and $\beta$, if the set $\operatorname{Suff}(\alpha)$ of atoms which are suffixes of $\alpha$ contains the set $\operatorname{Pref}(\beta)$ of atoms which are prefixes of $\beta$, then the product  $\alpha\beta$ is in left normal form.

\end{enumerate}

The first property is clear, because the presentations providing the Garside structures are
homogeneous. The second one is, by \autoref{prop:intervalissquarefree}, a consequence of the fact that all the Garside structures involved here can be defined as interval monoids associated to a generating set for $W$ made of involutions. This covers the case of the monoids attached to real reflection groups, to the $G(e,e,n)$, and to
the well-generated exceptional 2-reflection groups.

In this section, $B$ will be one of the mentioned complex braid groups, and $(B,B^+,\Delta)$ will denote the studied Garside structure, which is a homogeneous, support-preserving LCM-Garside structure. Moreover, the parabolic subgroups for these Garside structures correspond to the
parabolic subgroups of $B$ which were defined independently on the choice of a Garside structure, so
there is no ambiguity here when using this term. We will not use the square-free property in this section (although it was used in the previous section to show that the Garside structure is support-preserving).

Given such a complex braid group $B$ endowed with the corresponding Garside structure, we start by defining, for every parabolic subgroup $H$, a special element $z_H\in H$. In the case of an irreducible parabolic subgroup, we shall see that it coincides with our previous definition.
We start by defining it for standard parabolic subgroups. Notice that this definition coincides with the ones in \cite{GODELLE2003} and \cite{CGGW} for Artin--Tits groups of spherical type.

\begin{definition}
Let $B_0=G_X$ be a standard parabolic subgroup of $B$. We define $z_{B_0}=(\Delta_X)^e$, where $e$ is the smallest positive integer such that $(\Delta_X)^e$ is central in $B_0$.
\end{definition}

\begin{proposition}\label{P:conjugate standard_H_and_z_H}
Let $B_1=G_X$ and $B_2=G_Y$ be two standard parabolic subgroups of $B$ which are conjugate. Then $\Delta_X$ is conjugate to $\Delta_Y$. Moreover, if $z_{B_1}=(\Delta_X)^e$ then $z_{B_2}=(\Delta_Y)^e$, so $z_{B_1}$ is conjugate to $z_{B_2}$.
\end{proposition}

\begin{proof}
Let $c_1\in B$ be such that $(B_1)^{c_1}=B_2$, and consider $(\Delta_X)^{c_1}\in B_2$. Since $B_2$ is standard, we can conjugate $(\Delta_X)^{c_1}$ by an element $c_2\in B_2$ (a conjugating element for iterated swaps), so that $(\Delta_X)^{c_1c_2}$ is recurrent. But $\RC(\Delta_X)=C_+(\Delta_X)$, hence $(\Delta_X)^{c_1c_2}$ is positive.

We can now assume that $B_1$ is not equal to $B$, otherwise $B_1=B_2=B$ and the result is trivial. It follows that $\inf(\Delta_X)=0$ and $\sup(\Delta_X)=1$, and that one cannot conjugate $\Delta_X$ to an element of bigger infimum. Hence $\inf((\Delta_X)^{c_1c_2})=0$. We can now apply iterated {\it decycling}, to decrease the supremum of $(\Delta_X)^{c_1c_2}$ to its minimal possible value (which is 1). All the conjugating elements will belong to $B_2$, hence there is $c_3\in B_2$ such that $(\Delta_X)^{c_1c_2c_3}$ is a simple element in $B_2=G_Y$. Therefore $1\preccurlyeq (\Delta_X)^{c_1c_2c_3}\preccurlyeq \Delta_Y$.

Since relations are homogeneous in $B$, it follows that $|\Delta_X|=|(\Delta_X)^{c_1c_2c_3}| \leq |\Delta_Y|$. Now, inverting the roles of $B_1$ and $B_2$, the same proof yields $|\Delta_Y|\leq |\Delta_X|$. Therefore $|\Delta_X|=|\Delta_Y|$.

Recall that $(\Delta_X)^{c_1c_2c_3}\preccurlyeq \Delta_Y$. Since both positive elements have the same length, they are equal. This shows that $\Delta_X$ is conjugate to $\Delta_Y$.

Now notice that $(G_X)^{c_1c_2c_3}=(G_Y)^{c_2c_3}=G_Y$. Hence the center of $G_X$ is conjugated by $c_1c_2c_3$ to the center of $G_Y$. Since $((\Delta_X)^k)^{c_1c_2c_3}=(\Delta_Y)^k$ for every $k>0$, it follows that if $z_{B_1}=(\Delta_X)^e$ then $z_{B_2}=(\Delta_Y)^e$. Clearly this implies that $z_{B_1}$is conjugate (by $c_1c_2c_3$) to $z_{B_2}$.
\end{proof}

Now we can define $z_{B_0}$ for an arbitrary parabolic subgroup $B_0$:

\begin{definition}
Let $B_0$ be a parabolic subgroup of $B$. Let $c\in B$ be such that $B_0^c$ is standard, say $B_0^c=G_X$. Then we define $z_{B_0}=(z_{G_X})^{c^{-1}}$.
\end{definition}

\begin{proposition}
Under the above conditions, the element $z_{B_0}$ is well defined.
\end{proposition}

\begin{proof}
Suppose that $c_1,c_2\in B$ are such that $B_0^{c_1}=G_X$ and $B_0^{c_2}=G_Y$. We need to show that $(z_{G_X})^{c_1^{-1}}=(z_{G_Y})^{c_2^{-1}}$. But this follows from \autoref{P:conjugate standard_H_and_z_H}, since we have $(G_X)^{c_1^{-1}c_2}=G_Y$, and this implies that $(z_{G_X})^{c_1^{-1}c_2}=z_{G_Y}$.
\end{proof}

\begin{proposition} \label{prop:zB0welldefined} If $B_0$ is an irreducible parabolic subgroup of $B$, then $z_{B_0}$ is indeed the canonical
element defined in the introduction. In particular, it is independent on the choice of a Garside structure.
\end{proposition}
\begin{proof} We only need to check this for a standard parabolic subgroup. The proof is then easy in all cases,
the expression of the canonical element $z_{B_0}$ being known in all cases.
\end{proof}

Notice that the above statement is specific to the case where $B_0$ is irreducible. For instance, if $B = \mathcal{B}_5$ with generators $\sigma_1,\sigma_2,\sigma_3,\sigma_4$, and $B_0 = \langle \sigma_1 \rangle \times
\langle \sigma_3,\sigma_4 \rangle$, then the element `$z_{B_0}$' considered in this section would be $\sigma_1^2 (\sigma_3\sigma_4\sigma_3)^2$
when considering the classical Garside structure, while the one associated to the dual braid monoid would
yield $z_{B_0} = \sigma_1^3 (\sigma_3\sigma_4)^3 = \sigma_1^3 (\sigma_3\sigma_4\sigma_3)^2$.

\smallskip

The following are three important properties of $z_{B_0}$:

\begin{proposition}\label{P:PC(z_H)}
Let $B_0$ be a parabolic subgroup of $B$. Then $\PC(z_{B_0})=B_0$.
\end{proposition}

\begin{proof}
If $B_0$ is standard, say $B_0=G_X$ (for a saturated $X$), then $\Delta_X$ is a positive element whose support is $X$. Hence $\PC(\Delta_X)=G_X$. By \autoref{T:parabolic_closure_of_powers}, $\PC(z_{B_0})=\PC((\Delta_X)^e)=G_X=B_0$.

If $B_0$ is not standard, let $c$ be such that $B_0^c$ is standard, so $\PC(z_{B_0^c})=B_0^c$. By definition, $z_{B_0}=(z_{B_0^c})^{c^{-1}}$. Then, by \autoref{L:parabolic_closure_for_conjugates} and by the standard case, $\PC(z_{B_0})=\PC((z_{B_0^c})^{c^{-1}})= \PC(z_{B_0^c})^{c^{-1}}=(B_0^c)^{c^{-1}}=B_0$.
\end{proof}

\begin{proposition}\label{P:conjugate_H_and_z_H}
Let $B_1$ and $B_2$ be two parabolic subgroups of $B$. For every $c\in B$, one has $(B_1)^c=B_2$ if and only if $(z_{B_1})^c=z_{B_2}$.
\end{proposition}

\begin{proof}
Suppose that $(B_1)^c=B_2$. Let $d\in B$ be such that $B_1^{d}=G_X$. Then $c^{-1}d$ is such that $B_2^{c^{-1}d}=G_X$ and, by definition, $z_{B_1}=(z_{G_X})^{d^{-1}}$ and $z_{B_2}=(z_{G_X})^{d^{-1}c}$. Therefore $(z_{B_1})^c=z_{B_2}$.

Conversely, suppose that $(z_{B_1})^c=z_{B_2}$. Then, by \autoref{L:parabolic_closure_for_conjugates}, $\PC(z_{B_1})^c=\PC(z_{B_2})$ and, by \autoref{P:PC(z_H)}, $(B_1)^c=B_2$.
\end{proof}

\begin{proposition}\label{P:z_positive_is_standard}
Let $B_0$ be a parabolic subgroup of $B$. Then $B_0$ is standard if and only if $z_{B_0}$ is positive.
\end{proposition}

\begin{proof}
If $B_0$ is standard, say $B_0=G_X$, we know that $z_{B_0}$ is a positive power of $\Delta_X$, so it is positive.

Conversely, suppose that $z_{B_0}$ is positive. Then it is recurrent and, by \autoref{T:parabolic_closure_for_recurrent}, $PC(z_{B_0})=G_{X}$, where $X=\operatorname{Supp}(z_{B_0})$. On the other hand, by \autoref{P:PC(z_H)} we know that $PC(z_{B_0})=B_0$. Therefore $B_0=G_X$ is standard.
\end{proof}

Let us then prove that the intersection of parabolic subgroups is a parabolic subgroup. We need the following definition, taken from~\cite{CGGW}:

\begin{definition}{\rm (See \cite[Definition 9.3]{CGGW})}
For every element $\gamma\in B$ we define an integer $\varphi(\gamma)$ as follows: Conjugate $\gamma$ to $\gamma'\in \RC(\gamma)$. Let $U=\operatorname{Supp}(\gamma')$. Then let $\varphi(\gamma)=|\Delta_U|$, the length of the positive element $\Delta_U$ as a word in the atoms.
\end{definition}

Notice that we used $\RC(\gamma)$, instead of the set $RSSS_{\infty}(\gamma)$ which is used in~\cite{CGGW}. This is because it is theoretically much simpler, and satisfies all the needed properties.

\begin{proposition}\label{P:varphi(n)}\cite[Proposition 9.4]{CGGW}
The integer $\varphi(\gamma)$ is well defined. Moreover, if $\gamma$ is conjugate to a positive element, then $\varphi(\gamma)=|\Delta_X|$, where $X=\operatorname{Supp}(\beta)$ for any positive element $\beta$ conjugate to~$\gamma$.
\end{proposition}

\begin{proof}
We adapt the proof in~\cite{CGGW}. Suppose that $\gamma', \gamma''\in \RC(\gamma)$, and let $U=\operatorname{Supp}(\gamma')$ and $V=\operatorname{Supp}(\gamma'')$. Then $\PC(\gamma')=G_{U}$ and $\PC(\gamma'')=G_{V}$. Since $\gamma'$ and $\gamma''$ are conjugate, so are its parabolic closures $G_U$ and $G_V$. Hence, by \autoref{P:conjugate standard_H_and_z_H}, the elements $\Delta_{U}$ and $\Delta_{V}$ are also conjugate. The homogeneous relations of $B$ imply that $|\Delta_{U}|=|\Delta_{V}|$. This shows that $\varphi(\gamma)$ is well defined.

If $\gamma$ is conjugate to a positive element, the result follows from \autoref{P:recurrent=positive}, as in that case $\RC(\gamma)=C^+(\gamma)$.
\end{proof}

Finally, we can adapt the main result in~\cite{CGGW} to our case, but simplifying considerably the proof. We will not need results analogous to Lemma~9.1 and Lemma~9.2 in~\cite{CGGW}, and the final argument is much simpler.

\begin{theorem}
Let $B_1$ and $B_2$ be two parabolic subgroups of $B$ with Garside structure $(B,B^+,\Delta)$ as above. Then $B_1\cap B_2$ is also a parabolic subgroup.
\end{theorem}

\begin{proof}
First, we can assume that $B_1\cap B_2\neq \{1\}$, otherwise the result is trivial. Then we consider a nontrivial element $\alpha\in B_1\cap B_2$ such that $\varphi(\alpha)$ is maximal, and we will show that $B_1\cap B_2 = \PC(\alpha)$, so $B_1\cap B_2$ is a parabolic subgroup.

Up to conjugation of $\alpha$, $B_1$ and $B_2$ by the same element, we can assume that $\alpha$ is a recurrent element. Hence, if $X=\operatorname{Supp(\alpha)}$, we have that $\PC(\alpha)=G_X$, and we need to show that $G_X=B_1\cap B_2$.

Since $\alpha$ belongs to the parabolic subgroup $B_i$, for $i=1,2$, it follows that $\PC(\alpha)\subset B_i$ for $i=1,2$. Hence $G_X=\PC(\alpha)\subset B_1\cap B_2$, so we have one inclusion. Notice that this also yields $\Delta_X\in B_1\cap B_2$.

Let us then show the inclusion $B_1\cap B_2\subset G_X$. Let $w\in B_1\cap B_2$. In order to show that $w\in G_X$, we will consider the sequence of elements $\beta_m=w(\Delta_X)^m\in B_1\cap B_2$, for $m>0$. We will first show, using an argument taken from~\cite{CGGW}, that $\beta_m$ is conjugate to a positive element, for $m$ big enough.

Consider the reduced left-fraction decomposition $w=a^{-1}b$, where $a$ is a product of $r$ simple elements and $b$ is a product of $s$ simple elements. We have that $\beta_m=a^{-1}b(\Delta_X)^m$, so the denominator $D_L(\beta_m)$ is the product of at most $r$ simple elements, and the numerator $N_L(\beta_m)$ is the product of at most $s+m$ simple elements. Notice that these numbers could decrease, but never increase, if one applies iterated swaps to $\beta_m$.

We can now conjugate $\beta_m\in B_1\cap B_2$ to a recurrent element $\widetilde\beta_m$ by iterated swaps. Let $U_m=\operatorname{Supp}(\widetilde\beta_m)$, and denote $n=|\Delta_X|=\varphi(\alpha)$. By definition, $|\Delta_{U_m}|=\varphi(\beta_m)\leq \varphi(\alpha)=n$.

As we pointed out, if we consider the reduced left-fraction decomposition $\widetilde \beta_m=x_m^{-1}y_m$, then $x_m$ is the product of at most $r$ simple elements, and $y_m$ is the product of at most $s+m$ simple elements. Since $\widetilde\beta_m$ is recurrent, $x_m,y_m\in G_{U_m}^+$. It follows that the simple factors in the normal forms of $x_m$ and $y_m$ have length at most $n$. Let $N_m$ be the number of simple factors, in $y_m$, whose length is smaller than $n$. Since $y_m$ has at most $s+m$ simple elements, we have that $|y_m|\leq (s+m)n-N_m$. It follows that the exponent sum of $\widetilde \beta_m$ written as a product of atoms and their inverses,
that is $\ell(\widetilde{\beta}_m)$ for $\ell : B \onto \Z$ the natural homomorphism,
satisfies $\ell(\widetilde\beta_m)\leq (s+m)n-N_m$. But since $\widetilde\beta_m$ is a conjugate of
$\beta_m$ we have $\ell(\widetilde \beta_m)=\ell(\beta_m)=\ell(w)+\ell((\Delta_X)^m)= \ell(w)+nm$. We then have $\ell(w)+nm\leq (s+m)n-N_m$, so $N_m\leq sn-\ell(w)$, where the right-hand side of the inequality does not depend on $m$.
Therefore, the number of `small' factors in $y_m$ is bounded by a number which is independent of $m$. This implies, in particular, that there exists $M>0$ such that, for every $m>M$, some factor in the left normal form of $y_m$ has length $n$. Since the simple elements in $G_{U_m}$ have length at most $n$, it follows that some factor in the left normal form of $y_m$ equals $\Delta_{U_m}$, and that $|\Delta_{U_m}|=n$. But $x_m\in G_{U_m}^+$, which implies that $x_m=1$ since otherwise there would be cancellation between $x_m$ and $y_m$. Hence, for $m$ big enough ($m>M$), we have $\widetilde \beta_m=y_m\in B^+$.

Notice that, for every $m>M$, the left normal form of $\widetilde \beta_m$ is $\Delta_{U_m}^{m-N_m} s_1\cdots s_{N_m}$. Let us denote $R_m=s_1\cdots s_{N_m}$ the non-Delta part of the left normal form of $\widetilde \beta_m$. Since $N_m$ is bounded above by a number independent of $m$, it follows that the sequence $\{R_m\}_{m\geq 1}$ can take finitely many possible values.

Also, for every $m>M$, let $c_m$ be the minimal positive element such that $c_m \beta_m c_m^{-1}$ is positive. We know from Remark~\autoref{R:minimal_to_recurrent} that $c_m$ is precisely the conjugating element for iterated swaps, hence $c_m \beta_m c_m^{-1}=\widetilde \beta_m$. It is well known from \cite{BIRMANKOLEE} that, if $\beta_m$ is not positive (in which case $c_m=1$), there exists a conjugating element from $\beta_m$ to a positive element, whose length is bounded above by $|\inf(\beta_m)|\cdot |\Delta|$. Since $|\inf(\beta_m)|\leq r$ for every $m$ for which $\beta_m$ is not positive, it follows that the length of the positive element $c_m$ is bounded above by an integer not depending on $m$. That is, the sequence $\{c_m\}_{m\geq 1}$ can take a finite number of possible values.

Let $e>0$ be such that $(\Delta_V)^e$ is central in $G_V$ for every subset $V$ of atoms. Since the sequences $\{R_{em}\}_{m\geq 1}$ and $\{c_{em}\}_{m\geq 1}$ can take finitely many possible values, and there are only a finite number of subsets of atoms, there exist integers $m_1$ and $m_2$, with $M < m_1 < m_2$, such that $c_{em_1}=c_{em_2}$, $R_{em_1}=R_{em_2}$ and $U_{em_1}=U_{em_2}$. Let us denote $c=c_{em_1}=c_{em_2}$, $R=R_{em_1}=R_{em_2}$, $U=U_{em_1}=U_{em_2}$ and $t=m_2-m_1$. We have $\widetilde \beta_{em_1}=c\beta_{em_1}c^{-1}$ and:
$$
   \widetilde \beta_{em_2} = c\beta_{em_2}c^{-1} = c\beta_{em_1}\Delta_{X}^{et} c^{-1} = (c \beta_{em_1}c^{-1})(c\Delta_X^{et} c^{-1}) = \widetilde \beta_{em_1}(c\Delta_X^{et} c^{-1}).
$$
On the other hand, since $R=R_{em_1}=R_{em_2}$, it follows that $N:=N_{em_1}=N_{em_2}$. Hence:
$$
  \widetilde \beta_{em_2} = \Delta_{U}^{em_2-N}R = \Delta_U^{em_2-em_1}\Delta_U^{em_1-N}R = \Delta_U^{et} \widetilde \beta_{em_1} = \widetilde \beta_{em_1} \Delta_U^{et}.
$$
Therefore, $c\Delta_X^{et} c^{-1} = \Delta_U^{et}$. That is, $(\Delta_U^{et})^c = \Delta_X^{et}$.

Now notice that $\PC(\Delta_U^{et})=G_U$ and $\PC(\Delta_X^{et})=G_X$. Hence, by \autoref{L:parabolic_closure_for_conjugates}, $(G_U)^{c}=\PC(\Delta_U^{et})^{c} = \PC((\Delta_U^{et})^{c}) = \PC(\Delta_X^{et}) = G_X$. On the other hand, we know that $\PC(\widetilde \beta_{em_1})=G_U$. Hence, again by \autoref{L:parabolic_closure_for_conjugates}, $\PC(\beta_{em_1})=\PC((\widetilde \beta_{em_1})^c)= \PC(\widetilde \beta_{em_1})^c = (G_U)^c = G_X$. This implies that $\beta_{em_1}\in G_X$. Hence, as $w = \beta_{em_1} \Delta_X^{-em_1}$, we finally obtain that $w\in G_X$, as we wanted to show.
\end{proof}

\subsection{Characterization of adjacency for the Garside groups}

In the introduction we defined the {\em curve graph} $\Gamma$, as the graph whose vertices are irreducible parabolic subgroups, and where two such subgroups $B_1$ and $B_2$ are adjacent if either $B_1\subset B_2$, or $B_2\subset B_1$, or $B_1\cap B_2=[B_1,B_2]=\{1\}$.

We will see in this section that this notion of adjacency, which is very natural from the point of view of curves in a surface, can be characterized very easily in terms of the elements $z_{B_1}$ and $z_{B_2}$.

As we did in \autoref{Garside_argument}, we consider $B$ to be one of the irreducible complex braid groups whose Garside structure $(B,B^+,\Delta)$ has been studied in this paper, so it is a homogeneous, support-preserving LCM-Garside structure.

\begin{proposition}
 Two irreducible parabolic subgroups $B_1,B_2\subset B$ with Garside structure $(B,B^+,\Delta)$ are adjacent in $\Gamma$ if and only if $z_{B_1}$ and $z_{B_2}$ commute.
\end{proposition}

\begin{proof}
Suppose that $B_1$ and $B_2$ are adjacent. If $B_1\subset B_2$, since $z_{B_2}$ is central in $B_2$ it follows that $[z_{B_1},z_{B_2}]=1$. If $B_2\subset B_1$ the argument is the same. Finally, if $B_1\cap B_2=[B_1,B_2]=1$, every element of $B_1$ commutes with every element of $B_2$, hence $z_{B_1}$ commutes with $z_{B_2}$ also in this case.

Conversely, suppose that $z_{B_1}$ commutes with $z_{B_2}$. The first part of the proof follows the arguments of \cite[Theorem 2.2]{CGGW}: we will simultaneously conjugate $B_1$ and $B_2$ to standard parabolic subgroups.

Since $B_1$ is a parabolic subgroup, it is conjugate to a standard parabolic subgroup $G_T$ for some subset of atoms $T\subset \mathcal A$. We can simultaneously conjugate $B_1$ and $B_2$ to assume that $B_1=G_T$ is already standard. Hence $z_{B_1}$ is a positive element.

Now consider the reduced left-fraction decomposition of $z_{B_2}$, say $z_{B_2}=a^{-1}b$. By definition of reduced left-fraction decomposition, $a\wedge b=1$. Left-multiplying this equality by $a^{-1}$ we obtain $1\wedge a^{-1}b = a^{-1}$, that is, $1\wedge z_{B_2} = a^{-1}$.

It is clear that, as $z_{B_1}$ is positive (hence recurrent), $z_{B_1}$ conjugated by 1 is positive (hence recurrent). And, since $z_{B_1}$ and $z_{B_2}$ commute, $z_{B_1}$ conjugated by $z_{B_2}$ is also positive (hence recurrent). Therefore, by \autoref{P:convexity}, $(z_{B_1})^{a^{-1}}= (z_{B_1})^{1\wedge z_{B_2}}$ is recurrent (hence positive, as the recurrent conjugates of $z_{B_1}$ are precisely the positive conjugates of $z_{B_1}$).

It follows that, if we simultaneously conjugate $z_{B_1}$ and $z_{B_2}$ by $a^{-1}$, we replace $z_{B_1}$ by a positive conjugate, and we replace $z_{B_2}$ by $\phi(z_{B_2})$.

The interested reader may see that all the requirements of \cite[Lemma~9 and Theorem~3]{CUMP} are fulfilled (interchanging the prefix and the suffix orders) in the groups we are working with, and this implies that $\phi(z_{B_2})$ is already a positive element (the theorem says that $a$ is the shortest positive element such that $a z_{B_2} a^{-1}$ is positive).

But we do not need to use the results in~\cite{CUMP} to finish this proof. We have that the obtained conjugates of $z_{B_1}$ and $z_{B_2}$ satisfy the same initial hypotheses: they commute and the first one is positive. Hence we can simultaneously conjugate both elements by the left denominator of the second one, obtaining a new pair of elements satisfying the same hypotheses. Iterating the process, since we know that $z_{B_2}$ is conjugate to a positive element and that such a conjugate can be obtained by iterated swaps, we finally obtain simultaneous conjugates of $z_{B_1}$ and $z_{B_2}$ which are both positive (and commute).

We denote $z_1$ the positive conjugate of $z_{B_1}$ and $z_2$ the positive conjugate of $z_{B_2}$. Being conjugates of some $z_{B_i}$, we have that $z_i=z_{B_i'}$ for some parabolic subgroup $B_i'$, for $i=1,2$. Since $z_1$ and $z_2$ are positive, it follows from \autoref{P:z_positive_is_standard} that $B_1'$ and $B_2'$ are both standard parabolic subgroups, as we wanted to show.

We can then assume, up to a simultaneous conjugation, that $B_1=G_X$ and $B_2=G_Y$ are standard parabolic subgroups, where $X$ and $Y$ are saturated subsets of atoms.

Let us denote $X'$ and $Y'$ the images of $X$ and $Y$, respectively, on the reflection group $W$. Let $W_X$ and $W_Y$ be the subgroups of $W$ generated by $X'$ and $Y'$, respectively. These are irreducible parabolic subgroups of $W$, by definition.

Now, by \autoref{theo:zzcommW}, we have three possibilities. The first one is that $W_X\subset W_Y$, which implies that the set $X'$ is contained in $Y'$. Since the map from $B$ to $W$ is injective on the atoms, this implies that $X\subset Y$, hence $G_X\subset G_Y$.

The second possibility is that $W_Y\subset W_X$. By the same argument, we obtain that $Y\subset X$, and then $G_Y\subset G_X$.

The third and final possibility is that $W_X\cap W_Y=1$. In this case $X'\cap Y'=\varnothing$, hence $X\cap Y=\varnothing$. This implies that $G_X\cap G_Y=1$, since for every element $x\in G_X\cap G_Y$, its reduced left-fraction decomposition $a^{-1}b$ is such that $a,b\in G_X^+\cap G_Y^+$ and then, as $X$ and $Y$ are saturated, every representative of $a$ (and also of $b$) is a word in $X$ and also in $Y$. This is only possible if $a=b=1$. Hence $x=1$.

It remains to show that, in this third case, $[G_X,G_Y]=1$. This is equivalent to say that every element of $X$ commutes with every element of $Y$. Let $x\in X$ and $y\in Y$, and let $x'$ and $y'$ be their images in $W$. We know that $W_X$ and $W_Y$ commute, hence $x'y'=y'x'$. Since $x'$ and $y'$ are different (as $W_X\cap W_Y=1$) and have order 2, the element $x'y'=y'x'$ has order 2, and then $xy=yx$ is a relation in the considered interval monoid (associated to either an Artin group, $G(e,e,n)$, or an exceptional group). Therefore $G_X$ and $G_Y$ commute.
\end{proof}

\subsection{General complex braid groups}
\label{sect:intersectgeneralCBG}

We can then prove in general \autoref{maintheo:parabolicclosure}, \autoref{maintheo:intersectparabolics}
and finally \autoref{maintheo:curvecomplex} of the
introduction. We start with \autoref{maintheo:parabolicclosure} and \autoref{maintheo:intersectparabolics}. If $W$ is isodiscriminantal to some reflection group
$W'$ to which the methods of the previous sections can be applied, we have already proved \autoref{maintheo:parabolicclosure} and \autoref{maintheo:intersectparabolics} in
this case. Indeed, the case of an arbitrary collection of parabolics $(B_i)_{i \in I}$ is
easily deduced from the case $|I| = 2$, first for the case where $I$ is finite, and then by noticing
that, the $B_1 \cap B_2 \subsetneq B_1$ is possible only if the rank of $B_1 \cap B_2$ is smaller
than the rank of $B_1$, so that $\bigcap_{i \in I} B_i = \bigcap_{i \in J} B_i$ for some finite $J \subset I$.

This covers all well-generated reflection groups, so we now consider the other ones which are not
$G_{31}$.

If $W$ belongs to the general series $G(de,e,n)$ for $d > 1$ (or is isodiscriminantal to one
of these groups), by the description of the parabolics in \autoref{prop:parabsGdeen},
\autoref{maintheo:intersectparabolics} follows immediately from the case of $G(de,1,n)$.
But then, the existence of a parabolic closure in \autoref{maintheo:parabolicclosure} is an immediate
consequence of \autoref{maintheo:intersectparabolics} (consider the intersection of all parabolic subgroups
containing the given element). Moreover, if $x \in B$ is such that $x^m$ belongs to some parabolic subgroup
$B_0$, writing $B_0 = B \cap \hat{B}_0$ with $\hat{B}_0$ some parabolic subgroup of the braid group
$\hat{B}$ of type $G(de,1,n)$, from $x^m \in \hat{B}_0$ we get readily $x \in \hat{B}_0$
hence $x \in \hat{B}_0 \cap B = B_0$, and we get \autoref{maintheo:parabolicclosure}.

Therefore both theorems are proved for the infinite series. We now prove the third one. Let $B_0$ be
an irreducible parabolic subgroup of $B$, that we write as $\hat{B}_0 \cap B$ for $\hat{B}_0$ an irreducible
parabolic subgroup of $\hat{B}$ by \autoref{prop:parabsGdeen}.

First notice that such a parabolic $\hat{B}_0$ is \emph{unique}. Indeed, if we had another one $\hat{B}'_0$, then $\hat{B}_0 \cap \hat{B}'_0$
would be another parabolic of $\hat{B}$ containing $B_0$, and therefore having the same rank as $\hat{B}_0$. By \autoref{cor:B0inclusB1egal} this
proves $\hat{B}_0 = \hat{B}_0 \cap \hat{B}'_0 = \hat{B}'_0$ and the uniqueness of such a $\hat{B}_0$.

Setting $m_0 = |Z(\hat{W}_0)|/|Z(W_0)|$
we have $z_{B_0} = z_{\hat{B}_0}^{m_0}$ by \autoref{prop:parabsGdeen}. Denoting $\PC$ and $\widehat{\PC}$ the parabolic closures taken inside $B$ and $\hat{B}$, respectively, we have
$$
\PC(z_{B_0}) = B \cap \widehat{\PC}(z_{B_0}) = B \cap \widehat{\PC}(z_{\hat{B}_0}^{m_0})
 = B \cap \widehat{\PC}(z_{\hat{B}_0}) = B \cap \hat{B}_0 = B_0.
$$
Finally, if $B_1,B_2$ are two irreducible parabolic subgroups of $B$, for each $i=1,2$ we have $B_i = B \cap \hat{B}_i$
for some uniquely defined irreducible parabolic subgroup $\hat{B}_i$, and we have $z_{B_i} = z_{\hat{B}_i}^{m_i}$, for some $m_i \geq 1$. First assume we have $g \in B$ such that $B_1^g = B_2$. Then $\hat{B}_1^g$ is a parabolic
subgroup of $\hat{B}$ such that $\hat{B}_1^g \cap B = (\hat{B}_1 \cap B)^g = B_1^g = B_2 = \hat{B}_2 \cap B$,
whence $\hat{B}_1^g = \hat{B}_2$ by uniqueness of the parabolic subgroup of $\hat{B}$ corresponding to $B_2$.
By \autoref{maintheo:curvecomplex} for $\hat{B}$ it follows that $z_{\hat{B}_1}^g = z_{\hat{B}_2}$.
Now notice that $m_1=m_2$, as the images of $\hat{B}_i$ and of the $B_i$ inside $W$ are also conjugates
under the image of $g$. Therefore $z_{B_1}^g = z_{B_2}$.

 Assume now that $[z_{B_1},z_{B_2}]=1$. This means $[z_{\hat{B}_1}^{m_1},z_{\hat{B}_2}^{m_2}]=1$.
But then,
$$
\widehat{\PC}(z_{\hat{B}_1})^{z_{\hat{B}_2}^{m_2}}
=
\widehat{\PC}(z_{\hat{B}_1}^{m_1})^{z_{\hat{B}_2}^{m_2}}
=
\widehat{\PC}((z_{\hat{B}_1}^{m_1})^{z_{\hat{B}_2}^{m_2}})
=
\widehat{\PC}(z_{\hat{B}_1}^{m_1})
=
\widehat{\PC}(z_{\hat{B}_1})
$$
hence, by
\autoref{P:conjugate_H_and_z_H}, we get $[z_{\hat{B}_1},z_{\hat{B}_2}^{m_2}]=1$. Then,
$$
\widehat{\PC}(z_{\hat{B}_2})^{z_{\hat{B}_1}}
=
\widehat{\PC}(z_{\hat{B}_2}^{m_2})^{z_{\hat{B}_1}}
=
\widehat{\PC}((z_{\hat{B}_2}^{m_2})^{z_{\hat{B}_1}})
=
\widehat{\PC}(z_{\hat{B}_2}^{m_2})
=
\widehat{\PC}(z_{\hat{B}_2})
$$
and applying again \autoref{P:conjugate_H_and_z_H}, we get $[z_{\hat{B}_1},z_{\hat{B}_2}]=1$.
Then \autoref{maintheo:curvecomplex} applied to $\hat{B}$ implies that, either $\hat{B}_1 \subset \hat{B}_2$,
or $\hat{B}_2 \subset \hat{B}_1$, or $\hat{B}_1 \cap \hat{B}_2 = [\hat{B}_1,\hat{B}_2] = \{ 1 \}$.
In the first two cases, $\hat{B}_i \subset \hat{B}_j$ implies immediately $B_i = \hat{B}_i \cap B \subset
\hat{B}_j \cap B = B_j$, while in the third case we get $B_1 \cap B_2 \subset \hat{B}_1 \cap \hat{B}_2 = \{ 1 \}$
and $[B_1,B_2] \subset [\hat{B}_1,\hat{B}_2] = \{ 1 \}$. This completes the proof of \autoref{maintheo:curvecomplex} for the infinite series.

\subsection{Special cases in rank 2}

It remains to consider the exceptional groups $G_{12}$, $G_{13}$ and $G_{22}$. Since they have rank $2$,
we only need to prove \autoref{maintheo:parabolicclosure}. Indeed, assuming it to hold, let us consider
two parabolic subgroups $B_1,B_2$. Then, either one of them is $\{1 \}$ or $B$, in which case $B_1 \cap B_2 \in \{ B_1,B_2\}$
is obviously a parabolic subgroup, or they have both rank $1$. Therefore, either $B_1 \cap B_2 = \{ 1\}$,
or $B_1 \cap B_2$ contains a nontrivial element $x$, and we have $\{ 1 \} \subsetneq \PC(x) \subset B_i$
for $i=1,2$. But since $\PC(x)$ has rank $1$, by \autoref{cor:B0inclusB1egal} this implies $B_1 = \PC(x) = B_2$ and thus $B_1\cap B_2 = B_1=B_2$
is a parabolic subgroup. From this one gets \autoref{maintheo:intersectparabolics}, so we only need
to prove \autoref{maintheo:parabolicclosure} in order to get both theorems.

For $G_{12}$ and $G_{22}$ we use the description of \autoref{sect:descrparabsrank2}, and the fact that the presentations given there provide a Garside structure, with $\Delta = stus$ and $\Delta = stust$, respectively. It is obvious that these structures are LCM-Garside and it can be checked that they are support-preserving, in order to apply \autoref{T:parabolic_closure_exists}. However, it is quicker (and somewhat simpler) to apply the following property, actually shared by all the Garside monoids used in this paper.

\begin{proposition} \label{prop:dmmcrit} (\cite{DMM}, Proposition 2.2) Such monoids $M$ satisfy that, for any atom $r$ and
$x,y \in M$, if $r^n x = xy$ for some $n > 0$, then
$y = t^n$ for some atom $t$ such that $r x = xt$.
\end{proposition}

This property has the immediate consequence that, if $r$ is an atom of $M$, and $x \in M$ centralizes some $r^n$ for $n > 0$, then $r^n = q^n$ for some atom $q$. But here $\lcm(r,q)= \Delta$ if $r\neq q$, hence $r^n = q^n$ implies $r = q$, and then $rx = xr$. This proves that the centralizer of $r^n$ is equal to the centralizer of $r$.

Let $x \in B$ being nontrivial. We want to prove that it is contained inside a unique minimal parabolic. Since we are in rank $2$, the only proper parabolics containing $x$ are cyclic subgroups generated by distinguished braided reflections, and since all such distinguished braided reflections are conjugates, we can assume that
$x$ is a conjugate of some nontrivial power $s^k$ of $s$. Up to exchanging $x$ with $x^{-1}$ we can assume $k > 0$. Therefore we can assume $x = s^k$ for some $k$.

Then, if $x = s^k$ were also contained in
some other parabolic
$\{c s^n c^{-1} ; n \in \Z \}$ we would have $x = cs^nc^{-1}$ for some $n$.
But since $n$ is equal to $\ell(x) = k$ we get $n= k$.
Moreover, up to
multiplying $c$ with some central power of $\Delta$ we can assume that $c \in B^+$. Hence $c$ centralizes $s^n$ inside $B^+$, which implies that $c$ centralizes $s$. This proves that the parabolic closure of $x$ is well-defined, and equal to $\langle s \rangle$. Since this is independent of $k > 0$, this is also the
parabolic closure of every $x^m$ for $m \geq 1$, whence $\PC(x^m) = \PC(x)$.

For $G_{13}$, we proved in \autoref{sect:descrparabsrank2} that, inside its braid group $B$, which can be identified
with the Artin group of type $I_2(6)$, $B = \langle a,b \ | \ (ab)^3 = (ba)^3 \rangle$,
with Garside structure $B^+ = \langle a,b \ | \ (ab)^3 = (ba)^3 \rangle^+$, $\Delta = (ab)^3$,
the proper parabolic subgroups of $B$ are conjugates of either $\langle b^{-1} \rangle$ or to
$\langle \Delta a^{-2} \rangle$. Obviously, the second one is not a parabolic subgroup for the Garside structure,
but we shall nevertheless be able to use this structure in order to deal with this case.

Let $x \in B$ being non trivial. We want to prove that it is contained inside a unique minimal parabolic. Since we are in rank $2$, we can assume that $x$ is conjugate to a power of either $b$ or
 $\Delta. a^{-2}$.

 We first prove that it cannot belong to two such parabolics of different type, that is one conjugate of $\langle b \rangle$ and one conjugate of
 $\langle \Delta. a^{-2} \rangle$. The reason for this is that the abelianization map $B \to \Z^2$
mapping $a$ to $(1,0)$ and $b$ to $(0,1)$ is injective on such cyclic subgroups and maps subgroups of the first kind to $\Z (1,0)$ and subgroups of the second kind to $\Z(1,3)$. Since these two subgroups of $\Z^2$ have trivial intersection this proves this claim.

Now assume that $x$ belongs to two subgroups conjugates of $\langle b \rangle$. Without loss of generality we can assume $x = b^k$ for some $k > 0$ and $x = c b^k c^{-1}$ for some $c \in B$. Up to
multiplying $c$ with some central power of $\Delta$ we can assume that $c \in B^+$. Then $c b^k = b^k c$ hence by \autoref{prop:dmmcrit} we have $cb=bc$ and the two subgroups are the same. Also, $\PC(x^m) = \PC(x)$ for
every $m \geq 1$.

Finally assume that $x$ belongs to two
subgroups conjugates of $\langle \Delta a^{-2} \rangle$. Without loss of generality we can assume $x = (\Delta.a^{-2})^k = \Delta^k. a^{-2k}$, and
$x = c \Delta^k.a^{-2k} c^{-1}$ for some $c \in B$, with $k > 0$. Dividing by $\Delta^k \in Z(B)$
and taking the inverse we get $y = (\Delta^{-k} x)^{-1} =
a^{2k}$ and $y = c a^{2k} c^{-1}$, so as before
we get that $c$ commutes with $a$ and we get $\PC(x) = \langle \Delta a^{-2} \rangle$. Since we also
have $\PC(x^m) = \langle \Delta a^{-2} \rangle$ for every $m \geq 1$ this yields the conclusion in that
case, too.

We then prove \autoref{maintheo:curvecomplex} in the case of the groups of rank $2$. Let $B_0$ be an irreducible parabolic subgroup.
We have $\PC(z_B) = B$, for otherwise $z_B$ would be the power of some distinguished braided reflection $\sigma$,
so that $z_B$ and $\sigma$ would have some nontrivial common power inside the pure braid group $P$,
and considering the image inside
$P^{ab} = H_1(X,\Z) \simeq \Z \mathcal{A}$ immediately yields a contradiction, as the hyperplane arrangement $\mathcal{A}$ contains at least
$2$ hyperplanes.
If $B_0$ has rank $1$, clearly $B_0 = \langle z_{B_0} \rangle$, so that $\PC(z_{B_0}) = B_0$. This proves the
first part of the statement. Let $B_1,B_2$ be two irreducible parabolic subgroups.
Without loss of generality we can assume that they have rank $1$, for otherwise the statements are trivially true.
If $B_1^g = B_2$ for some $g \in B$, since $B_i = \langle z_{B_i} \rangle$ and $z_{B_i}$ is the unique
positive generator of $B_i$ (in the sense of \autoref{sect:defcurvegraph})
 we have necessarily $z_{B_1}^g = z_{B_2}$.

Then, $z_{B_1} z_{B_2} = z_{B_2} z_{B_1}$ trivially implies $[B_1,B_2]= 1$, as $B_i = \langle z_{B_i} \rangle$.
Moreover, if $B_1 \cap B_2$ contains some nontrivial element, there exists $a,b \neq 0$
such that $z_{B_1}^a = z_{B_2}^b$ whence
$$
B_1 = \PC(z_{B_1}) = \PC(z_{B_1}^a) = \PC(z_{B_2}^b) = \PC(z_{B_2}) = B_2
$$
and $B_1 = B_2$, which completes the proof of \autoref{maintheo:curvecomplex}.

\section{Hyperbolicity results}

\label{sect:proofhyperbolic}

In this section we prove Theorem \ref{maintheo:hyperbolic}.
According to Theorem \ref{maintheo:curvecomplex} we can define the curve graph attached to
our irreducible complex reflection group $W$ by using as vertices
the proper irreducible parabolic subgroups and joining $B_1$ and $B_2$ when
$z_{B_1}$ and $z_{B_2}$ commute.

We have the following useful tool.

\begin{proposition}\label{prop:commpuissanceszB} $z_{B_1}$ and $z_{B_2}$ commute if and only if there exist $n_1,n_2 \neq 0$ such that $z_{B_1}^{n_1}$ and $z_{B_2}^{n_2}$ commute.
\end{proposition}
\begin{proof} One implication is obvious. Set for short $z_i = z_{B_i}$. If $z_1^{n_1}$ and $z_2^{n_2}$ commute, we can first assume $n_1,n_2 > 0$.
Then, by Theorem \ref{maintheo:parabolicclosure} we get $B_1 = PC(z_1)= PC(z_1^{n_1})
= PC((z_1^{n_1})^{z_2^{n_2}})
= PC((z_1^{n_1}))^{z_2^{n_2}}
= PC(z_1)^{z_2^{n_2}}
= B_1^{z_2^{n_2}}$
and for instance by Theorem \ref{maintheo:hyperbolic} we get that
$z_1$ commutes with $z_2^{n_2}$.
Exchanging the roles of $z_1$ and $z_2$ one gets with the same argument that $z_1$ and $z_2$ commute.
\end{proof}

Consider the case $W = G(de,e,n)$ with $r = de$ and $d,n > 1$
and denote $B$ its braid group. The curve graph
$\mathcal{G}$ associated to $\hat{W} = G(r,1,n)$ is known to be hyperbolic,
by work of \cite{CALVEZCISNEROS}. Let us denote $\mathcal{G}(e)$ the curve
graph of $B$, which is a subgroup of the braid group $\hat{B}$ of $\hat{W}$.
We know from Proposition \ref{prop:parabsGdeen} -- and the easy fact that intersecting a parabolic subgroup of $\hat{W}$ with $W$ preserves irreducibility -- that the map $\hat{B}_0 \mapsto B_0 = \hat{B}_0   \cap B$ provides
a surjective map $h : \mathcal{G} \to \mathcal{G}(e)$. We saw in \S \ref{sect:intersectgeneralCBG}
that $h$ is injective, hence a bijection. Combining Proposition \ref{prop:parabsGdeen} with Proposition \ref{prop:commpuissanceszB} above we get that the adjacency relations are the same --
since,  with the previous notations, each $z_{B_0}$ is a power of $z_{\hat{B}_0}$ --
so that $h$ is a graph isomorphism, and in particular an isometry
for the natural metric. This proves that the curve complex is hyperbolic for such $W$ as well.

Now consider the case where $W$ is irreducible of rank $2$ and denote $\mathcal{G}$ its curve graph. The vertices correspond to parabolic subgroups of rank $1$, so the vertex set of $\mathcal{G}$ can be identified with
the set $V$ of all braided reflections of $B$. Then $\sigma_1,\sigma_2 \in V$
are connected in $\mathcal{G}$ if and only if they commute,
which is the same as asking for $\sigma_1^{|W|}$
and $\sigma_2^{|W|}$ to commute. Therefore one can take for vertices
the set $V' = \{ \sigma^{|W|}; \sigma \in V \}$, and $\mathcal{G}$ is the commutation graph on this set. But $V' \subset P$, and $P \simeq \Z \times F_{N-1}$
where $N$ is the number of reflecting hyperplanes for $W$ and $F_r$ denotes the free group of rank $r$.
Since $W$ is irreducible, we have $N \geq 3$. 
Also notice that,
since two elements $\sigma_1^{|W|},\sigma_2^{|W|}$ of $V'$ have for parabolic closures two parabolic subgroups of the same rank,
if they are connected and are distinct the corresponding subgroups
need to intersect trivially. Moreover, the elements of $V'$ have nontrivial
projection onto $F_{N-1}$, as they cannot be central. But inside a free group,
every two commuting elements are powers of some common element. This means that, under the
isomorphism $P \simeq \Z \times F_{N-1}$, there is $x_0 \in F_{N-1}$ such that each $\sigma_i^{|W|}$ is mapped to $(n_i,x_0^{r_i})$ for
some $n_i,r_i \in \Z$ with $r_1,r_2 \neq 0$. But then $\sigma_1^{r_2 |W|}\sigma_2^{-r_1 |W|}$ belongs to $Z(P) = \langle z_P \rangle$,
 so that $\sigma_1^{r_2 |W|}\sigma_2^{-r_1 |W|} = z_P^m$ for some $m \in \Z$. Now recall that $P^{ab} = P/[P,P]$ is a free abelian group on generators $a_H,H \in \mathcal{A}$ for $\mathcal{A}$ the reflecting arrangement of $W$.
If $H_1,H_2$ are the hyperplanes corresponding to $\sigma_1,\sigma_2$, the image of $\sigma_1^{r_2 |W|}\sigma_2^{-r_1 |W|}$ under the abelianization map is a linear combination of $a_{H_1}$ and $a_{H_2}$. On the other hand,
the image of $z_P$ is $\sum_{H \in \mathcal{A}} a_H$. Comparing both sides, since $N \geq 3$ we get $m = 0$
so that $\sigma_1^{r_2 |W|}=\sigma_2^{r_1 |W|}$, and this contradicts the fact that
the two parabolic subgroups intersect trivially. Therefore, this proves that the graph
is a disjoint union of points, and  generalizes
what happens for the more usual curve complexes in small rank.


\begin{thebibliography}{DWKL}
\bibitem{BANNAI} E. Bannai, {\it Fundamental groups of the spaces of regular orbits of the finite unitary groups of rank $2$}, J. Math. Soc. Japan {\bf 28} (1976), 447--454.
\bibitem{BESSISZAR} D. Bessis, {\it Zariski theorems and diagrams for braid groups}, Invent. math. {\bf 145}, 487–507 (2001).
\bibitem{BESSIS} D. Bessis, {\it Finite complex reflection arrangement are $K(\pi,1)$}, Ann. of Math. (2) {\bf 181} (2015), 809-904.
\bibitem{BOWDITCH} B. H. Bowditch, {\it A Course on Geometric Group Theory}, 
MSJ Memoirs {\bf 16}, Math. Soc. Japan, 2006.
\bibitem{BIRMANKOLEE} J. S. Birman, K. H. Ko, S. J. Lee, {\it A new approach to the word and conjugacy problems in the braid groups}, Adv. Math. {\bf 139} (1998), no. 2, 322--353.
\bibitem{BRADYWATT}  T. Brady, C. Watt, {\it A partial order on the orthogonal group}, Comm. Algebra {\bf 30} (2002), 3749-3754.
\bibitem{LIE456} N. Bourbaki, {\it Groupes et alg\`ebres de Lie, chapitres 4,5,6}, Masson, Paris, 1981.
\bibitem{BRIESKORN} E. Brieskorn, {\it Die Fundamentalgruppe des Raumes der regul\"aren Orbits einer endlichen komplexen Spiegelungsgruppe}, Invent. Math. {\bf 12} (1971), 57-61.
\bibitem{BMR} M. Brou\'e, G. Malle, R. Rouquier,
{\it Complex reflection groups, braid groups, Hecke
algebras}, J. Reine Angew. Math. {\bf 500} (1998), 127--190.
\bibitem{CALMAR} F. Callegaro, I. Marin, {\it Homology computations for complex braid groups}, J. Eur. Math. Soc. {\bf 16}
(2014) 103--164.
\bibitem{CALVEZCISNEROS} M. Calvez, B.A. Cisneros De La Cruz, {\it 
Curve graphs for Artin-Tits groups of type $B$,
$\tilde{A}$
and $\tilde{C}$
are hyperbolic}, Trans. Lond. Math. Soc. {\bf 8} (2021) 151--173. 
\bibitem{CHEN1} K.T. Chen, {\it Formal differential equations}, Annals of Math. {\bf 73} (1961), 110-133 .
\bibitem{CHEN2} K.T. Chen, {\it Algebras of iterated path integrals and fundamental groups}, Trans. Am. Math. Soc. {\bf 156} (1971), 359-379.
\bibitem{CORPIC} R. Corran, M. Picantin, {\it A new Garside structure for braid groups of type $(e,e,r)$}, J. London Math. Soc. {\bf 84} (2011) 689-711.
\bibitem{CUMP} M. Cumplido Cabello, {\it Sous-groupes paraboliques et généricité
dans les groupes d’Artin-Tits de type sphérique}, doctoral thesis, universities of Rennes and Sevilla, 2018.
\bibitem{CGGW} M. Cumplido, V. Gebhardt, J. Gonz\'alez-Meneses, B. Wiest, {\it On parabolic subgroups of Artin-Tits groups of spherical type}, Adv. Math. {\bf 352} (2019), 572–-610.
\bibitem{DAVIS} M.W. Davis, J. Huang, {\it Bordifications of hyperplane arrangements and their curve complexes}, J. Topol. {\bf 14} (2021),  419--459.
\bibitem{DGKM} P. Dehornoy, F.Digne, E. Godelle, D. Krammer, J. Michel,
{\it Foundations of Garside theory}, EMS Tracts in Mathematics, 22. European Mathematical Society (EMS), Zürich, 2015.
\bibitem{Dehornoy-Paris} P. Dehornoy and L. Paris. Garside groups, a generalization of Artin groups, Proc. London Math. Soc. 79 (1999) 569--604.
\bibitem{DENEFLOESER} J. Denef, F. Loeser, {\it Regular elements and monodromy of discriminants of finite reflection groups}, Indag. Mathem., N.S., {\bf 6} (1995), 129--143.
\bibitem{DMM} F. Digne, I. Marin, J. Michel, {\it The center of pure complex braid groups}, J. Algebra {\bf 347} (2011) 206-213.
\bibitem{THEO} T. Douvropoulos, {\it Applications of geometric techniques in Coxeter-Catalan
combinatorics}, Ph. D. thesis, University of Minnesota, 2017.
\bibitem{ELRIFAIMORTON}
  E. A. Elrifrai, H.R. Morton, {\it Algorithms for positive braids},
Quart. J. Math. Oxford Ser. (2) {\bf 45} (1994), 479--497.
\bibitem{EPSTEINETAL} D. B. A. Epstein, J. W. Cannon, D. F. Holt, S. V. F. Levy, M. S. Paterson, W. P. Thurston, {\it Word processing in groups}. Jones and Bartlett Publishers, Boston, MA, 1992.
\bibitem{FRANCOGM} N. Franco, J. Gonz\'alez-Meneses, {\it Conjugacy problem for braid groups and Garside groups}, J. Algebra {\bf 266} (2003), 112–-132.
\bibitem{FRANCOGM2}  N. Franco, J. Gonz\'alez-Meneses,
{\it Computation of centralizers in braid groups and Garside groups}, Proceedings of the International Conference on Algebraic Geometry and Singularities (Sevilla, 2001),
Rev. Mat. Iberoamericana {\bf 19} (2003), 367-384.
\bibitem{PARAB2} O. Garnier, {\it Parabolic subgroups of complex braid groups: the remaining case}, preprint
2024, arxiv:2403.02209v1.
\bibitem{GEBHARDT} V. Gebhardt, {\it A new approach to the conjugacy problem in Garside groups}, J. Algebra {\bf 292} (2005), no. 1, 282–-302.
\bibitem{GEBHARDTGM} V. Gebhardt, J. Gonz\'alez-Meneses, {\it The cyclic sliding operation in Garside groups}, Math. Z. {\bf 265} (1), 2010, 85--114.
\bibitem{GEBHARDTGM2} V. Gebhardt, J. Gonz\'alez-Meneses, {\it Solving the conjugacy problem in Garside groups by cyclic sliding}, J. Symb. Comp. {\bf 45} (6), 2010, 629--656.
\bibitem{GHYSHARPE} E. Ghys, P. de la Harpe, {\it Sur les groupes hyperboliques, d'après M. Gromov}, Progress Math. {\bf 83}, Birkh\"auser 1990. 
\bibitem{GODELLE2003} E. Godelle, {\it Normalisateur et groupe d'Artin de type sph\'erique}, J. Algebra {\bf 269} (2003) 263--274.
\bibitem{GODELLE2007} E. Godelle, {\it Parabolic subgroups of Garside groups}, J. Algebra {\bf 317} (2007) 1--16.
\bibitem{KOHNONAGOYA} T. Kohno, {\it On the holonomy Lie
algebra and the nilpotent completion of the fundamental group
of the complement of hypersurfaces}, Nagoya Math. J. {\bf 92}, 21-37 (1983).
\bibitem{KOHNOINVENT} T. Kohno,  {\it Série de Poincaré-Koszul associée au groupe de tresses pures},
 Invent. Math. {\bf 82} (1985), 57--75.
 \bibitem{LEHRERPOIN} G.I. Lehrer, {\it Poincaré polynomials for unitary reflection groups}, Invent. Math. {\bf 120} (1995), 411--425.
\bibitem{LEHRERTAYLOR} G.I. Lehrer, D.E. Taylor, {\it Unitary reflection groups}, Cambridge University Press, 2009.
\bibitem{LEWISMORALES}  J.B. Lewis, A.H.  Morales, {\it Factorization problems in complex reflection groups}, Canad. J. Math. {\bf 73} (2021), 899-946.
\bibitem{IH} I. Marin, {\it Infinitesimal Hecke Algebras}, Comptes Rendus Math\'ematiques {\bf 337} S\'erie I, 297-302 (2003).
\bibitem{LIETRANSP} I. Marin, {\it L'alg\`ebre de Lie des transpositions}, J. Algebra {\bf 310} (2007), 742-774.
\bibitem{IH2} I. Marin, {\it Infinitesimal Hecke Algebras II}, preprint 2009, arXiv:0911.1879.
\bibitem{KRAMCRG} I. Marin, {\it Krammer representations for complex braid groups}, J. Algebra {\bf 371} (2012), 175-206.
\bibitem{MARINPFEIFFER} I. Marin and G. Pfeiffer, {\it The BMR freeness conjecture for the 2-reflection groups}, Math. Comp. {\bf 86} (2017), 2005-2023.
\bibitem{MICHEL} J. Michel, {\it  A note on words in braid monoids}, J. Algebra {\bf 215} (1999), 366-377.
\bibitem{NEAIME} G. Neaime,
{\it Interval Garside structures for the complex braid groups $B(e,e,n)$}, Trans. Amer. Math. Soc. {\bf 372} (2019), 8815-8848.
\bibitem{ORLIKSOLOMON} P. Orlik, L. Solomon, {\it Discriminants in the invariant theory of reflection groups}, Nagoya Math. J. {\bf 109} (1988), 23--45.
\bibitem{ORLIKTERAO} P. Orlik, H. Terao, {\it Arrangements of hyperplanes},  Springer-Verlag, Berlin, 1992.
\bibitem{PICANTIN} M. Picantin, {\it Petits Groupes Gaussiens}, Thèse de l’université de Caen, 2000.
\bibitem{RIPOLL} V. Ripoll, {\it Orbites d'Hurwitz des factorisations primitives d'un élément de Coxeter}, J. Algebra {\bf 323} (2010), 1432--1453.
\bibitem{TAYLORREFL} D.E. Taylor, {\it Reﬂection subgroups of ﬁnite complex reﬂection groups}, J. Algebra {\bf 366} (2012) 218--234.
\end{thebibliography}
\end{document}